\pdfoutput=1
\documentclass[11pt]{amsart}

\usepackage{graphicx}
\usepackage[english]{babel}
\usepackage[T1]{fontenc}
\usepackage[latin1]{inputenc}
\usepackage{amsfonts}
\usepackage{amssymb}
\usepackage{amsthm}
\usepackage{amsmath}
\usepackage{tikz}
\usepackage{float}
\usepackage{enumerate}
\usepackage{accents}
\usepackage{mathtools}
\usepackage{mathrsfs}
\usepackage{comment}
\usepackage{afterpage}
\usepackage{bbm}
\usepackage{dsfont}
\usepackage{stmaryrd}
\usepackage{hyperref}
\usepackage{mleftright}
\usepackage{subfig}
\usepackage{soul}
\usepackage{tikz-cd} 
\usepackage{fullpage}
\usepackage{cleveref}

\numberwithin{equation}{section}

\theoremstyle{plain}
\newtheorem{thm}{Theorem}[section]
\newtheorem*{thm*}{Theorem}
\newtheorem{prop}[thm]{Proposition}
\newtheorem*{prop*}{Proposition}
\newtheorem{cor}[thm]{Corollary}
\newtheorem*{cor*}{Corollary}
\newtheorem{lem}[thm]{Lemma}

\newtheorem{thmintro}{Theorem}

\newtheorem{corintro}[thmintro]{Corollary}

\theoremstyle{definition}
\newtheorem{defn}[thm]{Definition}
\newtheorem*{defn*}{Definition}
\newtheorem{ex}[thm]{Example}
\newtheorem{rmk}[thm]{Remark}
\newtheorem*{rmk*}{Remarks}

\newtheorem*{conj*}{Conjecture}
\newtheorem*{twist}{Twist Conjecture}
\newtheorem*{prob*}{Problem}
\newtheorem*{quest*}{Question}

\newtheorem{setup}[thm]{Setup}

\newtheoremstyle{blue-environment}{}{}{}{}{\color{blue}\bfseries}{.}{ }{}
\theoremstyle{blue-environment}

\newcommand{\acts}{\curvearrowright}
\newcommand{\ra}{\rightarrow}
\newcommand{\Ra}{\Rightarrow}

\newcommand{\sq}{\subseteq}

\newcommand{\wh}{\widehat}
\newcommand{\x}{\times}

\newcommand{\mc}{\mathcal}
\newcommand{\mf}{\mathfrak}
\newcommand{\mscr}{\mathscr}

\newcommand{\Z}{\mathbb{Z}}

\newcommand{\A}{\mathbb{A}}

\newcommand{\s}{\sigma}
\newcommand{\eps}{\epsilon}
\newcommand{\Om}{\Omega}
\newcommand{\om}{\omega}
\newcommand{\g}{\gamma}

\newcommand{\Aut}{{\rm Aut}}

\DeclareMathOperator{\hull}{Hull}

\DeclareMathOperator{\lk}{lk}

\DeclareMathOperator{\supp}{supp}

\begin{document}

\title{The Twist Conjecture \\ and the Isomorphism Problem for Coxeter groups} 

\author[E.\,Fioravanti]{Elia Fioravanti}\address{Institute of Algebra and Geometry, Karlsruhe Institute of Technology}\email{elia.fioravanti@kit.edu} 
\thanks{The author is supported by Emmy Noether grant 515507199 and grant 541703614 of the DFG}

\begin{abstract}
    We prove M\"uhlherr's Twist Conjecture: any two angle-compatible Coxeter generating sets of a Coxeter group differ by a finite sequence of elementary twists and a conjugation. Combined with earlier work of Howlett--M\"uhlherr and Marquis--M\"uhlherr, this completes the resolution of the Isomorphism Problem for Coxeter groups. 
    
    A further consequence is that $\Aut(W)$ is finitely generated for every Coxeter group $W$, and there is an algorithm producing a finite set of generators for $\Aut(W)$ starting from any Coxeter matrix.

    Of the vast literature on the Twist Conjecture, we utilise only two results in an essential way: strong rigidity of $2$--spherical Coxeter systems, due to Caprace and M\"uhlherr, and the framework of markings and hierarchies developed by Caprace and Przytycki for the twist-rigid case. We also exploit in a fundamental way some soft ideas from JSJ theory and an observation of Mihalik--Tschantz on splittings of Coxeter groups. 

    No form of AI was used in the writing of this manuscript, nor in the research that it presents.
\end{abstract}

\maketitle

\section{Introduction}\label{sect:intro}

Coxeter groups were introduced by H.\,S.\,M.\ Coxeter in 1934 as an abstraction of reflection groups \cite{Coxeter}, and they have since taken on a prominent role throughout algebra, geometry and combinatorics. The term ``Coxeter group'' was coined by J.\ Tits in a 1961 preprint \cite{Tits}. 

Despite their importance, and almost a century of history, one of the most fundamental questions about Coxeter groups has remained open: how to tell them apart.

\begin{prob*}
    Given two Coxeter systems $(W,S)$ and $(W',S')$, are the groups $W$ and $W'$ isomorphic?
\end{prob*}

A Coxeter system $(W,S)$ is completely described by a Coxeter matrix, or equivalently by a Dynkin diagram or a presentation graph. However, several distinct phenomena can cause non-isomorphic Coxeter systems to give rise to isomorphic Coxeter groups, to the point that it has even remained unknown whether the above problem is algorithmically decidable. 

According to the survey \cite{Muehlherr-survey}, the first explicit mention of the isomorphism problem for Coxeter groups as a question worthy of attention seems to be in \cite[Problem~6.5]{Cohen}. It also appears as Question~2.13 in Bestvina's problem list \cite{BesQ}, and it is the subject of Bahls' book \cite{Bahls-book} and of \cite[Problems~1--2]{Muehlherr-survey}.

There is a very clear conjectural picture as to what a solution to the isomorphism problem should look like for Coxeter groups, which has been largely devised and championed by M\"uhlherr over the last 25 years. In particular, work of Howlett--M\"uhlherr \cite{Howlett-Muehlherr} and Marquis--M\"uhlherr \cite{Marquis-Muehlherr} has reduced the solution of the isomorphism problem to the proof of a clean statement commonly known as the Twist Conjecture. The goal of this article is to prove this conjecture.

Before stating it, we need to introduce the concept of an \emph{elementary twist}. Let $(W,S)$ be a Coxeter system and suppose that $S$ can be partitioned as $S=X\sqcup K\sqcup K^{\perp}\sqcup Y$, where $K$ is a nonempty irreducible spherical set, $K^{\perp}$ is the set of elements of $S$ commuting with $K$, and $X,Y$ are nonempty sets such that $xy$ has infinite order for all $x\in X$ and $y\in Y$. Denoting by $w_K$ the longest element of the spherical subgroup $\langle K\rangle$, 
we can consider the set $S':=X\sqcup K\sqcup K^{\perp}\sqcup w_KYw_K$. The pair $(W,S')$ is again a Coxeter system, and its Coxeter matrix may differ from that of $(W,S)$. This procedure --- called an \emph{elementary twist} --- was introduced in \cite{BMMN} generalising an example discovered in \cite{Muehlherr00} (reproduced here in \Cref{fig:1}).

\begin{figure}
    \begin{tikzpicture}
        \draw[fill] (0,1) -- (0,0);
        \draw[fill] (0,1) -- (1,1);
        \draw[fill] (0,1) -- (1,0);
        \draw[fill] (0,0) circle [radius=0.05cm];
        \draw[fill] (0,1) circle [radius=0.05cm];
        \draw[fill] (1,0) circle [radius=0.05cm];
        \draw[fill] (1,1) circle [radius=0.05cm];
        \node[left] at (0,1) {$a$};
        \node[right] at (1,1) {$b$};
        \node[left] at (0,0) {$c$};
        \node[right] at (1,0) {$d$};
    \end{tikzpicture}
    \hspace{1cm}
    \begin{tikzpicture}
        \draw[fill] (0,1) -- (0,0);
        \draw[fill] (0,1) -- (1,1);
        \draw[fill] (1,1) -- (1,0);
        \draw[fill] (0,0) circle [radius=0.05cm];
        \draw[fill] (0,1) circle [radius=0.05cm];
        \draw[fill] (1,0) circle [radius=0.05cm];
        \draw[fill] (1,1) circle [radius=0.05cm];
        \node[left] at (0,1) {$a$};
        \node[right] at (1,1) {$b$};
        \node[left] at (0,0) {$c$};
        \node[right] at (1,0) {$d'$};
    \end{tikzpicture}
    \caption{The presentation graphs of the two Coxeter systems yielding isomorphic Coxeter groups discovered in \cite{Muehlherr00}. Non-edges correspond to pairs of elements $s,t$ with $st$ of infinite order, while all edges correspond to pairs with $(st)^3=1$. The two systems differ by an elementary twist over the dihedral set $K=\{a,b\}$.}
    \label{fig:1} 
\end{figure}
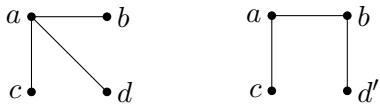

While elementary twists\footnote{Although the concept of an elementary twist is reminiscent of that of a \emph{Dehn twist} from \cite{RS94,Bass-Jiang,Levitt-GD}, it is a more general procedure: elementary twists do \emph{not} always extend to automorphisms of the Coxeter group.} are not the only procedure yielding non-isomorphic Coxeter systems for isomorphic Coxeter groups (see \cite{Muehlherr-survey} for an overview), they become the only phenomenon of concern if one restricts to \emph{angle-compatible} pairs of Coxeter generating sets (also known as \emph{sharp-angled} in the literature). Two Coxeter generating sets $S,R\sq W$ are said to be \emph{angle-compatible} if every spherical subset of $S$ of cardinality $\leq 2$ admits a $W$--conjugate contained in $R$. Although not obvious from the definition, angle-compatibility is an equivalence relation, see \cite[Appendix~A]{CP10}. Moreover, there are general algorithmic procedures that, starting with an arbitrary pair of Coxeter generating sets of $W$, construct an angle-compatible pair of generating sets related to the original ones; see \cite{Howlett-Muehlherr,Marquis-Muehlherr} and the overview in \cite{Muehlherr-survey}. 

With these definitions out of the way, we can state the:

\begin{twist}
    Let $W$ be a Coxeter group with two angle-compatible Coxeter generating sets $S,R\sq W$. Then $S$ and $R$ differ by a finite sequence of elementary twists, followed by a conjugation.
\end{twist}

This appears for instance as Conjecture~2 in \cite[Section~5]{Muehlherr-survey} and as Conjecture~5.13 in \cite{SRS}. An older formulation in \cite[Conjecture~8.1]{BMMN} was shown to be false in \cite[Section~10]{Ratcliffe-Tschantz}, as it did not take angle-deformations into account; this issue was later addressed in \cite{Marquis-Muehlherr}, leading to the formulation above.

There is a large body of work on the Twist Conjecture, so we only mention a brief selection of earlier results. Several classes of Coxeter groups are \emph{strongly rigid}, meaning that all Coxeter generating sets are conjugate to each other, without requiring any twists or angle-compatibility: this is the case for Coxeter groups admitting geometric actions on contractible manifolds \cite{CD00}, for $2$--spherical Coxeter groups \cite{FHM,CM07}, and more generally whenever all wall neighbourhoods coarsely disconnect the Coxeter group into exactly two components \cite{CP-bipolar}. The Twist Conjecture is also known to hold when the Coxeter system $(W,S)$ admits no elementary twists \cite{CP10}. Some situations where the conjecture has been proven and twists are necessary include: when no two elements of $S$ commute \cite{Muehlherr-Weidmann}, when the presentation graph is chordal \cite{Ratcliffe-Tschantz}, when certain triangle subgroups are forbidden \cite{Weigel}, and finally for Coxeter groups of FC-type admitting only twists over sets $K$ with cardinality $\leq 2$ \cite{Huang-Przytycki,Przytycki}. 

Despite the number of results, progress on the conjecture has been remarkably limited: for instance, with the exception of \cite{Ratcliffe-Tschantz}, it seems that it has not been proven in any situation where twists over irreducible spherical subsets $K$ of cardinality $\geq 3$ are allowed. 

We prove the Twist Conjecture in full generality and in the following stronger form: any subset of $S$ that already admits a conjugate contained in $R$ can be left untouched by all twists. All the terminology used in the theorem statement below is defined precisely in the next section.

\begin{thmintro}\label{thmintro:main}
    Let $W$ be a Coxeter group and let $S,R\sq W$ be two angle-compatible Coxeter generating sets. Let $\mscr{C}$ be the family of $\{S,R\}$--compatible subsets of $W$. Then $S$ and $R$ are twist-equivalent relative to $\mscr{C}$.
\end{thmintro}

We describe our proof strategy in \Cref{sect:strategy} below. There, we give a complete proof of \Cref{thmintro:main} assuming two results, Theorems~\ref{thm:step_one} and~\ref{thm:step_two}, which are shown in Sections~\ref{sect:shrubs} and~\ref{sect:geometrisation}, respectively.

Combining \Cref{thmintro:main} with the work of Howlett--M\"uhlherr \cite{Howlett-Muehlherr} and Marquis--M\"uhlherr \cite{Marquis-Muehlherr} as summarised in \cite{Muehlherr-survey}, we obtain the following consequence.

\begin{corintro}\label{corintro:iso_problem}
    The Isomorphism Problem is algorithmically decidable for Coxeter groups.
\end{corintro}

We briefly discuss in \Cref{app:automorphisms} the results needed to deduce \Cref{corintro:iso_problem} from \Cref{thmintro:main}. There, we also obtain the following straightforward consequence:

\begin{corintro}\label{corintro:automorphisms}
    For every Coxeter group $W$, the automorphism group $\Aut(W)$ is finitely generated.
\end{corintro}

In fact, we exhibit an algorithm producing a finite generating set of $\Aut(W)$ in finite time, taking as input the Coxeter matrix of any Coxeter generating set of $W$.

We emphasise that \Cref{corintro:automorphisms} is \emph{not} a direct consequence of our earlier finite generation result for automorphism groups of compact special groups \cite{Fio12}, and there are two important reasons for this. First, although Coxeter groups are known to be virtual fundamental groups of special cube complexes by \cite{HW10}, it remains wide open whether such cube complexes can be chosen to be \emph{compact}. Secondly, all results of \cite{Fio12} are about genuine compact special groups, and there are significant obstructions to extending them to groups that are only \emph{virtually} compact special. Overall, there is no overlap between the techniques of this article and those we developed for \cite{Fio11,Fio12}.

\smallskip
{\bf Acknowledgments.} I am grateful to Pierre-Emmanuel Caprace and Piotr Przytycki for encouraging me to work on the conjecture, and for their comments on earlier versions of this preprint.

\tableofcontents

\section{Terminology and proof overview}\label{sect:strategy}

Before sketching the proof of \Cref{thmintro:main}, we need to fix some terminology. We refer to \cite{Davis} for general background on Coxeter groups from a geometric perspective. In particular, the Davis complex $\A_S$ associated to a Coxeter system $(W,S)$ is introduced in \cite[Chapter~7]{Davis}.

Given a Coxeter system $(W,S)$, we denote by $S^W$ the set of \emph{$S$--reflections}, that is, the elements of $W$ that are conjugate to elements of $S$. The set of fixed points in $\A_S$ of a reflection $\rho\in S^W$ is called the \emph{wall} fixed by $\rho$, and we will usually denote it by $\mc{Y}_{\rho}$ or $\mc{W}_{\rho}$. Every wall disconnects $\A_S$ into exactly two connected components called \emph{halfspaces}.

We will consider subsets of $W$ satisfying some of the following classical notions. The terminology ``compatible'' is a little less standard, though it appears for instance in \cite[Definition~14.1]{HMN18}.

\begin{defn}
Let $(W,S)$ be a Coxeter system.
\begin{itemize}
    \setlength\itemsep{.2em}
    \item A subset $U\sq W$ is \emph{universal} if the pair $(\langle U\rangle,U)$ is a Coxeter system. 
    \item A subset $U\sq W$ is \emph{$S$--compatible} if there exists $g\in W$ with $gUg^{-1}\sq S$.
    \item A subgroup $P\leq W$ is \emph{$S$--parabolic} if there exist $U\sq S$ and $g\in W$ with $P=g\langle U\rangle g^{-1}$. By an abuse, we also say that a universal subset $U\sq W$ is \emph{$S$--parabolic} if $\langle U\rangle$ is $S$--parabolic.
    \item A set of reflections $U\sq S^W$ is \emph{$2$--geometric} in the Davis complex $\A_S$ if we can choose, for each element $u\in U$, a side $\mc{H}_u$ of the $u$--fixed wall in $\A_S$ so that, for all $u,u'\in U$, the sector $\mc{H}_u\cap\mc{H}_{u'}$ is a fundamental domain for the action $\langle u,u'\rangle\acts\A_S$. If the sides $\mc{H}_u$ can be chosen so that in addition we have $\bigcap_{u\in U}\mc{H}_u\neq\emptyset$, then we say that $U$ is \emph{$S$--geometric}.
\end{itemize}
If $\mscr{S}$ is a collection of Coxeter generating sets of $W$, we say that $U\sq W$ is \emph{$\mscr{S}$--compatible} if $U$ is $S$--compatible for all $S\in\mscr{S}$. We use the term \emph{$\mscr{S}$--parabolic} in the same way.
\end{defn}

It is convenient to record here two important facts of which we will make repeated use. The first is \cite[Fact~1.6]{CM07}, which was originally established in a different language in \cite{Hee}, and independently also in \cite[Theorem~1.2]{HRT}: 

\begin{lem}\label{lem:1.6}
    If a subset $U\sq S^W$ is universal and $2$--geometric in $\A_S$, then it is $S$--geometric.
\end{lem}

The second is an immediate consequence of \cite[Proposition~3.5]{CM07}:

\begin{lem}\label{lem:geometric+parabolic}
    If a subset $U\sq S^W$ is both $S$--geometric and $S$--parabolic, then $U$ is $S$--compatible.
\end{lem}
 
A nonempty subset $U\sq S$ is said to be \emph{spherical} if the subgroup $\langle U\rangle$ is finite, and \emph{$2$--spherical} if every product $uu'$ with $u,u'\in U$ has finite order. The set $U$ is of \emph{FC-type} if all its $2$--spherical subsets are in fact spherical. The set $U\sq S$ is \emph{irreducible} if it cannot be partitioned into two nonempty, mutually commuting subsets. As customary, we denote by $U^{\perp}$ the set of elements of $S\setminus U$ that commute with $U$, and refer to it as the \emph{orthogonal} of $U$.

Now, let $\mscr{C}$ be a family\footnote{Throughout the article, the only case of interest is when $\mscr{C}$ consists of all $W$--conjugates of certain subsets of $S$, as is the case in the statement of \Cref{thmintro:main}. However, it is convenient to not always have to assume this.} of subsets of $W$. Recall that elementary twists of the Coxeter system $(W,S)$ originate from partitions $S=X\sqcup K\sqcup K^{\perp}\sqcup Y$, where $K$ is a nonempty irreducible spherical set with longest element $w_K\in\langle K\rangle$, and where $X,Y$ are nonempty sets with $xy$ of infinite order for all $x\in X$ and $y\in Y$. We say that such a partition is \emph{relative to $\mscr{C}$} if every set $C\in\mscr{C}$ admits elements $g\in W$ such that $gCg^{-1}$ is contained in the subgroup generated by either $X\cup K\cup K^{\perp}$ or $K\cup K^{\perp}\cup Y$. In this case, we say that the Coxeter generating set $S':=X\cup K\cup K^{\perp}\cup w_KYw_K$ is obtained from $S$ by an elementary twist \emph{relative to $\mscr{C}$}.

Swapping the roles of $X$ and $Y$, that is, defining $S'':=w_KXw_K\cup K\cup K^{\perp}\cup Y$, is also an honest way of performing the twist, but this is of little consequence since the sets $S'$ and $S''$ are conjugate. Indeed $w_K$ is an involution with $w_KKw_K=K$ (see e.g.\ \cite[Lemma~4.6.1(5)]{Davis}), and so we have $S''=w_KS'w_K^{-1}$. Also note that, conversely, $S$ is obtained from $S'$ by a twist relative to $\mscr{C}$. 

\begin{defn}
    Two Coxeter generating sets $S,R\sq W$ are \emph{twist-equivalent} relative to $\mscr{C}$ if they differ by a finite sequence of elementary twists relative to $\mscr{C}$, possibly followed by a conjugation.
\end{defn}  

Conjugations can also be interspersed between any two consecutive twists without affecting the definition. In particular, twist-equivalence (relative to any family $\mscr{C}$) is an equivalence relation. 

\begin{rmk}\label{rmk:relative_preserves}
    If $S,R\sq W$ are twist-equivalent relative to a family $\mscr{C}$, then a set $U\in\mscr{C}$ is $S$--compatible if and only if it is $R$--compatible, and it is $S$--parabolic if and only if it is $R$--parabolic. In particular, twist-equivalent generating sets are always angle-compatible.
\end{rmk}

We can now give an overview of the proof of \Cref{thmintro:main}.

\smallskip
{\bf Proof sketch.} As in the statement, let $W$ be a Coxeter group with two angle-compatible Coxeter generating sets $S,R\sq W$, and let $\mscr{C}$ be the family of $\{S,R\}$--compatible subsets of $W$. Denote by $\mscr{T}_S$ the collection of all Coxeter generating sets of $W$ that are twist-equivalent to $S$ relative to $\mscr{C}$. Our goal is to show that $R\in\mscr{T}_S$.

Suppose that $S$ and $R$ are not conjugate. The basic idea is to find a low-complexity subset $\Om\sq S$ with $\Om\not\in\mscr{C}$, and then modify $S$ by elementary twists relative to $\mscr{C}$, so as to produce\footnote{It might seem more natural to keep $S$ fixed and modify $R$ by twists, but there are good reasons for our choice. One of them is to maintain notational coherence with \cite{CP10} later in the article.} 
some $S'\in\mscr{T}_S$ with a subset $\Om'\sq S'$ such that $\Om'$ is $R$--compatible and $\langle\Om'\rangle=\langle\Om\rangle$. If this is possible, then there are strictly more $\{S',R\}$--compatible subsets of $W$ than there are $\{S,R\}$--compatible ones (up to $W$--conjugacy). Iterating the procedure finitely many times, we eventually end up in the situation where $S$ and $R$ are conjugate. Note that there can be a resonable hope of carrying out such a strategy only if the set $\Om$ is $\mscr{T}_S$--parabolic, so we will need a source of such sets.

In greater detail, the proof of \Cref{thmintro:main} consists of two steps, corresponding to Theorems~\ref{thm:step_one} and~\ref{thm:step_two} below. Here and in the rest of the article, we encourage the reader to keep in mind the case when $(W,S)$ is of FC-type: indeed, \Cref{thmintro:main} is already new in this case and its proof still contains all the key ideas, while avoiding some of the technical complications in Sections~\ref{sub:shrub_defn} and~\ref{sect:markings}. 

\smallskip
{\bf The first step.} This roughly amounts to analysing the structure of the minimal subsets $\Om\sq S$ such that $\Om$ is $\mscr{T}_S\cup\{R\}$--parabolic and $\Om\not\in\mscr{C}$. Note that, if $S$ and $R$ are not conjugate, there indeed exist subsets of $S$ combining the latter two properties: for instance, $S$ itself is $\mscr{T}_S\cup\{R\}$--parabolic and not in $\mscr{C}$. Therefore, there also exist \emph{minimal} sets with these properties.

It turns out that these minimal sets $\Om$ have a particularly simple structure, which we term a \emph{$\mscr{C}$--shrub} (\Cref{defn:shrub}). Roughly, this means that $\Om$ decomposes as a finite tree (or rather a finite forest) of sets in $\mscr{C}$ and that, conversely, every subset of $\Om$ lying in $\mscr{C}$ is contained in a ``vertex set'' of this tree-like decomposition. The simplest example of a $\mscr{C}$--shrub is a subset of $S$ of the form $A\cup B$ with $A,B\in\mscr{C}$ and $\{a,b\}\not\in\mscr{C}$ for all $a\in A\setminus B$ and $b\in B\setminus A$. The fact that minimal $\mscr{T}_S\cup\{R\}$--parabolic subsets $\Om\sq S$ with $\Om\not\in\mscr{C}$ are $\mscr{C}$--shrubs is shown in \Cref{prop:minimal->shrub}, where we make fundamental use of some ideas of Mihalik and Tschantz from \cite{MT09}. The proof of this result also implicitly uses various ideas from JSJ theory (see \cite{GL-JSJ} for a survey).

When our Coxeter groups are \emph{not} of FC-type, it is important that our sets $\Om$ have an additional property that we call being \emph{tucked}\footnote{Because it means that $\Om$ does not have any ``small portions'' of irreducible, non-spherical, $2$--spherical subsets ``sticking out'' from itself.} (\Cref{defn:tucked}), but we prefer not to elaborate on this slightly technical point at this stage. In summary, the first step of the proof amounts to the following:

\begin{thm}\label{thm:step_one}
    If $S$ and $R$ are not conjugate, then there exists a subset $\Om\sq S$ such that:
    \begin{enumerate}
        \item $\Om$ is $\mscr{T}_S\cup\{R\}$--parabolic and $\Om\not\in\mscr{C}$;
        \item $\Om$ is a $\mscr{C}$--shrub (\Cref{defn:shrub});
        \item $\Om$ is $\mscr{T}_S$--tucked (\Cref{defn:tucked}).
    \end{enumerate}
\end{thm}

\smallskip
{\bf Intermission.} Before continuing with the proof overview of \Cref{thmintro:main}, we would like to highlight an important idea, which should further clarify our strategy: namely, we should consider the subsets of $S$ that lie in $\mscr{C}$ as generalisations of $2$--spherical subsets, and we should expect them to behave identically in all the important ways.

To begin with, it follows from \cite{CM07} that all $2$--spherical subsets of $S$ lie in $\mscr{C}$. Building on \cite{CM07}, we prove an analogue of \Cref{lem:1.6} in this context: a subset $U\sq S$ lies in $\mscr{C}$ if and only if all its subsets of cardinality $\leq 2$ do (\Cref{cor:2-compatible->compatible}). These observations naturally lead us to consider the graph $\wh S_{\mscr{C}}$ obtained as follows (we define this a little more generally for later use).

\begin{defn}\label{defn:hat_graph}
    Let $U\sq S$ be a subset. We denote by $\wh U_{\mscr{C}}$ the graph that has vertex set identified with $U$ and edges corresponding to pairs of distinct elements $u,u'\in U$ with $\{u,u'\}\in\mscr{C}$. 
\end{defn}

Since $S$ and $R$ are angle-compatible, the graph $\wh S_{\mscr{C}}$ is simply obtained from the presentation graph of $(W,S)$ by adding edges connecting the elements $s,s'\in S$ such that $\langle s,s'\rangle$ is infinite and $\{s,s'\}$ is $R$--compatible. By \Cref{cor:2-compatible->compatible}, a subset of $S$ lies in $\mscr{C}$ if and only if it spans a clique in $\wh S_{\mscr{C}}$.

Going back to the statement of \Cref{thm:step_one}, the fact that $\Om$ is a $\mscr{C}$--shrub implies that the subgraph $\wh\Om_{\mscr{C}}\sq\wh S_{\mscr{C}}$ is chordal. Forgetting about the (inconsequential) difference between $\wh\Om_{\mscr{C}}$ and the presentation graph of $\Om$, it follows from the work of Ratcliffe and Tschantz \cite{Ratcliffe-Tschantz} that there are elementary twists of the Coxeter group $\langle\Om\rangle$ taking its Coxeter generating set $\Om\sq S$ to a subset of $R$ (up to conjugacy). 

The real difficulty is that these elementary twists of the Coxeter group $\langle\Om\rangle$ have no reason whatsoever to extend to elementary twists of the \emph{ambient} Coxeter group $W$. Consequently, we will have to work much harder in the second step of the proof of \Cref{thmintro:main}.

\smallskip
{\bf The second step.} The final step of the proof is the following ``geometrisation'' result.

\begin{thm}\label{thm:step_two}
    Suppose that \Cref{thmintro:main} holds for all proper $S$--parabolic subgroups of $W$. Let $\Om\sq S$ be a $\mscr{T}_S$--parabolic, $\mscr{T}_S$--tucked $\mscr{C}$--shrub. Then there exist a generating set $S'\in\mscr{T}_S$ and a subset $\Om'\sq S'$ with $\langle\Om'\rangle=\langle\Om\rangle$ such that $\Om'$ is $R$--geometric.
\end{thm}

An important ingredient in the proof of \Cref{thm:step_two} is the framework of hierarchies and markings developed by Caprace and Przytycki in their work on twist-rigid Coxeter groups \cite{CP10} (by analogy with the Masur--Minsky machinery of the same name \cite{Masur-Minsky}). Roughly, a marking is a gadget consisting of a reflection $\rho\in S^W=R^W$ and a marker $m\in S$, and it encodes the choice of one of the two sides of the $\rho$--fixed wall in the Davis complex $\A_R$. Classically, one aims to show that this halfspace choice does not change if we modify the marking according to certain moves taking place within the presentation graph of $S$. It turns out that the graph $\wh S_{\mscr{C}}$ works just as well for this purpose and, strikingly, the original proofs from \cite{CP10} require almost no modifications in order to be carried out in this context (see \Cref{sect:markings}).

With this generalised framework of hierarchies and markings in hand, the core of the proof of \Cref{thm:step_two} lies in a ``shortening theorem'' (\Cref{thm:shortening}). Roughly, the idea is to identify a measure of complexity for the set $\Om\sq S$, which describes how far $\Om$ is from being an $R$--geometric set of reflections. We then proceed to identify a finite sequence of elementary twists that reduces the complexity of $\Om$, thus bringing $\Om$ a little closer to being $R$--geometric. 

As hinted at above, the difficulty lies in ensuring that all necessary twists extend to the ambient Coxeter group $W$. We also point out that a \emph{single} twist rarely suffices to reduce complexity, so one really needs a \emph{finite sequence} of twists at each shortening step.

\smallskip
{\bf Concluding the proof.} We obtain \Cref{thm:step_one} in \Cref{sect:shrubs}, and \Cref{thm:step_two} in \Cref{sect:geometrisation}. Assuming these two results, we can immediately give a proof of \Cref{thmintro:main}.

\begin{proof}[Proof of \Cref{thmintro:main}]
    We prove the theorem by induction on the cardinality of $S$. The base step is trivial. In the inductive step, we aim to prove the theorem for a Coxeter system $(W,S)$ under the assumption that it holds for all proper $S$--parabolic subgroups of $W$.
    
    Let $R\sq W$ be a Coxeter generating set that is angle-compatible with $S$. Define the families $\mscr{C}$ and $\mscr{T}_S$ as above. For every generating set $S'\in\mscr{T}_S$, the sets in the family $\mscr{C}$ are all $\{S',R\}$--compatible (see \Cref{rmk:relative_preserves}). Thus, up to replacing $S$ with some $S'\in\mscr{T}_S$ that maximises the family of $\{S',R\}$--compatible sets, we can assume that:
    \begin{itemize}
        \item[$(*)$] for every $S'\in\mscr{T}_S$, the family of $\{S',R\}$--compatible sets coincides with $\mscr{C}$.
    \end{itemize}
    Under Assumption~$(*)$, our goal becomes showing that $S$ and $R$ are conjugate.

    Suppose for the sake of contradiction that $S$ and $R$ are not conjugate. \Cref{thm:step_one} yields a $\mscr{T}_S$--tucked, $\mscr{T}_S\cup\{R\}$--parabolic $\mscr{C}$--shrub $\Om\sq S$ not lying in $\mscr{C}$. The inductive hypothesis then allows us to apply \Cref{thm:step_two}, so that we obtain a generating set $S'\in\mscr{T}_S$ and an $R$--geometric subset $\Om'\sq S'$ such that $\langle\Om\rangle=\langle\Om'\rangle$. The latter subgroup is $R$--parabolic because so was the set $\Om$, and therefore the fact that $\Om'$ is $R$--geometric implies that $\Om'$ is $R$--compatible (\Cref{lem:geometric+parabolic}). Assumption~$(*)$ then implies that $\Om'\in\mscr{C}$, so that $\Om'$ is actually conjugate to a subset of $S$. This contradicts the fact $\langle\Om'\rangle=\langle\Om\rangle$ and $\Om\not\in\mscr{C}$, concluding the proof of the theorem.
\end{proof}

Here is the plan for the rest of the article. In \Cref{sect:splittings}, we discuss actions of Coxeter groups on simplicial trees, and we develop some ideas from \cite{MT09} into a result producing many $\mscr{T}_S$--parabolic subsets of $S$ (\Cref{cor:universally_parabolic_vertex_stabilisers}). Then, in \Cref{sect:shrubs}, we show that the subsets of $S$ lying in $\mscr{C}$ are precisely those that span cliques in the graph $\wh S_{\mscr{C}}$ (\Cref{cor:2-compatible->compatible}), we define $\mscr{C}$--shrubs (\Cref{defn:shrub}), and we prove \Cref{thm:step_one} (the first step in the proof of \Cref{thmintro:main}).

The subsequent sections are entirely devoted to the proof of \Cref{thm:step_two} (the second step in the proof of \Cref{thmintro:main}). In \Cref{sect:inflexibility}, we identify a simple combinatorial property of $\mscr{T}_S$--parabolic subsets of $S$, which we term \emph{inflexibility} (\Cref{prop:inflexible}). In \Cref{sect:markings}, we generalise in various directions the results of \cite{CP10} on markings and hierarchies; here the main result is \Cref{prop:markings_new}. Finally, \Cref{sect:geometrisation} constitutes the true core of the article, and its most involved part: it contains the proof of \Cref{thm:step_two} and of the ``shortening theorem'' (\Cref{thm:shortening}).

In \Cref{app:automorphisms}, we prove finite generation of the automorphism group $\Aut(W)$ (\Cref{corintro:automorphisms}), and we briefly recall the additional results from \cite{Howlett-Muehlherr,Marquis-Muehlherr} that combined with \Cref{thmintro:main} yield the solution to the Isomorphism Problem for Coxeter groups (\Cref{corintro:iso_problem}).

\section{Splittings}\label{sect:splittings}

This section is devoted to actions of Coxeter groups on simplicial trees, and to developing some of the ideas in \cite{MT09} for our needs. The main result is \Cref{cor:universally_parabolic_vertex_stabilisers}.

\subsection{Generalities}

Let $W$ be a Coxeter group. A \emph{splitting} is a non-elliptic, minimal action $W\acts \mc{T}$ on a simplicial tree without edge-inversions. Given a family $\mscr{C}$ of subsets\footnote{One typically works with families of \emph{subgroups} rather than \emph{subsets} in JSJ theory, but the former are more convenient here. The reader can safely replace each subset with the subgroup that it generates if they are so inclined.} of $W$, the splitting is \emph{relative to} $\mscr{C}$ if each set in $\mscr{C}$ fixes a vertex of $\mc{T}$. The splitting is \emph{one-edge} if $W$ acts edge-transitively. Since $W$ is generated by torsion elements, the quotient $W/\mc{T}$ is a (finite) tree for all splittings $\mc{T}$. In particular, there exist \emph{fundamental subtrees}: finite subtrees $\mc{F}\sq \mc{T}$ on which the quotient projection restricts to a bijection $\mc{F}\ra \mc{T}/W$.

A splitting $W\acts \mc{T}_1$ \emph{dominates} another splitting $W\acts \mc{T}_2$ if the subgroups that are elliptic in $\mc{T}_1$ are also elliptic in $\mc{T}_2$; equivalently, there is a continuous $W$--equivariant map $\mc{T}_1\ra \mc{T}_2$. We say that $\mc{T}_1$ \emph{strongly dominates} $\mc{T}_2$ if $\mc{T}_1$ dominates $\mc{T}_2$ and, in addition, each edge-stabiliser of $\mc{T}_1$ is contained in an edge-stabiliser of $\mc{T}_2$.

A splitting is \emph{strongly reduced} if no vertex-stabiliser fixes an edge. Note that, in a strongly reduced splitting, vertex-stabilisers are precisely maximal elliptic subgroups.

\begin{lem}\label{lem:strongly_reduced}
    Let $W\acts \mc{T}$ be a splitting such that no edge-stabiliser properly contains a conjugate of itself. Then there exists a splitting $W\acts \mc{T}_*$, obtained from $\mc{T}$ by collapsing edges, such that $\mc{T}_*$ is strongly reduced and has the same elliptic subgroups as $\mc{T}$.
\end{lem}
\begin{proof}
    See the proof of \cite[Lemma~4.13]{Fio11}. The fact that $\mc{T}/W$ is simply connected ensures that $\mc{T}$ and $\mc{T}_*$ have the same elliptic subgroups.
\end{proof}

For any Coxeter generating set $S\sq W$ and every $S$--parabolic subgroup $P\leq W$, we have that $P$ does not properly contain any conjugates of itself. This allows us to apply \Cref{lem:strongly_reduced} to splittings with parabolic edge groups, and this is the only situation in which we will make use of the lemma.

\subsection{Visual splittings}\label{sub:visual}

Let $(W,S)$ be a Coxeter system. The following notion originated in \cite{MT09}.

\begin{defn}\label{defn:S-visual}
    A splitting $W\acts \mc{T}$ is \emph{$S$--visual} if there exists a fundamental subtree $\mc{F}\sq \mc{T}$ with the property that, for each vertex $v\in \mc{F}$, there exists a subset $S_v\sq S$ such that the $W$--stabiliser of $v$ is precisely the standard parabolic subgroup $\langle S_v\rangle$.
\end{defn}

Note that the $W$--stabiliser of each edge $e\sq \mc{F}$ is then also of the form $\langle S_e\rangle$ for a subset $S_e\sq S$. We speak of \emph{$S$--amalgams} when referring to $S$--visual splittings that are edge-transitive. 

\begin{rmk}\label{rmk:finitely_many_visual}
    For each Coxeter generating set $S\sq W$, there exist only finitely many strongly reduced, $S$--visual splittings up to $W$--equivariant isomorphisms of trees.
\end{rmk}

There is also the following related notion, which we formulate for general subsets for later use.

\begin{defn}\label{defn:visual_decomposition}
    Let $\Om\sq S$ be a subset. A \emph{visual decomposition} of $\Om$ is the data of a finite tree $\mc{F}$ together with a subset $\Om_v\sq\Om$ for each vertex $v\in \mc{F}$, so that the following conditions hold:
    \begin{enumerate}
        \item for each $\om\in\Om$, the set $\{v\in \mc{F}\mid \om\in\Om_v\}$ is the vertex set of a nonempty subtree of $\mc{F}$;
        \item if $x,y\in\Om$ and $xy$ has finite order, then there exists a vertex $v\in \mc{F}$ with $\{x,y\}\sq\Om_v$;
        \item $\mc{F}$ has at least one edge and, for every leaf $v\in \mc{F}$ with adjacent vertex $w$, we have $\Om_v\not\sq\Om_w$.
    \end{enumerate}
    Given a family $\mscr{C}$ of subsets of $W$, a visual decomposition of $\Om$ is \emph{relative to $\mscr{C}$} if, for every set $C\in\mscr{C}$ with $C\sq\langle\Om\rangle$, there exist an element $g\in W$ and a vertex $v\in\mc{F}$ with $gCg^{-1}\sq\langle\Om_v\rangle$.
\end{defn}

Given a visual decomposition $\mf{F}=(\mc{F},\{\Om_v\})$ and an edge $e\sq\mc{F}$ with vertices $v,w$, we will often use the notation $\Om_e:=\Om_v\cap\Om_w$. As for splittings, we say that $\mf{F}$ is \emph{strongly reduced} if we do not have any inclusions $\Om_v\sq\Om_w$ for distinct vertices $v,w\in\mc{F}$.

When the finite tree $\mc{F}$ is a single edge, a visual decomposition of $S$ simply corresponds to a decomposition $S=S_1\cup S_2$ such that the differences $S_1\setminus S_2$ and $S_2\setminus S_1$ are nonempty, and such that there are no elements $s_1\in S_1\setminus S_2$ and $s_2\in S_2\setminus S_1$ with $s_1s_2$ of finite order. Such a decomposition determines a splitting of $W$ as the amalgamated product of the subgroups $\langle S_1\rangle$ and $\langle S_2\rangle$ over $\langle S_1\cap S_2\rangle$. With an abuse, we refer also to the writing $S=S_1\cup S_2$ as an \emph{$S$--amalgam}. Note that the decompositions considered in the definition of elementary twists (in Sections~\ref{sect:intro} and~\ref{sect:strategy}) are a particular case of this construction, setting $S_1:=X\cup K\cup K^{\perp}$ and $S_2:=K\cup K^{\perp}\cup Y$.

A family $\mscr{C}$ of subsets of $W$ is \emph{conjugacy-closed} if we have $C'\in\mscr{C}$ whenever $C'=gCg^{-1}$ for some $C\in\mscr{C}$ and $g\in W$. We record here the following observation.

\begin{rmk}\label{rmk:amalgams_rel_C}
    Consider a decomposition $S=S_1\cup S_2$ with $S_1\not\sq S_2$ and $S_2\not\sq S_1$. 
    \begin{enumerate}
        \item The decomposition is an $S$--amalgam if and only if no element of $S_1\setminus S_2$ is adjacent to an element of $S_2\setminus S_1$ in the presentation graph of $(W,S)$ (the graph with $S$ as vertex set and edges joining pairs of elements involved in a relation within the Coxeter presentation).
        \item Let $\mscr{C}$ be a conjugacy-closed family of subsets of $W$ containing all spherical subsets of $S$. Then the decomposition is an $S$--amalgam relative to $\mscr{C}$ if and only if no element of $S_1\setminus S_2$ is adjacent to an element of $S_2\setminus S_1$ within the graph $\wh S_{\mscr{C}}$ (\Cref{defn:hat_graph}).
    \end{enumerate}
\end{rmk}

The equivalence between Definitions~\ref{defn:S-visual} and~\ref{defn:visual_decomposition} is straightforward and left to the reader:

\begin{lem}\label{lem:visual_equivalence}
    Let $\mscr{C}$ be a conjugacy-closed family of subsets of $W$.
    \begin{enumerate}
        \item Each visual decomposition $\mf{F}=(\mc{F},\{S_v\}_v)$ of $S$ relative to $\mscr{C}$ gives rise to a unique $S$--visual splitting $W\acts\mc{T}$ relative to $\mscr{C}$ with a fundamental subtree identified with $\mc{F}$, whose vertex-stabilisers are the subgroups $\langle S_v\rangle$. 
        \item Conversely, if $W\acts \mc{T}$ is an $S$--visual splitting relative to $\mscr{C}$ and $\mc{F}\sq \mc{T}$ is a fundamental subtree as in \Cref{defn:S-visual}, then the data of the subsets $S_v\sq S$ generating the $W$--stabilisers of the vertices of $\mc{F}$ forms a visual decomposition of $S$ relative to $\mscr{C}$.
    \end{enumerate}
\end{lem}

It is also easy to see that correspondence between visual splittings and visual decompositions in \Cref{lem:visual_equivalence} preserves the property of being strongly reduced.

\begin{rmk}
    We work with conjugacy-closed families $\mscr{C}$ of subsets of $W$ throughout the article. However, we would like to point out that, in most cases, the only thing that matters is the finite collection of subsets of $S$ that lie in $\mscr{C}$. We prefer to work with infinite, conjugacy-closed families for two reasons: we will not need to modify $\mscr{C}$ every time we change $S$ by an elementary twist (relative to $\mscr{C}$), and we will not need to worry about distinct subsets of $S$ being conjugate to each other.
\end{rmk}

\subsection{Domination by visual splittings}

The following is a rephrasing of the proof of \cite[Theorem~1]{MT09}; we recall the brief argument for the reader's convenience.

\begin{lem}\label{lem:MT}
    Let $W$ be a Coxeter group with a splitting $W\acts \mc{T}$. For each Coxeter generating set $S\sq W$, there exists a strongly reduced $S$--visual splitting $W\acts \mc{T}_S$ strongly dominating $\mc{T}$. Moreover, $\mc{T}$ and $\mc{T}_S$ have the same $S$--parabolic elliptic subgroups.
\end{lem}
\begin{proof}
    Let $\overline{\mc{F}}\sq \mc{T}$ be the smallest subtree that intersects the fixed sets of all elements of $S$. For each vertex $\overline v\in\overline{\mc{F}}$ and edge $\overline e\sq\overline{\mc{F}}$, let $S_{\overline v}$ and $S_{\overline e}$ be the subsets of $S$ fixing $\overline v$ and $\overline e$, respectively. (The $W$--stabiliser of $\overline v$ contains the subgroup $\langle S_{\overline v}\rangle$, but it can be strictly larger in general.) Observe that the data of the tree $\overline{\mc{F}}$ and subsets $S_{\overline v}\sq S$ forms a visual decomposition of $S$. Using \Cref{lem:visual_equivalence}, we obtain an $S$--visual splitting $W\acts \mc{T}_S$ with a fundamental subtree $\mc{F}\sq \mc{T}_S$ identified with $\overline{\mc{F}}$ via a map that we denote by $v\mapsto\overline v$ and $e\mapsto\overline e$. The $W$--stabilisers of vertices $v\in \mc{F}$ and edges $e\sq\mc{F}$ are \emph{precisely} the subgroups $\langle S_{\overline v}\rangle$ and $\langle S_{\overline e}\rangle$. Thus, it is clear that $\mc{T}_S$ strongly dominates $\mc{T}$.
    
    Since the edge-stabilisers of $\mc{T}_S$ are parabolic, none of them properly contains a conjugate of itself, and so we can use \Cref{lem:strongly_reduced} to make $\mc{T}_S$ strongly reduced. This does not affect the elliptic subgroups of $\mc{T}_S$ and does not increase its family of edge-stabilisers, so it does not affect strong domination of $\mc{T}$. Finally, if $P\leq W$ is $S$--parabolic and elliptic in $\mc{T}$, then $P$ is conjugate to a subgroup of $\langle S_{\overline v}\rangle$ for some $\overline v\in\overline{\mc{F}}$ (by Helly's lemma for trees). This subgroup fixes the vertex $v\in \mc{F}\sq \mc{T}_S$, and so $P$ is elliptic in $\mc{T}_S$. Together with domination, this shows that $\mc{T}$ and $\mc{T}_S$ have the same $S$--parabolic elliptic subgroups, proving the lemma.
\end{proof}

The previous lemma yields the following simple parabolicity criterion.

\begin{lem}\label{lem:universally_elliptic->parabolic}
    Let $W\acts \mc{T}$ be a strongly reduced splitting relative to a family $\mscr{C}$. Let $v\in \mc{T}$ be a vertex whose stabiliser $W_v$ is elliptic in all splittings of $W$ relative to $\mscr{C}$ strongly dominating $\mc{T}$. Then $W_v$ is $S$--parabolic for every Coxeter generating set $S\sq W$ such that all sets in $\mscr{C}$ are $S$--parabolic.
\end{lem}
\begin{proof}
    Let $S\sq W$ be a Coxeter generating set such that all sets in $\mscr{C}$ are $S$--parabolic. By \Cref{lem:MT}, there exists an $S$--visual splitting $W\acts \mc{T}'$ that strongly dominates $\mc{T}$ and in which all sets in $\mscr{C}$ are still elliptic. By hypothesis, it follows that the subgroup $W_v$ fixes some vertex $w\in \mc{T}'$. Letting $f\colon \mc{T}'\ra \mc{T}$ be a continuous $W$--equivariant map sending vertices to vertices, we have $W_v\leq W_w\leq W_{f(w)}$. Since $\mc{T}$ is strongly reduced, we have $f(w)=v$ and hence $W_v=W_w$. Since $\mc{T}'$ is $S$--visual, this shows that $W_v$ is $S$--parabolic, as desired. 
\end{proof}

The following consequence of \Cref{lem:MT} is particularly useful because it can be used to produce subgroups $P\leq W$ that are simultaneously parabolic with respect to \emph{all} Coxeter generating sets of $W$ (or with respect to those in a given subclass). The starting point is an arbitrary splitting of $W$:

\begin{cor}\label{cor:universally_parabolic_vertex_stabilisers}
    Let $W$ be a Coxeter group with a splitting $W\acts \mc{T}$. Let $\mscr{S}$ be a family of Coxeter generating sets of $W$, and consider some element $S\in\mscr{S}$. Then $\mc{T}$ is strongly dominated by a splitting $W\acts \mc{T}_{\mscr{S},S}$ such that:
    \begin{enumerate}
        \item $\mc{T}_{\mscr{S},S}$ is strongly reduced and $S$--visual;
        \item the vertex-stabilisers of $\mc{T}_{\mscr{S},S}$ are $\mscr{S}$--parabolic;
        \item $\mc{T}$ and $\mc{T}_{\mscr{S},S}$ have the same $\mscr{S}$--parabolic elliptic subgroups.
    \end{enumerate}
\end{cor}
\begin{proof}
    Let $\mf{T}$ be the family of strongly reduced, $S$--visual splittings that strongly dominate $\mc{T}$ and have the same $\mscr{S}$--parabolic elliptic subgroups as $\mc{T}$. Observe that $\mf{T}$ is nonempty by \Cref{lem:MT}.

    For $\mc{T}_1,\mc{T}_2\in\mf{T}$, we write $\mc{T}_1\leq \mc{T}_2$ if $\mc{T}_2$ strongly dominates $\mc{T}_1$, and we write $\mc{T}_1\sim \mc{T}_2$ when $\mc{T}_1\leq \mc{T}_2\leq \mc{T}_1$. An equivalence $\mc{T}_1\sim\mc{T}_2$ implies that $\mc{T}_1$ and $\mc{T}_2$ have the same elliptic subgroups. The relation $\leq$ is transitive and so it descends to a partial order on the quotient $\mf{T}/\sim$. 
    
    Now, since $\mf{T}$ is finite by \Cref{rmk:finitely_many_visual}, there exists a splitting $\mc{T}_{\max}\in\mf{T}$ such that $[\mc{T}_{\max}]$ is a maximal element of $\mf{T}/\sim$. We set $\mc{T}_{\mscr{S},S}:=\mc{T}_{\max}$ and we are only left to verify Item~(2).

    Thus, consider a vertex $v\in \mc{T}_{\max}$ and any generating set $R\in\mscr{S}$; we will show that the stabiliser $W_v$ is $R$--parabolic. In view of \Cref{lem:universally_elliptic->parabolic}, it suffices to show that $W_v$ is elliptic in every splitting $W\acts\mc{T}'$ that strongly dominates $\mc{T}_{\max}$ and has the same $\mscr{S}$--parabolic elliptic subgroups as $\mc{T}_{\max}$ (and hence as $\mc{T}$). If $\mc{T}'$ is such a splitting, then \Cref{lem:MT} yields another splitting $\mc{T}''$ that strongly dominates $\mc{T}'$, has the same $\mscr{S}$--parabolic elliptic subgroups as $\mc{T}_{\max}$, and is in addition $S$--visual. In other words $\mc{T}''\in\mf{T}$ and $\mc{T}_{\max}\leq\mc{T}'\leq\mc{T}''$, which implies that $\mc{T}''\sim\mc{T}_{\max}$. The latter implies that $W_v$ is elliptic in $\mc{T}''$, and hence $W_v$ is also elliptic in $\mc{T}'$, as desired.
\end{proof}

Note that, if we enlarge the family $\mscr{S}$ in \Cref{cor:universally_parabolic_vertex_stabilisers}, then the information provided by Item~(2) becomes stronger, but the information provided by Item~(3) becomes weaker.

In \Cref{cor:universally_parabolic_vertex_stabilisers}, the best we can do is finding a splitting that has $\mscr{S}$--parabolic vertex-stabilisers and is $S$--visual for \emph{one} generating set $S\in\mscr{S}$. In general, we cannot find a splitting that is $S$--visual for all $S\in\mscr{S}$, as the next example shows.

\begin{ex}
    Let $W\cong\Z/2\Z\ast\Z/2\Z\ast\Z/2\Z$. Consider a Coxeter generating set $S=\{x,y,z\}$ and set $R:=\{x,y,xyzyx\}$. Let $W\acts\mc{T}$ be any splitting with vertex-stabilisers conjugate to $\langle x\rangle$, $\langle y\rangle$ and $\langle z\rangle$. Then $\mc{T}$ is not dominated by any splitting that is both $S$-- and $R$--visual.
\end{ex}

\section{Shrubs}\label{sect:shrubs}

This section is concerned with the proof of \Cref{thm:step_one}, which constitutes the first step of the proof of \Cref{thmintro:main} sketched in \Cref{sect:strategy}. More precisely, \Cref{sub:2-compatibility} is devoted to showing that, if $S,R\sq W$ are angle-compatible Coxeter generating sets and $\mscr{C}$ is the family of $\{S,R\}$--compatible subsets of $W$, then all cliques of the graph $\wh S_{\mscr{C}}$ (\Cref{defn:hat_graph}) originate from sets in $\mscr{C}$ (\Cref{cor:2-compatible->compatible}). This is actually a fairly straightforward result once we assume the deeper work in \cite{CM07}. Then, in \Cref{sub:shrub_defn}, we define shrubs (\Cref{defn:shrub}), tucked sets (\Cref{defn:tucked}), and use the material in Sections~\ref{sect:splittings} and~\ref{sub:2-compatibility} to finally prove \Cref{thm:step_one} (see \Cref{thm:step_one_repeated}).

\subsection{$2$--compatibility vs compatibility}\label{sub:2-compatibility}

Let $(W,S)$ be a Coxeter system. We say that a subset $U\sq W$ is \emph{$2$--compatible} (with respect to $S$) if $U$ consists of $S$--reflections and every cardinality--$2$ subset of $U$ is conjugate to a subset of $S$. We will see that, in our cases of interest, $2$--compatible subsets are in fact $S$--compatible (\Cref{cor:2-compatible->compatible}). We begin with some preliminary results.

\begin{lem}\label{lem:3_reflections}
    Let $U=\{r_1,r_2,r_3\}$ be a $2$--compatible set of cardinality $3$. If $U$ is $2$--spherical and non-spherical, assume in addition that $U$ is universal. Then $U$ is $2$--geometric in $\A_S$.
\end{lem}
\begin{proof}
    Consider the Davis complex $\A_S$ and let $\mc{W}_1,\mc{W}_2,\mc{W}_3\sq\A_S$ be the walls fixed by the reflections $r_i$. Our goal is to choose halfspaces $\mc{H}_i\sq\A_S$ bounded by the $\mc{W}_i$ so that, for all $j\neq k$, the pair $\{\mc{H}_j,\mc{H}_k\}$ is geometric, that is, $\mc{H}_j\cap\mc{H}_k$ is a fundamental domain for the action $\langle r_j,r_k\rangle\acts\A_S$. 

    Suppose first that the three walls pairwise intersect, that is, the set $U$ is $2$--spherical. If two of the reflections commute, say $r_1$ and $r_3$, then we can simply choose $\mc{H}_1$ arbitrarily, then choose $\mc{H}_2$ so that $\{\mc{H}_1,\mc{H}_2\}$ is geometric, and finally choose $\mc{H}_3$ so that $\{\mc{H}_2,\mc{H}_3\}$ is geometric. If no two of the reflections commute, then $U$ is non-spherical and hence universal by hypothesis; we can then invoke \cite[Lemma~11.2]{CM07} to obtain that $U$ is $2$--geometric.

    Suppose now that the three walls are pairwise disjoint. Since any two of the $r_i$ can be simultaneously conjugated into $S$, no two of the three walls are separated by another wall of $\A_S$. Thus, it suffices to define each $\mc{H}_i$ as the side of $\mc{W}_i$ containing the other two walls. 

    In the rest of the proof, we can thus assume that $\mc{W}_1\cap\mc{W}_2=\emptyset$ and $\mc{W}_1\cap\mc{W}_3\neq\emptyset$. We define $\mc{H}_1$ and $\mc{H}_2$, respectively, as the side of $\mc{W}_1$ containing $\mc{W}_2$, and as the side of $\mc{W}_2$ containing $\mc{W}_1$. In order to be able to pick $\mc{H}_3$, we need to check that the following two configurations cannot arise:
    \begin{enumerate}
        \item The wall $\mc{W}_2$ is disjoint from $\mc{W}_3$ and contained in an obtuse-angled sector of $\A_S\setminus(\mc{W}_1\cup\mc{W}_3)$.
        \item The wall $\mc{W}_2$ crosses $\mc{W}_3$ and, for any choice of $\mc{H}_3$, one of the sectors $\mc{H}_1\cap\mc{H}_3$, $\mc{H}_2\cap\mc{H}_3$ is acute-angled and the other is obtuse-angled.
    \end{enumerate}
    By a result of Deodhar and Dyer \cite{Deodhar-geometric,Dyer}, used for instance in the formulation from \cite[Theorem~A.1]{CP10}, it suffices to rule out the above two configurations under the additional assumption that $|S|=3$ and at least one of the labels of $S$ is infinite. 
    This is readily verified.
\end{proof}

\begin{prop}\label{prop:2-compatible->geometric}
    Let $U\sq W$ be finite, universal and $2$--compatible. Then $U$ is $S$--geometric.
\end{prop}
\begin{proof}
    By \Cref{lem:1.6}, it suffices to find a $2$--geometric choice of halfspaces $\mc{H}_u\sq\A_S$ for $u\in U$, in the sense that any two of these halfspaces form a geometric pair. We argue by induction on $|U|$, the base step being immediate. Throughout, we say that two elements of $U$ are \emph{adjacent} if they are adjacent in the presentation graph of $(W,S)$, that is, if their product has finite order.
    
    If $U$ is $2$--spherical, then $U$ is geometric by \cite[Proposition~11.7]{CM07}. Thus, we can assume that there is an element $x\in U$ such that the set $N\sq U\setminus\{x\}$ of elements non-adjacent to $x$ is nonempty. By the inductive hypothesis, the set $U\setminus\{x\}$ is geometric and thus admits a $2$--geometric choice of halfspaces $\{\mc{H}_u\}_{u\in U\setminus\{x\}}$. By $2$--compatibility of $U$, the walls $\mc{W}_n$ with $n\in N$ are all on the same side of $\mc{W}_x$, and so we can define $\mc{H}_x$ as the side of $\mc{W}_x$ containing all these walls. 
    
    We are left to check (or rather ensure) that the pairs $\{\mc{H}_x,\mc{H}_u\}$ are geometric for $u\in U\setminus\{x\}$. The following is the main observation needed for this.

    \smallskip
    {\bf Claim.} \emph{If $U\setminus\{x\}$ is not $2$--spherical, then $\mc{W}_x\sq\mc{H}_n$ for all $n\in N$. If $U\setminus\{x\}$ is $2$--spherical, then the same conclusion holds, possibly after flipping all $\mc{H}_u$ with $u\in U\setminus\{x\}$.}

    \smallskip\noindent
    \emph{Proof of claim.}
    Suppose that there exists $w\in N$ such that $\mc{W}_x\sq\mc{H}_w^*$. This implies that we actually have $\mc{W}_x\sq\mc{H}_n^*$ for all $n\in N$, due to having ruled out Configuration~(1) in the proof of \Cref{lem:3_reflections}. It follows that each element of $N$ is adjacent to all elements of $U\setminus\{x\}$. 
    
    If $U\setminus\{x\}$ is $2$--spherical, the claim is proven. Therefore, we suppose that there is a non-adjacent pair $a,b\in U\setminus\{x\}$ and aim for a contradiction. Consider the four walls $\mc{W}_x,\mc{W}_w,\mc{W}_a,\mc{W}_b$, for some $w\in N$. By the previous paragraph, $a$ and $b$ cannot lie in $N$, and so they are adjacent to both $x$ and $w$. The pairs $(a,b)$ and $(x,w)$ are instead non-adjacent. Note that $\mc{H}_a$ and $\mc{H}_b$ contain $\mc{W}_b$ and $\mc{W}_a$, respectively, and the sectors $\mc{H}_a\cap\mc{H}_w$ and $\mc{H}_b\cap\mc{H}_w$ are both acute-angled, by construction. Having ruled out Configuration~(2) in the proof of \Cref{lem:3_reflections}, it follows that the sectors $\mc{H}_a\cap\mc{H}_x^*$ and $\mc{H}_b\cap\mc{H}_x^*$ are also acute-angled. 
    
    Now, working with respect to the CAT(0) metric on $\A_S$ (see \cite[Chapter~12]{Davis}), let $Q$ be a minimal-perimeter geodesic quadrilateral with vertices in the four intersections $\mc{W}_x\cap\mc{W}_a$, $\mc{W}_a\cap\mc{W}_w$, $\mc{W}_w\cap\mc{W}_b$ and $\mc{W}_b\cap\mc{W}_x$. Since all walls are convex subspaces with respect to the CAT(0) metric, the angles of $Q$ then equal the dihedral angles between these four wall-pairs, and so all four of these angles are obtuse. This violates the CAT(0) property (see e.g.\ \cite[Theorem~II.2.11]{BH}).
    \hfill$\blacksquare$

    \smallskip
    By the claim, we can assume that $\mc{W}_x\sq\mc{H}_n$ for all $n\in N$, which means that the pairs $\{\mc{H}_x,\mc{H}_n\}$ are all geometric. For $n\in N$ and $r\in U\setminus(N\cup\{x\})$, the sector $\mc{H}_n\cap\mc{H}_r$ is acute-angled, and so the sector $\mc{H}_x\cap\mc{H}_r$ must also be acute-angled (again, because we ruled out Configuration~(2) in \Cref{lem:3_reflections}). In conclusion, the halfspace family $\{\mc{H}_u\}_{u\in U}$ is $2$--geometric, proving the lemma.
\end{proof}

We now come to the main result of this subsection. Note that, in terms of the graph $\wh S_{\mscr{C}}$ from \Cref{defn:hat_graph}, Item~(1) of the corollary is simply saying that subsets of $S$ span cliques in $\wh S_{\mscr{C}}$ if and only if they lie in $\mscr{C}$ (the forward implication being the only nontrivial one).

\begin{cor}\label{cor:2-compatible->compatible}
    Let $S,R\sq W$ be angle-compatible and $\mscr{C}$ the family of $\{S,R\}$--compatible sets.
    \begin{enumerate}
        \item If a subset $U\sq S$ is $2$--compatible with respect to $R$, then $U$ is $R$--compatible.
        \item If a subset $U\sq S$ is not $R$--compatible, then there exists a splitting of $W$ relative to $\mscr{C}$ in which the subgroup $\langle U\rangle$ is non-elliptic.
    \end{enumerate}
\end{cor}
\begin{proof}
    We begin with the following observation, which is an (a priori) weaker version of Item~(2).

    \smallskip
    {\bf Claim.} \emph{If a subset $U\sq S$ is not $2$--compatible with respect to $R$, then $\langle U\rangle$ is non-elliptic in a splitting of $W$ relative to $\mscr{C}$.}

    \smallskip\noindent
    \emph{Proof of claim.}
    Since $U$ is not $2$--compatible with $R$, there exist a pair of distinct elements $u,u'\in U$ that are not adjacent within the graph $\wh S_{\mscr{C}}$. Let ${\rm st}_{\mscr{C}}(u)$ be the star of $u$ in $\wh S_{\mscr{C}}$. Since $S$ and $R$ are angle-compatible, the decomposition $S={\rm st}_{\mscr{C}}(u)\cup(S\setminus\{u\})$ is an $S$--amalgam relative to $\mscr{C}$ (see \Cref{rmk:amalgams_rel_C}). Let $W\acts \mc{T}$ be the splitting determined by this visual decomposition through \Cref{lem:visual_equivalence}, and note that $\mc{T}$ is relative to $\mscr{C}$. Finally, the element $uu'$ is loxodromic in $\mc{T}$ because $u'\not\in{\rm st}_{\mscr{C}}(u)$ and $u\not\in S\setminus\{u\}$. Thus, $\langle U\rangle$ is indeed non-elliptic in $\mc{T}$.
    \hfill$\blacksquare$

    \smallskip
    In view of the claim, it suffices to prove Item~(1) of the corollary. Thus, consider a subset $U\sq S$ that is $2$--compatible with $R$. Suppose for a moment that $U$ is contained in a proper $\{S,R\}$--parabolic subgroup of $W$, and let $P$ be a minimal such subgroup. Note that $P$ has two angle-compatible Coxeter generating sets $S'$ and $R'$ that are conjugate to subsets of $S$ and $R$, respectively. Moreover, a $P$--conjugate of $U$ is contained in $S'$. Thus, up to replacing $S$ and $R$ with $S'$ and $R'$, we can safely assume that $U$ is not contained in any proper $\{S,R\}$--parabolic subgroups of $W$.

    If $S$ is itself $2$--compatible with respect to $R$, then $S$ is $R$--geometric by \Cref{prop:2-compatible->geometric}, and it follows from \Cref{lem:geometric+parabolic} that $S$ is $R$--compatible. In this case, the subset $U$ is certainly $R$--compatible. Thus, we can assume in the rest of the proof that $S$ is not $2$--compatible with $R$, and thus the claim yields the existence of a splitting $W\acts \mc{T}$ relative to $\mscr{C}$. By \Cref{cor:universally_parabolic_vertex_stabilisers}, there exists also a splitting $W\acts \mc{T}'$ relative to $\mscr{C}$ with $\{S,R\}$--parabolic vertex-stabilisers.

    Now, all cardinality--$2$ subsets of $U$ lie in $\mscr{C}$, and so they are elliptic in $\mc{T}'$. It follows from Serre's lemma \cite[p.\,64]{Serre} that $U$ is elliptic in $\mc{T}'$. Therefore, $U$ is contained in a proper $\{S,R\}$--parabolic subgroup of $\mc{T}$, contradicting our assumptions and thus proving the corollary.
\end{proof}

\subsection{Finding shrubs}\label{sub:shrub_defn}

We are now ready to start working towards the proof of \Cref{thm:step_one}. Thus, let $S,R\sq W$ be angle-compatible Coxeter generating sets, and let $\mscr{C}$ be the family of $\{S,R\}$--compatible subsets of $W$. We begin by defining shrubs more precisely in terms of \Cref{defn:visual_decomposition}.

\begin{defn}\label{defn:shrub}
    A subset $\Om\sq S$ is a \emph{$\mscr{C}$--shrub} if it admits a strongly reduced visual decomposition $\mf{F}=(\mc{F},\{\Om_v\}_v)$ with the following two properties:
    \begin{enumerate}
        \item $\mf{F}$ is relative\footnote{Recall that this simply means that all subsets of $\Om$ that lie in $\mscr{C}$ also lie in some $\Om_v$. In general, there are many sets in $\mscr{C}$ with no conjugates contained in $\Om$.} to $\mscr{C}$, and we have $\Om_v\in\mscr{C}$ for all $v\in\mc{F}$;
        \item for every vertex $v\in\mc{F}$ and any two edges $e,f\sq\mc{F}$ incident to $v$, the sets $\Om_e$ and $\Om_f$ are conjugate within the subgroup $\langle\Om_v\rangle$.
    \end{enumerate}
    Note that we do \emph{not} require $\mf{F}$ to extend to a visual decomposition of the generating set $S$.
\end{defn}

Regarding Item~(2) of \Cref{defn:shrub}, note that there exists an element $g\in\langle\Om_v\rangle$ with $g\Om_eg^{-1}=\Om_f$ if and only if there exists an element $w\in W$ with $w\langle\Om_e\rangle w^{-1}=\langle\Om_f\rangle$ (see for instance \cite[Sec\-tion~4.10]{Davis}). We will also need the following observation.

\begin{rmk}\label{rmk:disconnecting_shrub}
    Let $\Om\sq S$ be a set with a visual decomposition $(\mc{F},\{\Om_v\}_v)$ satisfying Item~(1) of \Cref{defn:shrub}. If $\Delta\sq\Om$ is a subset that does not contain any set $\Om_e$ with $e\sq\mc{F}$ an edge, then $\Delta$ does not disconnect the graph $\wh\Om_{\mscr{C}}$. Indeed, given two points $\om,\om'\in\Om$, we can construct a path connecting them in $\wh\Om_{\mscr{C}}\setminus\Delta$ as follows. Choose vertices $w,w'\in\mc{F}$ such that $\om\in\Om_w$ and $\om'\in\Om_{w'}$. Let $\g\sq\mc{F}$ be the geodesic from $w$ to $w'$. Choosing a point $\om_e\in\Om_e\setminus\Delta$ for each edge $e\sq\g$, we obtain the required path from $\om$ to $\om'$ (note that each set $\Om_v$ spans a clique in $\wh\Om_{\mscr{C}}$).
\end{rmk}

The first step in the proof of \Cref{thm:step_one} is the following result, which we already hinted at in \Cref{sect:strategy}. In fact, for Coxeter groups of FC-type, this constitutes the entire proof. As mentioned, $\mscr{T}_S\cup\{R\}$--parabolic subsets of $S$ not lying in $\mscr{C}$ indeed exist, as soon as $S$ and $R$ are not conjugate.

\begin{prop}\label{prop:minimal->shrub}
     Let $\mscr{T}_S$ be the collection of generating sets twist-equivalent to $S$ relative to $\mscr{C}$. If $\Om\sq S$ is minimal among $\mscr{T}_S\cup\{R\}$--parabolic subsets of $S$ not lying in $\mscr{C}$, then $\Om$ is a $\mscr{C}$--shrub. 
\end{prop}
\begin{proof}
    Since $\Om$ is contained in $S$ and does not lie in $\mscr{C}$, it is not $R$--compatible. By \Cref{cor:2-compatible->compatible}(2), it follows that $\langle\Om\rangle$ is non-elliptic in a splitting of $W$ relative to $\mscr{C}$. By \Cref{cor:universally_parabolic_vertex_stabilisers}, $\langle\Om\rangle$ is also non-elliptic in an $S$--visual splitting $W\acts \mc{T}$ relative to $\mscr{C}$ that has $\mscr{T}_S\cup\{R\}$--parabolic vertex-stabilisers. The $\langle\Om\rangle$--stabiliser of each vertex of $\mc{T}$ is a proper, $\mscr{T}_S\cup\{R\}$--parabolic subset of $\Om$ and so, by minimality of $\Om$, each of these vertex-stabilisers is generated by a set in $\mscr{C}$. Using \Cref{lem:visual_equivalence}, the splitting $\mc{T}$ corresponds to a visual decomposition of $S$ relative to $\mscr{C}$. Restricting the latter to $\Om$, we obtain a visual decomposition $\mf{F}=(\mc{F},\{\Om_v\}_v)$ of $\Om$ relative to $\mscr{C}$ with $\Om_v\in\mscr{C}$ for all vertices $v\in\mc{F}$. Up to collapsing some edges of $\mc{F}$, we can assume that $\mf{F}$ is strongly reduced.
  
    We are only left to check that $\mf{F}$ satisfies Item~(2) of \Cref{defn:shrub}, which will be the subject of the rest of the proof. Let $\mc{E}$ be the collection of sets $\Om_e$ with $e$ varying through the edges of $\mc{F}$. Given two sets $E_1,E_2\in\mc{E}$, we write $E_1\sim E_2$ (resp.\ $E_1\leq E_2$) if there exists an element $g\in\langle\Om\rangle$ such that $gE_1g^{-1}=E_2$ (resp.\ $gE_1g^{-1}\sq E_2$). The relation $\leq$ descends to a partial order on $\mc{E}/\sim$. Our goal is to check that the quotient $\mc{E}/\sim$ is a singleton.

    Let $E\in\mc{E}$ be a set projecting to a $\leq$--minimal element of $\mc{E}/\sim$. Let $\mf{F}'=(\mc{F}',\{\Psi_w\}_w)$ be the visual decomposition of $\Om$ obtained by collapsing all edges $e\sq\mc{F}$ with $\Om_e\not\sim E$ (and consequently merging the subsets of $\Om$ associated to vertices of $\mc{F}$ getting collapsed to the same point of $\mc{F}'$). Supposing for the sake of contradiction that $\mc{E}/\sim$ is not a singleton, there exists a vertex $x\in\mc{F}'$ whose preimage $\mc{F}_x\sq\mc{F}$ is a subtree with at least two vertices. In particular, since $\mf{F}$ was strongly reduced and relative to $\mscr{C}$, we have $\Psi_x\not\in\mscr{C}$.
    
    Set $H:=\langle\Om\rangle$ and let $H\acts\mc{T}'$ be the splitting arising from the visual decomposition $\mf{F}'$ of $\Om$. This is a splitting relative to the family $\mscr{C}|_H$ of sets in $\mscr{C}$ contained in $H$. The tree $\mc{F}'$ is naturally identified with a fundamental subtree of $\mc{T}'$, so we can think of the vertex $x\in\mc{F}'$ as a vertex of $\mc{T}'$. The stabiliser $H_x$ is generated by the set $\Psi_x$, which is the union of the sets $\Om_v$ with $v\in\mc{F}_x$. This gives rise to a visual decomposition $\mf{F}_x=(\mc{F}_x,\{\Om_v\}_v)$ of the set $\Psi_x$. Note that, for each edge $e\sq\mc{F}_x$, the set $\Om_e$ does not admit any $W$--conjugates contained in the set $E$. 
    
    Now, we claim that the subgroup $H_x$ is elliptic in all splittings of $H$ relative to $\mscr{C}|_H$ that strongly dominate $\mc{T}'$. Otherwise, \Cref{lem:MT} would guarantee that $H_x$ is non-elliptic in an $\Om$--visual such splitting $H\acts\mc{T}''$. By \Cref{lem:visual_equivalence}, the latter splitting corresponds to one last visual decomposition $\mf{F}''=(\mc{F}'',\{\Phi_u\}_u)$ of the set $\Om$ relative to $\mscr{C}$. Considering any edge $f\sq\mc{F}''$, the fact that $\mc{T}''$ strongly dominates $\mc{T}'$ implies that we have $g\Phi_fg^{-1}\sq \langle E\rangle$ for some $g\in H$, and hence also $\overline g\Phi_f\overline g^{-1}\sq E$ for some other element $\overline g\in H$. At the same time, the fact that $H_x=\langle\Psi_x\rangle$ is non-elliptic in $\mc{T}''$ implies that there is an edge $f\sq\mc{F}''$ such that the set $\Psi_x\sq\Om$ intersects two distinct connected components of the graph $\wh \Om_{\mscr{C}}\setminus\Phi_f$ (recall \Cref{rmk:amalgams_rel_C}(2)). In particular, the graph $(\wh\Psi_x)_{\mscr{C}}\setminus (\Psi_x\cap\Phi_f)$ must be disconnected and so, applying \Cref{rmk:disconnecting_shrub} to the visual decomposition $\mf{F}_x$ and the set $\Delta:=\Psi_x\cap\Phi_f$, we see that $\Phi_f$ must contain $\Om_e$ for some edge $e\sq\mc{F}_x$. In conclusion, we have $E\supseteq\overline g\Phi_f\overline g^{-1}\supseteq\overline g\Om_e\overline g^{-1}$ for an edge $e\sq\mc{F}_x\sq\mc{F}$, violating our choice of $E$ and $x$.

    Finally, the claim in the previous paragraph combined with \Cref{lem:universally_elliptic->parabolic} implies that the subgroup $H_x$ is $\mscr{T}_S\cup\{R\}$--parabolic. Since $\Psi_x\not\in\mscr{C}$, this violates minimality of $\Om$, yielding one last contradiction and proving the proposition.
\end{proof}

Back to the statement of \Cref{thmintro:main}, it involves one last property that we have not yet introduced: ``tuckedness''. As mentioned, this is only needed to handle Coxeter groups not of FC-type, so the reader looking for a simplified proof of \Cref{thmintro:main} can skip the rest of this section.

\begin{defn}\label{defn:tucked}
    A subset $\Om\sq S$ is:
    \begin{itemize}
        \setlength\itemsep{.2em}
        \item \emph{constricted} if every non-spherical irreducible subset $J\sq\Om$ with $J\in\mscr{C}$ satisfies $J^{\perp}\cap\Om\in\mscr{C}$;
        \item \emph{$S$--tucked} if, for every non-spherical irreducible subset $J\sq S$ with $J\in\mscr{C}$ and such that $\Om$ intersects at least two connected components of the graph $\wh S_{\mscr{C}}\setminus (J\cup J^{\perp})$, we have that either the difference $J\setminus\Om$ is non-spherical, or $J\sq\Om$;
        \item \emph{$\mscr{T}_S$--tucked} if $\Om$ is $\mscr{T}_S$--parabolic and, for all generating sets $S'\in\mscr{T}_S$ and all subsets $\Om'\sq S'$ with $\langle\Om'\rangle=\langle\Om\rangle$, the set $\Om'$ is $S'$--tucked.
    \end{itemize}
\end{defn}

There is no direct connection between the properties of being constricted and tucked, but both are important in the following discussion. There is no need to define ``$\mscr{T}_S$--constricted'' sets, due to \Cref{rmk:constricted_independent} below.

Given two subsets $V\sq U\sq S$, we say that $V$ is a \emph{factor} of $U$ if $U\sq V\cup V^{\perp}$. An \emph{irreducible factor} is a factor that is both irreducible and nonempty. We say that $U$ is \emph{aspherical} if it has no spherical factors. When $U$ is aspherical, the centraliser of $U$ in $W$ coincides with the subgroup $\langle U^{\perp}\rangle$; see for instance \cite[Section~4.10]{Davis}. We emphasise that the latter fact badly fails when $U$ has spherical factors.

\begin{rmk}\label{rmk:constricted_independent}
    If $\Om\sq S$ is constricted and if we have $\Om'\sq S'\in\mscr{T}_S$ with $\langle\Om'\rangle=\langle\Om\rangle$, then $\Om'$ is automatically constricted with respect to $S'$. Indeed, for every non-spherical irreducible subset $J\sq\Om$ with $J\in\mscr{C}$, there exists a subset $J'\sq\Om'$ conjugate to $J$ (since $S$ and $S'$ differ by twists relative to $\mscr{C}$). Moreover, the orthogonal of $J$ within $\Om$ is conjugate to the orthogonal of $J'$ within $\Om'$: this is because $J$ is aspherical and so the subgroup $\langle J^{\perp}\rangle$ is simply the centraliser of $J$ within $W$. (Note that, instead, orthogonals of spherical sets can change much more drastically with a twist.)
\end{rmk}

\begin{rmk}\label{rmk:tucked_in_tucked}
    Given two subsets $V\sq U\sq S$, the following are straightforward observations.
    \begin{enumerate}
        \item If $U$ is constricted, then $V$ is constricted.
        \item If $V$ is $U$--tucked and $U$ is $S$--tucked, then $V$ is $S$--tucked.
    \end{enumerate}
\end{rmk}

We need two more lemmas, after which we will quickly be able to prove \Cref{thm:step_one}. 

\begin{lem}\label{lem:constricted}
    If $S$ and $R$ are not conjugate, there exists a $\mscr{T}_S\cup\{R\}$--parabolic subset $\Om\sq S$ such that $\Om\not\in\mscr{C}$ and $\Om$ is both constricted and $\mscr{T}_S$--tucked.
\end{lem}
\begin{proof}
    If $S$ is constricted, we can simply set $\Om:=S$. Thus, we can suppose that $S$ has nonempty aspherical subsets $J\in\mscr{C}$ with $J^{\perp}\not\in\mscr{C}$. Let $J$ be a maximal such set and define $\Om:= J^{\perp}$. Since $J$ is aspherical, we have the subgroup $\langle\Om\rangle$ coincides with the centraliser of $J$ in $W$, and this implies that $\Om$ is $T$--parabolic for every Coxeter generating set $T\sq W$ that contains a conjugate of $J$. In particular, $\Om$ is $\mscr{T}_S\cup\{R\}$--parabolic. Maximality of $J$ also clearly implies that $\Om$ is constricted.

    We are left to show that $\Om$ is $\mscr{T}_S$--tucked, and we begin by showing that it is $S$--tucked. Consider an irreducible subset $I\sq S$ such that $\Om$ intersects at least two connected component of $\wh S_{\mscr{C}}\setminus (I\cup I^{\perp})$. (Here it will not matter whether $I$ is spherical or not, or whether it lies in $\mscr{C}$.) Since $\Om=J^{\perp}$, every point of $\Om$ is adjacent to every point of $J$ within the graph $\wh S_{\mscr{C}}$. It follows that $J\sq I\cup I^{\perp}$ and hence every irreducible factor $J_i$ of $J$ is contained either in $I$ or in $I^{\perp}$. If we have $J_i\sq I$ for some $i$, then $I\setminus\Om$ is non-spherical. Otherwise, we have $J\sq I^{\perp}$ and hence $I\sq J^{\perp}=\Om$ as desired. 

    This shows that $\Om$ is $S$--tucked. In fact, the same exact proof can be carried out with respect to any other generating set $S'\in\mscr{T}_S$, and so we obtain that $\Om$ is $\mscr{T}_S$--tucked. (The important point here is again that orthogonals of aspherical elements of $\mscr{C}$ are always the same, regardless of which generating set in $\mscr{T}_S$ is used to compute them, since they correspond to centralisers.)
\end{proof}

Before continuing, we record the following observation. We say that $W$ \emph{splits} over a subgroup $H$ relative to a family $\mscr{C}$ if there exists a splitting of $W$ relative to $\mscr{C}$ that has $H$ as an edge-stabiliser.

\begin{rmk}\label{rmk:splittings_over_cliques}
    The following two conditions are equivalent for the Coxeter system $(W,S)$:
    \begin{enumerate}
        \item the graph $\wh S_{\mscr{C}}$ is not disconnected by any (possibly empty) clique;
        \item $W$ does not split relative to $\mscr{C}$ over any subgroup
        contained in $\langle C\rangle$ for some $C\in\mscr{C}$.
    \end{enumerate}
    The implication $(2)\Ra (1)$ is clear in view of \Cref{rmk:amalgams_rel_C}, while $(1)\Ra (2)$ is an immediate consequence of \Cref{lem:MT}. 
\end{rmk}

\begin{lem}\label{lem:no_splitting_over_clique}
    If $S$ and $R$ are not conjugate, there exists a $\mscr{T}_S$--tucked, $\mscr{T}_S\cup\{R\}$--parabolic subset $\Om\sq S$ such that $\Om\not\in\mscr{C}$ and one of following holds:
    \begin{enumerate}
        \item either $\Om$ is a $\mscr{C}$--shrub; 
        \item or the graph $\wh\Om_{\mscr{C}}$ is not disconnected by any (possibly empty) clique.
    \end{enumerate}
\end{lem}
\begin{proof}
    If the graph $\wh S_{\mscr{C}}$ is not disconnected by any clique, we can simply take $\Om:=S$. Otherwise, let $C\sq S$ be a (possibly empty) least-cardinality
    subset such that $C\in\mscr{C}$ and $\wh S_{\mscr{C}}\setminus C$ is disconnected. Let $\mf{T}$ be the family of strongly reduced, $S$--visual splittings relative to $\mscr{C}$ with all edge-stabilisers conjugate to $\langle C\rangle$. As usual, we order $\mf{T}$ by strong domination: we write $\mc{T}_1\leq\mc{T}_2$ if $\mc{T}_2$ strongly dominates $\mc{T}_1$. Since $\mf{T}$ is finite by \Cref{rmk:finitely_many_visual}, it has a maximal element $\mc{T}_{\max}$.

    We claim that the vertex-stabilisers of $W\acts\mc{T}_{\max}$ are all $\mscr{T}_S\cup\{R\}$--parabolic. For this, consider a vertex $x\in\mc{T}_{\max}$ and its stabiliser $W_x$. In view of \Cref{lem:universally_elliptic->parabolic}, it suffices to show that $W_x$ is elliptic in all splittings $W\acts\mc{T}$ relative to $\mscr{C}$ that strongly dominate $\mc{T}_{\max}$ (possibly with $\mc{T}\not\in\mf{T}$). Thus, suppose for the sake of contradiction that $W_x$ is non-elliptic in such a splitting $\mc{T}$. We can then use this to construct a (possibly non-visual) splitting $W\acts\mc{T}'$ relative to $\mscr{C}$ such that $\mc{T}'\geq\mc{T}_{\max}$ and $W_x$ is non-elliptic in $\mc{T}'$: more precisely, $\mc{T}'$ is obtained from $\mc{T}_{\max}$ by blowing up each point in the $W$--orbit of $x$ to a copy of the minimal subtree of $W_x\acts\mc{T}$; see \cite[Proposition~2.2]{GL-JSJ} for details. 
    Now, \Cref{lem:MT} yields an $S$--visual splitting $\mc{T}''\geq\mc{T}'\geq\mc{T}_{\max}$ relative to $\mscr{C}$, and minimality of $C$ implies that all edge-stabilisers of $\mc{T}''$ are conjugate to $\langle C\rangle$. Therefore, we have $\mc{T}''\in\mf{T}$. At the same time, $W_x$ is non-elliptic in $\mc{T}'$ and hence non-elliptic in $\mc{T}''$, which shows that $\mc{T}''\not\leq\mc{T}_{\max}$. This violates maximality of $\mc{T}_{\max}$ within $\mf{T}$, proving our claim.
    
    Now, let $\mf{F}=(\mc{F},\{S_v\}_v)$ be a visual decomposition that corresponds to the $S$--visual splitting $\mc{T}_{\max}$ via \Cref{lem:visual_equivalence}. We claim that the sets $S_v$ are all $S$--tucked. For this, consider a vertex $x\in\mc{F}$, and consider a subset $J\sq S$ such that $J\in\mscr{C}$ and such that $S_x$ intersects at least two connected components of $\wh S_{\mscr{C}}\setminus (J\cup J^{\perp})$ (for this argument it is irrelevant whether $J$ is irreducible, or whether $J$ and $J\setminus S_x$ are spherical). The set $J$ fixes a vertex of $\mc{T}_{\max}$, and so there exists a vertex $y\in\mc{F}$ such that $J\sq S_y$; we pick $y$ so that it is closest to $x$ within the finite tree $\mc{F}$. If $y=x$, then $J\sq S_x$ and we are done. Otherwise, let $e\sq\mc{F}$ be the edge incident to $x$ in the direction of $y$, and observe that we have $S_x\cap (J\cup J^{\perp})\sq S_e$. Moreover, the set $S_e\sq S$ is conjugate to $C$ because $\mc{T}_{\max}\in\mf{T}$. Since $S_x$ intersects two components of $\wh S_{\mscr{C}}\setminus (J\cup J^{\perp})$, it follows that the set $S_x\cap (J\cup J^{\perp})$ disconnects the graph $(\wh S_x)_{\mscr{C}}$. This can be used to refine the visual decomposition $\mf{F}$ into a strongly reduced visual decomposition $\tilde{\mf{F}}=(\tilde{\mc{F}},\{\tilde S_v\}_v)$ relative to $\mscr{C}$, where the vertex $x\in\mc{F}$ gets blown up to an edge $f\sq\tilde{\mc{F}}$ with $\tilde S_f=S_x\cap (J\cup J^{\perp})$.
    If the set $S_x\cap (J\cup J^{\perp})$ is properly contained in $S_e$, this violates minimality of the set $C$. If instead $S_x\cap (J\cup J^{\perp})$ is equal to $S_e$, then this violates maximality of the tree $\mc{T}_{\max}$. Either way, this proves our second claim.

    In fact, the argument in the previous paragraph also shows that the sets $S_v$ are $\mscr{T}_S$--tucked. Indeed, consider a generating set $S'\in\mscr{T}_S$ and let $\mc{T}_{\max}'$ be maximal under strong domination among strongly reduced, $S'$--visual splittings relative to $\mscr{C}$ with all edge-stabilisers conjugate to $\langle C\rangle$. A double application of \Cref{lem:MT} shows that $\mc{T}_{\max}$ and $\mc{T}_{\max}'$ dominate each other and thus have the same vertex-stabilisers (here it is convenient to keep in mind the characterisation in \Cref{rmk:splittings_over_cliques}(2), which is independent of the choice of a generating set). We can then repeat the argument in the previous paragraph to obtain that the sets $S_v'\sq S'$ with $\langle S_v'\rangle=\langle S_v\rangle$ are $S'$--tucked.

    Summing up, we have shown that all vertex-stabilisers of the splitting $W\acts\mc{T}_{\max}$ are $\mscr{T}_S\cup\{R\}$--parabolic and (up to conjugacy) generated by subsets of $S$ that are $\mscr{T}_S$--tucked. Moreover, all edge-stabilisers of $\mc{T}_{\max}$ are conjugate to each other.

    Now, if all vertex-stabilisers of $\mc{T}_{\max}$ are generated by sets in $\mscr{C}$, then $S$ is a $\mscr{C}$--shrub (by \Cref{lem:visual_equivalence}) and we can take $\Om:=S$. Otherwise, we pick a subset $U\sq S$ such that $U$ generates a vertex-stabiliser of $\mc{T}_{\max}$ and $U\not\in\mscr{C}$. We replace $S$ with $U$, and repeat the whole argument for $U$. After finitely many iterations of this, keeping \Cref{rmk:tucked_in_tucked}(2) in mind, the lemma is proven.
\end{proof}

We are finally ready to prove \Cref{thm:step_one}, of which we repeat the statement for convenience.

\begin{thm}\label{thm:step_one_repeated}
    If $S$ and $R$ are not conjugate, then there exists a subset $\Om\sq S$ such that:
    \begin{enumerate}
        \item $\Om$ is $\mscr{T}_S\cup\{R\}$--parabolic and $\Om\not\in\mscr{C}$;
        \item $\Om$ is a $\mscr{C}$--shrub;
        \item $\Om$ is $\mscr{T}_S$--tucked.
    \end{enumerate}
\end{thm}
\begin{proof}
    To begin with, let $\Om_1\sq S$ be a subset provided by \Cref{lem:constricted}. That is, $\Om_1$ is $\mscr{T}_S\cup\{R\}$--parabolic with $\Om_1\not\in\mscr{C}$, and $\Om_1$ is both constricted and $\mscr{T}_S$--tucked. We then apply \Cref{lem:no_splitting_over_clique} to the Coxeter group $\langle\Om_1\rangle$ to produce another constricted, $\mscr{T}_S$--tucked, $\mscr{T}_S\cup\{R\}$--parabolic subset $\Om_2\sq\Om_1$ with $\Om_2\not\in\mscr{C}$ (recall \Cref{rmk:tucked_in_tucked}). If $\Om_2$ is a $\mscr{C}$--shrub, then the proposition is proven.

    Otherwise, \Cref{lem:no_splitting_over_clique} guarantees that the graph $(\wh\Om_2)_{\mscr{C}}$ is not disconnected by any clique. Let $\Om_3\sq\Om_2$ be a minimal $\mscr{T}_S\cup\{R\}$--parabolic subset with $\Om_3\not\in\mscr{C}$. \Cref{prop:minimal->shrub} shows that $\Om_3$ is a $\mscr{C}$--shrub, so we only need to check that $\Om_3$ is $\mscr{T}_S$--tucked.  We first show that $\Om_3$ is $S$--tucked. 
    
    For this, suppose that $J\sq S$ is a non-spherical irreducible subset with $J\in\mscr{C}$ and such that $\Om_3$ intersects at least two connected components of $\wh S_{\mscr{C}}\setminus (J\cup J^{\perp})$. Then $\Om_2$ intersects these two components as well, and so the fact that $\Om_2$ is tucked implies that either $J\setminus\Om_2\sq J\setminus\Om_3$ is non-spherical (and we are done), or there is an inclusion $J\sq\Om_2$. In fact, the latter cannot happen: if we had $J\sq\Om_2$, then the set $(J\cup J^{\perp})\cap\Om_2$ would span a clique in $(\wh\Om_2)_{\mscr{C}}$ because $\Om_2$ is constricted; this clique would then disconnect the graph $(\wh\Om_2)_{\mscr{C}}$, because $\Om_3\sq\Om_2$ intersects two components of $\wh S_{\mscr{C}}\setminus (J\cup J^{\perp})$, and this is a contradiction. This proves that $\Om_3$ is $S$--tucked.

    As usual, the same argument shows that $\Om_3$ is $\mscr{T}_S$--tucked (using \Cref{rmk:splittings_over_cliques}), so this concludes the proof of the theorem.
\end{proof}

\section{Inflexibility}\label{sect:inflexibility}

This is the first of three sections working towards a proof \Cref{thm:step_two}, which is the second and last step in the proof of \Cref{thmintro:main} sketched in \Cref{sect:strategy}. In this section, our goal is to prove that $\mscr{T}_S$--parabolic subsets of $S$ satisfy a combinatorial property, which we call \emph{inflexibility}.

Let $(W,S)$ be a Coxeter system. Let $\mscr{C}$ be a family of subsets of $W$ that is closed under passing to subsets and taking conjugates, and that contains all spherical subsets of $S$ of cardinality $\leq 2$. We can form the graph $\wh S_{\mscr{C}}$ as in \Cref{defn:hat_graph}. Let $\mscr{T}_S$ be the collection of all Coxeter generating sets of $W$ that are twist-equivalent to $S$ relative to $\mscr{C}$. (In fact, the generating sets differing from $S$ by a \emph{single} elementary twist will suffice for all purposes in this section.)

\begin{defn}
    A subset $\Om\sq S$ is \emph{$\mscr{C}$--inflexible} if, for every irreducible spherical subset $K\sq S$ for which $\Om$ intersects at least two connected components of the graph $\wh S_{\mscr{C}}\setminus (K\cup K^{\perp})$, we have $K\sq\Om$.
\end{defn}

The main result of this section is the following.

\begin{prop}\label{prop:inflexible}
    Suppose that \Cref{thmintro:main} holds for all proper $S$--parabolic subgroups of $W$. Then all $\mscr{T}_S$--parabolic subsets of $S$ are $\mscr{C}$--inflexible.
\end{prop}

\begin{ex}\label{ex:converse}
    Note that the converse of \Cref{prop:inflexible} fails: not all $\mscr{C}$--inflexible subsets of $S$ are $\mscr{T}_S$--parabolic. For instance, suppose that $S=\{a,b,c,d\}$ with $o(ab)=o(bc)=3$, $o(cd)=2$, and with all other products of distinct elements of $S$ having infinite order. Let $\mscr{C}$ be trivial (i.e.\ simply the family of $W$--conjugates of spherical subsets of $S$ of cardinality $\leq 2$). Setting $w:=bcb$, the set $S'=\{a,b,c,wdw\}$ lies in $\mscr{T}_S$. Now, the set $\Om:=\{a,b,c\}$ is $\mscr{C}$--inflexible as a subset of $S$, but it is no longer $\mscr{C}$--inflexible as a subset of $S'$. Moreover, we can twist $S'$ with respect to the spherical subset $\{wdw\}$ and produce $S'':=\{a,b,(wdw)c(wdw),wdw\}\in\mscr{T}_S$, so that $\Om$ is not $S''$--parabolic. Here the whole point is that some twists that cannot be performed on $S$ become available once we move to $S'$.
\end{ex}

\begin{rmk}\label{rmk:parabolic_vs_inflexible_explanation}
    The main step in the proof of \Cref{thm:step_two} in \Cref{sect:geometrisation} is a ``shortening theorem'' (\Cref{thm:shortening}) that replaces a $\mscr{C}$--inflexible $\mscr{C}$--shrub $\Om\sq S$ with a $\mscr{C}$--shrub $\Om'\sq S'$ that is closer to being $R$--geometric, for some $S'\in\mscr{T}_S$. If we did not know that the set $\Om$ is actually $\mscr{T}_S$--parabolic, rather than simply $\mscr{C}$--inflexible, then the set $\Om'$ might no longer be $\mscr{C}$--inflexible (\Cref{ex:converse}), and so we would not be able to continue the shortening process. This explains why $\mscr{T}_S$--parabolicity is so important, even though in practice we will mostly work with $\mscr{C}$--inflexibility in \Cref{sect:geometrisation}.
\end{rmk}

We now start working towards the proof of \Cref{prop:inflexible}. We begin with three general results. Denote by $|g|$ the length of the shortest words in $S$ representing an element $g\in W$. A word representing $g$ is \emph{reduced} if it has length $|g|$. All reduced words representing $g$ contain the same letters, see e.g.\ \cite[Proposition~4.1.1]{Davis}; we call this set of letters the \emph{support} of $g$, and denote it by $\supp(g)\sq S$. The letters in $\supp(g)$ must appear in all words representing $g$.

\begin{lem}\label{lem:Bruhat}
    Let $K\sq S$ be a spherical subset. For every reduced word $s_1\dots s_m$ representing the longest element $w_K\in\langle K\rangle$, and for every element $h\in\langle K\rangle$, there exist indices $1\leq i_1<\dots<i_{\ell}\leq m$ such that $s_{i_1}\dots s_{i_{\ell}}$ is a reduced word representing $h$.
\end{lem}
\begin{proof}
    This is due to $w_K$ being the maximum of the Bruhat order on $\langle K\rangle$; see for instance \cite[Section~5.10]{Humphreys} or \cite[Theorem~2.2.2]{Bjoerner-Brenti}.
\end{proof}

The \emph{Dynkin diagram} of the Coxeter system $(W,S)$ is the graph having $S$ as its vertex set and edges corresponding to non-commuting pairs of elements. When we speak of the \emph{label} of an edge of a Dynkin diagram (or of a presentation graph), we refer to the order of the product of the two corresponding elements of $S$.

\begin{lem}\label{lem:w_K=gg'}
    Let $K\sq S$ be irreducible and spherical. 
    If we have $w_K=gg'$ for two elements $g,g'\in W$, then we have either $K\sq\supp(g)$ or $K\sq\supp(g')$.
\end{lem}
\begin{proof}
    Suppose for the sake of contradiction that there exist two elements $t,t'\in K$ such that $t\not\in\supp(g)$ and $t'\not\in\supp(g')$. Up to replacing $g$ and $g'$ within $\langle\supp(g)\rangle$ and $\langle\supp(g')\rangle$, respectively, \cite[Lemma~4.3.1]{Davis} allows us to assume that $|w_K|=|g|+|g'|$. Thus, we obtain a reduced expression $w_K=s_1\dots s_m$ with the property that there exists an index $1\leq\ell<m$ such that $s_i\neq t$ for all indices $1\leq i\leq\ell$ and $s_i\neq t'$ for all $\ell<i\leq m$. We will show that this is not possible.

    Since $\supp(w_K)=w_K$, we have $t\neq t'$. Let $\Pi\sq K$ be the smallest irreducible subset containing $\{t,t'\}$. The Dynkin diagram of $\Pi$ is a path having $t$ and $t'$ as its endpoints; in particular, $\Pi$ is either dihedral or of one of the types $A_n,B_n,F_4,H_3,H_4$. Call $t_1,\dots,t_k$ the elements of $\Pi$ in the order in which they appear along the path, with $t_1:=t$ and $t_k:=t'$. Now, consider the element $h=t_1t_2\dots t_{k-1}t_kt_{k-1}\dots t_2t_1$. A straightforward check using Tits' solution to the word problem (see \cite[Theorem~3.4.2]{Davis}) shows that all reduced words in $\Pi$ representing $h$ contain both occurrences of $t_1$ preceding occurrences of $t_k$, and occurrences of $t_k$ preceding occurrences of $t_1$ (not necessarily consecutively).
    Invoking \Cref{lem:Bruhat}, we then contradict the existence of the expression $w_K=s_1\dots s_m$ described above, and this concludes the proof.
\end{proof}

\begin{lem}\label{lem:simpler_Deodhar}
    Consider an irreducible subset $T\sq S$. 
    If an element $g\in W$ satisfies $gTg^{-1}\sq S$, then either $T\sq\supp(g)$ or $g\in\langle T^{\perp}\rangle$.
\end{lem}
\begin{proof}
    Consider an element $g\in W$ with $gTg^{-1}\sq S$. By a result of Deodhar \cite{Deodhar} (see \cite[Theorem~4.10.6]{Davis}), there is a finite sequence $T=:T_0,\dots,T_k:=gTg^{-1}$ of irreducible subsets of $S$ and a writing $g=\vartheta_k\dots\vartheta_1\om$ such that $\vartheta_iT_{i-1}\vartheta_i^{-1}=T_i$ for each $1\leq i\leq k$ and:
    \begin{itemize}
        \item $\om$ is either trivial, or possibly the longest element of $\langle T\rangle$ if $T$ is spherical;
        \item either $\vartheta_i\in T^{\perp}$, or we have $\vartheta_i=w_{U_i}w_{T_{i-1}}$ for an irreducible spherical subset $U_i\sq S$ with $T_{i-1}\sq U_i$ and $|U_i\setminus T_{i-1}|=1$;
        \item plugging reduced words representing $\om$ and the $\vartheta_i$ into the writing $\vartheta_k\dots\vartheta_1\om$, we obtain a reduced word representing $g$.
    \end{itemize}
    Now, we can suppose that $\om=1$, as otherwise we have $T\sq\supp(\om)\sq\supp(g)$. If all $\vartheta_i$ lie in $T^{\perp}$, it is clear that $g\in\langle T^{\perp}\rangle$. Otherwise, let $i$ be the smallest index such that $\vartheta_i\not\in T^{\perp}$. Then \Cref{lem:w_K=gg'} implies that we have $U_i\sq\supp(\vartheta_i)\sq\supp(g)$. Since $T=T_{i-1}\sq U_i$, we conclude that $T\sq\supp(g)$ as desired.
\end{proof}

The next lemma is the main technical step in the proof of \Cref{prop:inflexible}. Recall that, for all spherical subsets $K\sq S$, we denote by $w_K$ the longest element of $\langle K\rangle$.

\begin{lem}\label{lem:hard}
    Suppose that $S=\{a\}\sqcup K\sqcup\{b\}$, where $K$ is irreducible spherical with $K^{\perp}=\emptyset$,
    and where $ab$ has infinite order. Suppose that \Cref{thmintro:main} holds for all proper $S$--parabolic subgroups of $W$. Then, for all proper subsets $K_0\subsetneq K$, the subgroup $\langle a,K_0,w_Kbw_K\rangle$ is not $S$--parabolic.
\end{lem}
\begin{proof}  
    Set $U':=\{a,w_Kbw_K\}\cup K_0$ and $P:=\langle U'\rangle$. Suppose for the sake of contradiction that $P$ is $S$--parabolic. Thus, there exist an element $g\in W$ and a subset $U\sq S$ such that $P=g\langle U\rangle g^{-1}$.

    Let $W\acts\mc{T}$ be the $S$--visual one-edge splitting having an edge $e\sq \mc{T}$ with vertex-stabilisers generated by $\{a\}\cup K$ and $K\cup\{b\}$, respectively. Since neither $a$ nor $w_Kbw_K$ lies in $\langle K\rangle$, neither of these elements fixes the edge $e$, and so their product is loxodromic in $\mc{T}$. Hence $P$ is non-elliptic in $\mc{T}$, and the same is then true of $\langle U\rangle$. In particular, $a$ and $b$ must both lie in the set $U$, and so we have $U=\{a,b\}\cup L_0$ for some $L_0\sq K$.
    
    Now, note that the set $S':=\{a,w_Kbw_K\}\cup K$ is a Coxeter generating set of $W$ differing from $S$ by an elementary twist; in particular, $S$ and $S'$ are angle-compatible. Being a subset of $S'$, the set $U'$ is a Coxeter generating set of $P$ and it is angle-compatible with $gUg^{-1}$. Since $P$ is a proper subgroup of $W$, our hypotheses imply that \Cref{thmintro:main} holds for $P$, and so $U$ and $U'$ are twist-equivalent within $P$, relative to the family of $\{U,U'\}$--compatible subsets. Since the sets $\{a\}\cup K_0$ and $K_0\cup\{w_Kbw_K\}$ are $\{S,S'\}$--compatible, 
    they are also $\{U,U'\}$--compatible, and so it suffices to consider elementary twists relative to these two sets. This guarantees that the only twists that we might need to perform on $U'$ are those by the longest elements of the irreducible factors of $\langle K_0\rangle$. 
    In conclusion, up to left-multiplying $g$ by an element of $P$, we obtain that 
    \begin{equation}\label{eq:gUg-1}
        g\big(\{a,b\}\cup L_0\big)g^{-1}=\{a,w_0w_Kbw_Kw_0\}\cup K_0 ,
    \end{equation}
    where $w_0$ is the product of the longest elements of some of the irreducible factors of $\langle K_0\rangle$.

    Considering again the splitting $W\acts\mc{T}$ and its edge $e\sq\mc{T}$, observe that $e$ is the unique edge intersecting the fixed sets of both $a$ and $b$, and also the unique edge intersecting those of $a$ and $w_0w_Kbw_Kw_0$. Thus, \Cref{eq:gUg-1} implies that $ge=e$ 
    and, since there are no edge inversions, it follows that $g$ fixes each of the vertices of $e$. This implies that we have all of the following:
    \begin{align}\label{eq:element_g}
        g&\in\langle K\rangle, & g^{-1}K_0g=L_0&\sq K, & g^{-1}ag&=a, & g^{-1}(w_0w_Kbw_Kw_0)g&=b .
    \end{align}
    The rest of the proof is devoted to ruling out the existence of such an element $g$. 

    Let $\langle K_1\rangle,\dots,\langle K_k\rangle$ be the irreducible factors of $\langle K_0\rangle$ whose longest elements $w_1,\dots,w_k$ appear in the element $w_0$; thus, the $w_i$ pairwise commute and we have $w_0=w_1\dots w_k$. Note that it is possible that $k=0$ and $w_0=1$. We now distinguish two cases.

    \smallskip
    {\bf Case~1.} \emph{For each $1\leq i\leq k$, either $K_i$ commutes with $a$, or $K_i$ commutes with $w_Kbw_K$.} 

    \smallskip\noindent
    Let $h$ be the element obtained by left-multiplying $g$ by the elements $w_i$ for which $K_i$ commutes with $a$. The remaining $w_i$ commute with $w_Kbw_K$ and so \Cref{eq:element_g} yields:
    \begin{align*}
        h&\in\langle K\rangle, & h^{-1}ah&=a, & h^{-1}(w_Kbw_K)h&=b .
    \end{align*}
    Applying \Cref{lem:simpler_Deodhar} with $T=\{a\}$ and $T=\{b\}$, we obtain that $h\in\langle K\cap a^{\perp}\rangle$ and $w_Kh\in\langle K\cap b^{\perp}\rangle$. In particular, we have $w_K\in\langle b^{\perp}\rangle\cdot\langle a^{\perp}\rangle$, and $K$ is not contained in $a^{\perp}$ nor $b^{\perp}$ by the hypothesis that $K^{\perp}=\emptyset$. This contradicts \Cref{lem:w_K=gg'}.

    \smallskip
    {\bf Case~2.} \emph{There exists an index $1\leq i\leq k$ such $K_i$ does not commute with $a$ or $w_Kbw_K$.} 
    
    \smallskip\noindent
    Up to permuting the $K_i$, we can assume that it is $K_1$ that commutes with neither $a$ nor $w_Kbw_K$. In particular, the set $\{a\}\cup K_1$ is irreducible and there exists an element $\beta\in K_1\setminus (w_Kb^{\perp}w_K)$. Since $w_KKw_K=K$, we also have $w_K\beta w_K\in K\setminus b^{\perp}$. Now, applying \Cref{lem:simpler_Deodhar} with $T=\{a\}\cup K_1$ and $T=\{b\}$, the equalities in \Cref{eq:element_g} yield:
    \begin{align*}
        \supp(g)&\sq \big(\{a\}\cup K_1\big)^{\perp}\sq K_1^{\perp}, & \supp(w_Kw_0g)&\sq K\cap b^{\perp}\sq K\setminus\{w_K\beta w_K\}.
    \end{align*}
    In particular, we have
    \[ \supp(gw_0)\sq\supp(g)\cup\supp(w_0)\sq K_1^{\perp}\cup K_0=K_1^{\perp}\cup K_1 ,\]
    and irreducibility of $K$ implies that $K\not\sq K_1^{\perp}\cup K_1$. Finally, writing $w_K=(w_Kw_0g)(gw_0)$, we obtain a contradiction to \Cref{lem:w_K=gg'}. This concludes the proof.
\end{proof}

We can now easily extend the previous lemma to more general decompositions of generating sets. Recall that we way that two distinct elements $s,s'\in S$ are \emph{adjacent} if $ss'$ has finite order (i.e.\ they are adjacent in the presentation graph of $(W,S)$).

\begin{lem}\label{lem:not_parabolic}
    Let $S=X\sqcup(K\cup K^{\perp})\sqcup Y$ be a partition where $K\sq S$ is irreducible spherical and no element of $X$ is adjacent to an element of $Y$. Suppose that \Cref{thmintro:main} holds for all proper $S$--parabolic subgroups of $W$. Let $\Om\sq S$ be a subset that intersects both $X$ and $Y$. If the set $(\Om\setminus Y)\cup w_K(\Om\cap Y)w_K$ generates an $S$--parabolic subgroup of $W$, then $K\sq\Om$.
\end{lem}
\begin{proof}
    Let $P$ be the subgroup generated by the set $(\Om\setminus Y)\cup w_K(\Om\cap Y)w_K$, and define $K_0:=\Om\cap K$. Suppose that $P$ is $S$--parabolic and $K_0\subsetneq K$. We will reach a contradiction by reducing to the situation in \Cref{lem:hard}.

    By hypothesis, there exist $a\in\Om\cap X$ and $b\in\Om\cap Y$. Considering $U:=\{a,b\}\cup K$, the subgroup $P\cap\langle U\rangle$ is $U$--parabolic (see e.g.\ \cite[Lemma~5.3.6]{Davis}). Thus, in order to obtain a contradiction from \Cref{lem:hard}, it suffices to check that $P\cap\langle U\rangle$ is generated by the set $\{a,w_Kbw_K\}\cup K_0$.

    Let $W\acts\mc{T}$ be the $S$--visual one-edge splitting having an edge $[x,y]\sq\mc{T}$ with vertex-stabilisers generated by $X\cup K\cup K^{\perp}$ and $K\cup K^{\perp}\cup Y$. Let $Q_1$ and $Q_2$ be the subgroups of $W$ generated by $\Om\setminus Y$ and $(\Om\cap K)\cup (\Om\cap K^{\perp})\cup w_K(\Om\cap Y)w_K$, respectively. We have $P=\langle Q_1,Q_2\rangle$. Observe that the $Q_1$--stabiliser of the edge $[x,y]$ coincides with the $Q_2$--stabiliser of $[x,y]$, and so these stabilisers also equal the intersection $Q_1\cap Q_2$. As a consequence, a standard ping-pong argument shows that $Q_1$ is precisely the $P$--stabiliser of $x$, and similarly $Q_2$ is the $P$--stabiliser of $y$.

    Now, consider the intersection $P\cap\langle U\rangle$. This group contains the element $a$, which fixes $x$ and not $y$, and the element $w_Kbw_K$, which fixes $y$ and not $x$. As a consequence, the group $P\cap\langle U\rangle$ is non-elliptic in $\mc{T}$ and the edge $[x,y]$ is a fundamental domain for the action of $P\cap\langle U\rangle$ on its minimal subtree. It follows that $P\cap\langle U\rangle$ is generated by its subgroups fixing $x$ and $y$ and, by the previous paragraph, these are $Q_1\cap\langle U\rangle=\langle\{a\}\cup K_0\rangle$ and $Q_2\cap\langle U\rangle=\langle K_0\cup\{w_Kbw_K\}\rangle$. In conclusion, the intersection $P\cap\langle U\rangle$ is generated by $\{a,w_Kbw_K\}\cup K_0$ as desired, concluding the proof.
\end{proof}

We are finally ready to prove \Cref{prop:inflexible}.

\begin{proof}[Proof of \Cref{prop:inflexible}]
    Let $\Om\sq S$ be a $\mscr{T}_S$--parabolic subset. Consider an irreducible spherical subset $K\sq S$ with $K\not\sq \Om$. Suppose for the sake of contradiction that $\Om$ intersects two distinct components of the graph $\wh S_{\mscr{C}}\setminus (K\cup K^{\perp})$.

    The previous sentence means that there exists a partition $S=X\sqcup (K\cup K^{\perp})\sqcup Y$ such that $\Om$ intersects both $X$ and $Y$, no element of $X$ is adjacent to an element of $Y$, and each subset in $\mscr{C}$ is disjoint from either $X$ or $Y$. Consider $S':=(S\setminus Y)\cup w_KYw_K$ and observe that $S'\in\mscr{T}_S$. In particular, the subgroup $\langle \Om\rangle$ is $S'$--parabolic by our hypotheses. Together with the fact that $K\not\sq \Om$, this contradicts \Cref{lem:not_parabolic} applied to the generating set $S'$, proving the proposition. 
\end{proof}

\section{Markings and hierarchies}\label{sect:markings}

Throughout, let $S,R\sq W$ be angle-compatible Coxeter generating sets and let $\mscr{C}$ be the family of $\{S,R\}$--compatible subsets of $W$. Let $\wh S_{\mscr{C}}$ be the graph from \Cref{defn:hat_graph}.

Bases and markings are certain pairs $(s,w)$ and $\mu=((s,w),m)$ with $s,m\in S$ and $w\in W$, which were introduced by Caprace and Przytycki in \cite{CP10}; we recall their definition and main properties in \Cref{sub:markings} below. Each (admissible) marking $\mu$ defines a halfspace in the Davis complex $\A_R$, and we write $\mu\equiv_R\mu'$ when two markings determine the same halfspace there.

The goal of this section is to prove \Cref{prop:markings_new} below, which yields conditions under which two markings define the same halfspace. In order to state this, we need some more terminology. A subset $U\sq S$ is \emph{tree--$2$--spherical} if it is irreducible\footnote{This is slightly different from \cite[Definition~3.4]{CP10}, where irreducibility is not assumed.} and its Dynkin diagram is a tree. For each subset $U\sq S$, we write:
\begin{align*}
        \mc{I}_U&:=\{ K\sq U\mid \text{$K$ is irreducible, spherical and nonempty} \} \\
        \mc{J}_U&:=\{ K\sq U\mid \text{$K$ is tree--$2$--spherical and nonempty} \} .
\end{align*}  

\begin{defn}\label{defn:inseparable}
    Let $(K,\gamma)$ be a pair where $K\in\mc{J}_S$ and $\gamma$ is a path in $\wh S_{\mscr{C}}\setminus (K\cup K^{\perp})$ with endpoints labelled $a$ and $b$. We say that the pair $(K,\gamma)$ is:
    \begin{itemize}
        \item \emph{$J$--unseparated}, for some $J\sq S$, if $\{a,b\}$ meets at most one component of $\wh S_{\mscr{C}}\setminus (J\cup J^{\perp})$;
        \item \emph{weakly inseparable} if it is $I$--unseparated for every $I\in\mc{I}_S$ with $K\sq I$ and $I\cap\gamma\neq\emptyset$;
        \item \emph{strongly inseparable} if it is $J$--unseparated for every $J\in\mc{J}_S$ with $K\sq J$ and $J\cap\gamma\neq\emptyset$ such that $J\setminus K$ is spherical.
    \end{itemize}
\end{defn}

\begin{prop}\label{prop:markings_new}
    Let $(s,w)$ be a 
    base with support $\Sigma$, and let $K\sq \Sigma$ be an irreducible subset. Let $\gamma\sq\wh S_{\mscr{C}}\setminus (K\cup K^{\perp})$ be a geodesic with endpoints $a$ and $b$ such that $\{a,b\}\not\in\mscr{C}$. Suppose that the pair $(K,\gamma)$ is strongly inseparable and consider one of the following three situations:
    \begin{enumerate}
        \item $\mu:=((s,w),a)$ and $\mu':=((s,w),b)$;
        \item $\mu:=((s,wa),b)$ and $\mu':=((s,w),b)$, where $(s,wa)$ is a base with\footnote{The requirement that $|wa|>|w|$ is needed only in case $a\in \Sigma$. Here we need to allow this eventuality because, in the terminology \cite{CP10}, we will often have to work with \emph{non-simple} bases; see e.g.\ the proof of \Cref{prop:atoms_2}(1).} $|wa|>|w|$;
        \item $\mu:=((s,wa),b)$ and $\mu':=((s,wb),a)$, where $a,b\not\in\Sigma$.
    \end{enumerate}
    Then, supposing in each of three cases that $\mu$ and $\mu'$ are complete markings, we have $\mu\equiv_R\mu'$.
\end{prop}

There is also a version of \Cref{prop:markings_new} for the case when the pair $(K,\gamma)$ is just \emph{weakly} inseparable (see \Cref{prop:markings_general} below), and this suffices for instance when dealing with Coxeter groups of FC-type. Despite this, we will only use the above version of the proposition in the proof of \Cref{thm:step_two} in \Cref{sect:geometrisation}. As the reader may already suspect, the concept of tucked sets (\Cref{defn:tucked}) was designed precisely to enable the construction of \emph{strongly} inseparable pairs.

\Cref{prop:markings_new} is more general than the results of \cite[Sections~6--8]{CP10} in three ways: first, inseparability of pairs $(K,\gamma)$ is typically much weaker than twist-rigidity of $(W,S)$; second, we work with the graph $\wh S_{\mscr{C}}$ instead of the presentation graph of $(W,S)$; third, we allow the set $K$ to be \emph{properly} contained in the support $\Sigma$ of the base (cf.\ \Cref{prop:Sect7_2}). Despite these differences, most proofs in this section are essentially identical to those in \cite{CP10}.

\subsection{Bases and markings}\label{sub:markings}

Let $S$, $R$ and $\mscr{C}$ be as above. We first focus only on $S$.

We denote by $1_S$ the base vertex of the Davis complex $\A_S$. For each reflection $\rho\in S^W$, denote by $\mc{Y}_{\rho}$ the $\rho$--fixed wall in the Davis complex $\A_S$. An edge of $\A_S$ with endpoints in different halfspaces associated to $\mc{Y}_{\rho}$ is \emph{dual} to $\mc{Y}_{\rho}$. Given a wall $\mc{Y}\sq\A_S$ and a vertex $p\in\A_S$, the distance $d(p,\mc{Y})$ is defined as the shortest distance between $p$ and an edge dual to $\mc{Y}$, working with the path metric on the $1$--skeleton of $\A_S$.

A \emph{base} is a pair $(s,w)\in S\x W$ such that $d(w1_S,\mc{Y}_s)=|w|$ and no wall separates $w1_S$ from $\mc{Y}_s$ (recall that $|w|=d(1_S,w1_S)$ is the length of the reduced words representing $w$). The \emph{support} $\supp(s,w)$ is the smallest subset $K\sq S$ such that $s\in K$ and $w\in\langle K\rangle$. 

\begin{rmk}\label{rmk:tree-2-spherical_support}
    If $(s,w)$ is a base, then the set $\supp(s,w)$ is tree--$2$--spherical; see \cite[Lemma~3.5]{CP10}. Conversely, if $(s,w)$ is a base and $x\in S\setminus\supp(s,w)$ is an element such that $\supp(s,w)\cup\{x\}$ is tree--$2$--spherical, 
    then $(s,wx)$ is a base; this can be easily shown using \cite[Remark~3.2]{CP10}.
\end{rmk}

\begin{rmk}\label{rmk:PW}
    There exists an integer $N=N(W,S)$ such that $|w|\leq N$ for every base $(s,w)$ with $s\in S$ and $w\in W$. This follows from the Parallel Wall Theorem \cite[Theorem~2.8]{Brink-Howlett-PW}.
\end{rmk}

A \emph{marking} is a triple $\mu=((s,w),m)$ where $(s,w)$ is a base and $m\in S\setminus\supp(s,w)^{\perp}$. The element $s$ is the \emph{core} of $\mu$, and $m$ is the \emph{marker}. The marking $\mu$ is \emph{complete} if $\mc{Y}_s\cap w\mc{Y}_m=\emptyset$; equivalently, the subgroup $\langle s,wmw^{-1}\rangle$ is infinite. The marking is \emph{semicomplete}\footnote{More generally, we could call ``semicomplete'' all markings for which the set $\supp(s,w)\cup\{m\}$ is non-spherical and lies in $\mscr{C}$, but this would not yield any significant benefits.}
if the set $\supp(s,w)\cup\{m\}$ is $2$--spherical and non-spherical. Finally, the marking is \emph{admissible} if $\supp(s,w)\cup\{m\}$ is non-spherical. 

\begin{rmk}\label{rmk:complete_or_base}
    If $\mu=((s,w),m)$ is a marking and $d(wm1_S,\mc{Y}_s)>d(w1_S,\mc{Y}_s)$, then either $\mu$ is complete or $(s,wm)$ is a base (depending on whether the wall $w\mc{Y}_m$ is disjoint from $\mc{Y}_s$ or not).
\end{rmk}

\begin{rmk}\label{rmk:complete_or_semicomplete}
     Remarks~\ref{rmk:tree-2-spherical_support} and~\ref{rmk:complete_or_base} show that, for an arbitrary marking $\mu=((s,w),m)$, either $\mu$ is complete or the union $\supp(s,w)\cup\{m\}$ is tree--$2$--spherical. In particular, a marking is admissible if and only if it is either complete or semicomplete. 
    
    Note that a marking can be both complete and semicomplete at the same time and, despite the name, neither of these two properties implies the other.
\end{rmk}

Now, we also consider the Coxeter generating set $R$. Since $R$ is angle-compatible with $S$, the sets of reflections $S^W$ and $R^W$ coincide. For each reflection $\rho$, we denote by $\mc{W}_{\rho}$ the wall fixed by $\rho$ in the Davis complex $\A_R$ (so as to distinguish it from the wall $\mc{Y}_{\rho}\sq\A_S$). Every admissible marking $\mu=((s,w),m)$ determines a choice of a halfspace $\mc{H}^{\mu}\sq\A_R$ bounded by the wall $\mc{W}_s$:
\begin{itemize}
    \item If $\mu$ is complete, we define $\mc{H}^{\mu}$ as the side of $\mc{W}_s$ that contains the wall $w\mc{W}_m$.
    \item If $\mu$ is semicomplete, then there is a unique $R$--geometric choice of halfspaces for the walls of $\A_R$ fixed by the elements of $\supp(s,w)\cup\{m\}$ (since this set is non-spherical, irreducible and lies in $\mscr{C}$). We thus define $\mc{H}^{\mu}$ as the side of $\mc{W}_s$ picked by this choice.
\end{itemize}
Given two admissible markings $\mu,\mu'$ with the same core, we write $\mu\equiv_R\mu'$ if we have $\mc{H}^{\mu}=\mc{H}^{\mu'}$.

Throughout the article, markings will always have core and marker in the Coxeter generating set denoted by $S$, and they will always define halfspaces $\mc{H}^{\mu}$ in the Davis complex associated to the generating set denoted by $R$ (this is the same convention as in \cite{CP10}).

\subsection{Generalised moves}\label{sub:moves}

As in \cite[Section~4]{CP10}, we can modify a marking by certain ``moves''. We now discuss a version of these moves where adjacency in the presentation graph of $(W,S)$ is replaced by adjacency in the graph $\wh S_{\mscr{C}}$. 

Let $\mu=((s,w),m)$ and $\mu'=((s,w'),m')$ be admissible markings with the same core. We say that $\mu$ and $\mu'$ are related by move:
\begin{enumerate}
    \item[(N1)$_{\mscr{C}}$] if $w=w'$ and $\{m,m'\}\in\mscr{C}$;
    \item[(M1)$_{\mscr{C}}$] if both $\mu$ and $\mu'$ are complete and they differ by move (N1)$_{\mscr{C}}$;
    \item[(M2)$_{\mscr{C}}$] if $m=m'$ and $w'=wx$ for an element $x\in S$ with $\{m,x\}\in\mscr{C}$;
    \item[(M3)$_{\mscr{C}}$] if the union $\supp(s,w)\cup\supp(s,w')\cup\{m,m'\}$ lies in $\mscr{C}$.
\end{enumerate}
We write $\mu\equiv^{\mf{w}}\mu'$ if two admissible markings are connected by a finite sequence of the moves (N1)$_{\mscr{C}}$, (M2)$_{\mscr{C}}$ and (M3)$_{\mscr{C}}$. We instead write $\mu\equiv^{\mf{s}}\mu'$ if moves (M1)$_{\mscr{C}}$, (M2)$_{\mscr{C}}$ and (M3)$_{\mscr{C}}$ suffice. The letters $\mf{w}$ and $\mf{s}$ stand for \emph{weak} and \emph{strong} equivalence, respectively. The reason for this slightly awkward choice of notation will become clear in \Cref{sub:hierarchy} (see the statement of \Cref{prop:markings_general}).

Note that the definition of the equivalence relations $\equiv^{\mf{w}}$ and $\equiv^{\mf{s}}$ only involves the generating set $S$ and the family $\mscr{C}$, while $R$ plays no role. Still, there is a fundamental connection to the equivalence relation $\equiv_R$ from \Cref{sub:markings}, as explained in the next lemma (cf.\ \cite[Lemma~4.2]{CP10}).

\begin{lem}\label{lem:generalised_moves}
    If $\mu,\mu'$ are admissible markings with $\mu\equiv^{\mf{s}}\mu'$, then $\mu\equiv_R\mu'$.
\end{lem}
\begin{proof}
    It suffices to show that, when $\mu$ and $\mu'$ differ by one of the three moves (M1)$_{\mscr{C}}$--(M3)$_{\mscr{C}}$, the halfspaces $\mc{H}^{\mu}$ and $\mc{H}^{\mu'}$ coincide in $\A_R$.

    In the case of move (M1)$_{\mscr{C}}$, we have $\mu=((s,w),m)$ and $\mu'=((s,w),m')$ with $\{m,m'\}\in\mscr{C}$ and with both $\mu$ and $\mu'$ complete. In particular, the walls $w\mc{W}_m$ and $w\mc{W}_{m'}$ are both disjoint from $\mc{W}_s$. Since $\{m,m'\}$ lies in $\mscr{C}$, no wall of $\A_R$ separates $\mc{W}_m$ from $\mc{W}_{m'}$. Thus, $w\mc{W}_m$ and $w\mc{W}_{m'}$ are not separated by any walls, and so they must lie on the same side of $\mc{W}_s$.

    Regarding move (M2)$_{\mscr{C}}$, we have $\mu=((s,w),m)$ and $\mu'=((s,wx),m)$ with $\{m,x\}\in\mscr{C}$ and, without loss of generality, with $|wx|=|w|+1$. If $\mu$ is complete, then so is $\mu'$ (e.g.\ using \cite[Remark~3.2]{CP10}), and it suffices to show that $w\mc{W}_m$ and $wx\mc{W}_m$ are on the same side of $\mc{W}_s$. This again holds because $\mc{W}_m$ and $x\mc{W}_m$ are not separated by any walls of $\A_R$, since $\{m,x\}\in\mscr{C}$. 

    If instead the marking $\mu$ is not complete, then it is semicomplete by \Cref{rmk:complete_or_semicomplete}, and so the set $\supp(s,w)\cup\{m\}$ spans a clique in $\wh S_{\mscr{C}}$. The set $\supp(s,w)\cup\{x\}$ also spans a clique in $\wh S_{\mscr{C}}$, by \Cref{rmk:tree-2-spherical_support}, since $(s,wx)$ is a base and $|wx|=|w|+1$. Together with the hypothesis that $\{m,x\}\in\mscr{C}$, this shows that the set $\supp(s,w)\cup\{x,m\}$ spans a clique in $\wh S_{\mscr{C}}$, and so \Cref{cor:2-compatible->compatible}(1) implies that it lies in $\mscr{C}$. Thus, $\mu$ and $\mu'$ also differ by move (M3)$_{\mscr{C}}$, which we are about to discuss.

    Finally, let $\mu=((s,w),m)$ and $\mu'=((s,w'),m')$ be arbitrary admissible markings differing by move (M3)$_{\mscr{C}}$. Consider the set $\Delta:=\supp(s,w)\cup\supp(s,w')\cup\{m,m'\}$, which is non-spherical and lies in $\mscr{C}$ by hypothesis. Note that $\Delta$ irreducible, since both $\supp(s,w)\cup\{m\}$ and $\supp(s,w')\cup\{m'\}$ are irreducible and these sets intersect (at least) at $s$. Thus, there is a unique $R$--geometric choice of halfspaces for the walls fixed by the elements of $\Delta$ in $\A_R$. Letting $\mc{H}$ be the halfspace assigned to $\mc{W}_s$ by this choice, it suffices to show that $\mc{H}^{\mu}=\mc{H}$ (the same argument yields that $\mc{H}^{\mu'}=\mc{H}$).

    If $\mu$ is semicomplete, the equality $\mc{H}^{\mu}=\mc{H}$ is immediate from definitions. If instead $\mu$ is complete, consider a convex subcomplex $\mc{D}\sq\A_R$ with a $\langle\Delta\rangle$--equivariant isometry $f\colon\A_{\Delta}\ra\mc{D}$ (this exists because $\Delta\in\mscr{C}$). It follows that the wall $w\mc{W}_m$ lies on the side of $\mc{W}_s$ containing the point $f(1_{\Delta})$, and this is precisely $\mc{H}$, concluding the proof of the lemma.
\end{proof}

\begin{rmk}\label{rmk:N1_bad}
    If two admissible markings differ by move (N1)$_{\mscr{C}}$, they do \emph{not} define the same halfspace in $\A_R$ in general. This is not a concern if the markings are both complete (as they then differ by move (M1)$_{\mscr{C}}$) or both semicomplete (as they differ by move (M3)$_{\mscr{C}}$), but it can cause issues when the two markings are of different types. Thus, $\mu\equiv^{\mf{w}}\mu'\not\Ra\mu\equiv_R\mu'$ in general.
\end{rmk}

\begin{rmk}\label{rmk:FC->complete}
    In Coxeter groups of FC-type, all admissible markings are complete and so the difference between moves (M1)$_{\mscr{C}}$ and (N1)$_{\mscr{C}}$ vanishes, as does that between $\equiv^{\mf{w}}$ and $\equiv^{\mf{s}}$.
\end{rmk}

\subsection{Hierarchies}\label{sub:hierarchy}

Let $S\sq W$ and $\mscr{C}$ be as above. The Coxeter generating set $R$ will play no role in the rest of the section, and we can forget about it.

The main tool to find a sequence of moves connecting two markings is the notion of a \emph{hierarchy}. Like most concepts in this section, it was introduced in \cite{CP10}. We will follow many of the core arguments from \cite[Sections~6--8]{CP10}, but we are forced to adapt their form in several ways, as \Cref{prop:markings_new} would not otherwise fit in this framework. 

Before continuing, we state a more general version of \Cref{prop:markings_new}, namely \Cref{prop:markings_general}. This version is not used anywhere in the article, but we believe it helps clarify analogies and differences with \cite{CP10}, and it shows that Coxeter groups of FC-type are a little easier to treat.

In the rest of this subsection, we always implicitly assume that the following data are fixed. Given an oriented path $\kappa$, we denote by $i(\kappa)$ and $t(\kappa)$ its initial and terminal endpoint, respectively.

\begin{setup}\label{setup}
    Consider a set $\overline K\in\mc{J}_S$ and an oriented geodesic $\overline\kappa\sq\wh S_{\mscr{C}}\setminus (\overline K\cup\overline K^{\perp})$. We define $a:=i(\overline\kappa)$ and $b:=t(\overline\kappa)$, and we assume that $\{a,b\}\not\in\mscr{C}$. We also consider a pair $(\eps_I,\eps_T)$ equaling either $(0,0)$, $(1,0)$, or $(1,1)$, depending on whether our goal is to prove Item~(1), Item~(2) or Item~(3) of \Cref{prop:markings_new} (or those of \Cref{prop:markings_general} below).
\end{setup}

Recall that we introduced in \Cref{defn:inseparable} a notion of weakly and strongly inseparable pairs. From now on, we refer to these as \emph{$\mf{w}$--inseparable} and \emph{$\mf{s}$--inseparable} pairs, respectively. Also recall the equivalence relations $\equiv^{\mf{w}}$ and $\equiv^{\mf{s}}$ on admissible markings, which were introduced in \Cref{sub:moves}. 

Now, \Cref{prop:markings_new} corresponds to the case of the following result with $\star=\mf{s}$ (using \Cref{lem:generalised_moves} to deduce $\equiv_R$ from $\equiv^{\mf{s}}$). The case with $\star=\mf{w}$ would instead suffice to treat Coxeter groups of FC-type (in view of \Cref{rmk:FC->complete}), and it is closer in spirit to the results in \cite{CP10}.

\begin{prop}\label{prop:markings_general}
    Let $\star\in\{\mf{w},\mf{s}\}$. Let $(s,w)$ be a base with support $\Sigma\sq S$. Suppose that $\overline K\sq\Sigma$, that the pair $(\overline K,\overline\kappa)$ is $\star$--inseparable, and that we are in one of the following three situations:
    \begin{enumerate}
        \item $\mu:=((s,w),a)$ and $\mu':=((s,w),b)$;
        \item $\mu:=((s,wa),b)$ and $\mu':=((s,w),b)$, where $(s,wa)$ is a base with $|wa|>|w|$;
        \item $\mu:=((s,wa),b)$ and $\mu':=((s,wb),a)$, where $a,b\not\in\Sigma$.
    \end{enumerate}
    If $\star=\mf{s}$, suppose in addition that $\mu$ and $\mu'$ are complete. Then we have $\mu\equiv^{\star}\mu'$.
\end{prop}

The rest of the section is devoted to the proof of \Cref{prop:markings_general}. We will treat the cases $\star=\mf{w}$ and $\star=\mf{s}$ simultaneously, since the argument is essentially the same.

\subsubsection{Definitions and first properties}

Consider the data described in \Cref{setup}. We do not need to worry about the letter $\star\in\{\mf{w},\mf{s}\}$ just yet.

\begin{defn}\label{defn:geodesic_with_domain}
    A \emph{geodesic-with-domain} is a pair $(K,\kappa)$, where $K\in\mc{J}_S$ and $\kappa$ is an oriented geodesic in $\wh S_{\mscr{C}}\setminus (K\cup K^{\perp})$.
\end{defn}

We will usually simply write $\kappa$ in place of $(K,\kappa)$. The set $K$ is the \emph{domain} of the geodesic $\kappa$, and we also denote it by $D(\kappa)$. 

\begin{rmk}
    Although we work with geodesics $\kappa$ in \Cref{defn:geodesic_with_domain}, the only property we will ever use is that they are \emph{chordless} paths: they are injective, and non-consecutive points of the path are not adjacent in the graph $\wh S_{\mscr{C}}$.
\end{rmk}

\begin{defn}\label{defn:arrows_2}
Consider a set $J\in\mc{J}_S$ and two geodesics-with-domain $(K,\kappa)$ and $(L,\lambda)$.
\begin{enumerate}
    \item We write $\kappa\swarrow J$ if there exists a point $x\in\kappa$ such that $J=K\cup\{x\}$. The equality $x=i(\kappa)$ is allowed only if $(K,\kappa)=(\overline K,\overline\kappa)$ and $\eps_I=1$.
    \item We write $\kappa\swarrow\lambda$ if we have $\kappa\swarrow L$ and one of the following two options occurs:
        \begin{itemize}
            \item either $x\neq i(\kappa)$ and $i(\lambda)$ is the point immediately preceding $x$ along $\kappa$;
            \item or $x=i(\kappa)$ and $i(\lambda)=b$ (recall here that $b$ is the \emph{terminal} endpoint of $\overline\kappa$).
        \end{itemize}
\end{enumerate}
The relations $\searrow$ are defined analogously, using $t(\cdot),\eps_T,a$ in place of $i(\cdot),\eps_I,b$, and replacing the word `preceding' by `following'.
\end{defn}

\begin{rmk}
    \begin{enumerate}
        \item[]
        \item If we have $\kappa\swarrow J$, then we also have $J\searrow\kappa$ unless $J=D(\kappa)\cup\{t(\kappa)\}$ and either $\kappa\neq\overline\kappa$ or $\eps_T=0$. Similarly, if we have $J\searrow\kappa$, then we also have $\kappa\swarrow J$ unless $J=D(\kappa)\cup\{i(\kappa)\}$ and either $\kappa\neq\overline\kappa$ or $\eps_I=0$.
        \item If $\kappa\swarrow\lambda$ or $\lambda\searrow\kappa$, then we have $|D(\lambda)|=|D(\kappa)|+1$.
    \end{enumerate}
\end{rmk}

We say that a subset $U\sq S$ is \emph{$\mf{w}$--small} is it is spherical, and that it is \emph{$\mf{s}$--small} if it is tree--$2$--spherical and the difference $U\setminus(\overline K\cup\{a,b\})$ is spherical.

\begin{defn}\label{defn:partial_hierarchy_2}
    A \emph{partial hierarchy} is a set $\mc{H}$ of geodesics-with-domain satisfying the following.
    \begin{enumerate}
        \item[(i)] We have $\overline\kappa\in\mc{H}$. (We call $\overline\kappa$ the \emph{bottom} of the partial hierarchy.)
        \item[(ii)] If $\beta,\varphi\in\mc{H}$ are elements with $D(\beta)=D(\varphi)$, then we have $\beta=\varphi$.
        \item[(iii)] For every $\kappa\in\mc{H}\setminus\{\overline\kappa\}$, there exist $\beta,\varphi\in\mc{H}$ such that $\beta\swarrow\kappa\searrow\varphi$.
    \end{enumerate}
    A partial hierarchy is a \emph{$\star$--hierarchy}, for some $\star\in\{\mf{w},\mf{s}\}$, if it additionally satisfies:
    \begin{enumerate}
        \item[(iv)$_{\star}$] If we have $\beta\swarrow J\searrow\varphi$ for some $\beta,\varphi\in\mc{H}$ and a $\star$--small set $J\in\mc{J}_S$, then there exists a geodesic $\kappa\in\mc{H}$ with $D(\kappa)=J$ and $\beta\swarrow\kappa\searrow\varphi$.
    \end{enumerate}
\end{defn}

Here, $\mf{w}$--hierarchies are almost exactly the same object as the ``hierarchies'' introduced in \cite[Definition~6.4]{CP10}. Instead, $\mf{s}$--hierarchies can be significantly larger, and we will use them to overcome issues related to the difference between the moves (N1)$_{\mscr{C}}$ and (M1)$_{\mscr{C}}$.

Given a set $J\in\mc{J}_S$ and a partial hierarchy $\mc{H}$, define:
\[ \Sigma^-(J;\mc{H}):=\{\overline\kappa\}\cup \{\kappa\in\mc{H}\mid \text{$D(\kappa)\sq J$ and $i(\kappa)\not\in J$} \} .\]
Note that, if $\kappa\swarrow J$ for some $\kappa\in\mc{H}$, then we have $\kappa\in\Sigma^-(J;\mc{H})$. The set $\Sigma^+(J;\mc{H})$ is defined analogously, replacing the point $i(\kappa)$ with $t(\kappa)$.

\begin{lem}\label{lem:from_CP10_2}
    Consider a partial hierarchy $\mc{H}$ and a set $J\in\mc{J}_S$.
    \begin{enumerate}
        \item We have $\Sigma^-(J;\mc{H})=\{\beta_0\dots,\beta_n\}$ with $\overline\kappa=\beta_n\swarrow\dots\swarrow\beta_0$. The geodesics $\beta_1,\dots,\beta_n$ all intersect the set $J$.
        \item If $\beta\swarrow J$ and $\beta'\swarrow J$ for some $\beta,\beta'\in\mc{H}$, then $\beta=\beta'$.
    \end{enumerate}
    Let now $\star\in\{\mf{w},\mf{s}\}$. If $\mc{H}$ is a $\star$--hierarchy and $J$ is $\star$--small, then we also have the following.
    \begin{enumerate}
        \setcounter{enumi}{2}
        \item If $\beta\swarrow J$ for some $\beta\in\mc{H}$, then there exists $\kappa\in\mc{H}$ with $D(\kappa)=J$ and $\beta\swarrow\kappa$.
    \end{enumerate}
    Analogues of all three statements hold for $\Sigma^+$ and $\searrow$.
\end{lem}
\begin{proof}
    The three parts of the lemma correspond to, respectively, Lemma~6.7, Corollary~6.8 and Proposition~6.9 in \cite{CP10}. Despite small differences in terminology, identical proofs apply here.
\end{proof}

\subsubsection{Existence of hierarchies}

While the existence of partial hierarchies is clear, that of $\mf{w}$-- and $\mf{s}$--hierarchies is more delicate. In order to prove the latter, we will need the concept of a \emph{safety line}, which is a new feature of our treatment and will also prove useful later in \Cref{lem:CP_4.6}.

Let $\mc{H}$ be a partial hierarchy and consider a geodesic $\kappa\in\mc{H}\setminus\{\overline\kappa\}$. By \Cref{defn:partial_hierarchy_2}(iii), we can write $\kappa:=\varphi_0\searrow\varphi_1\searrow\dots\searrow\varphi_n$ for some geodesics-with-domain $\varphi_i\in\mc{H}$ and a maximal integer $n\geq 0$. Up to discarding $\varphi_n$, we can assume that the point in $D(\varphi_{n-1})\setminus D(\varphi_n)$ lies in $\varphi_n\setminus\{t(\varphi_n)\}$. After discarding, we have either $\varphi_n=\overline\kappa$ or $\varphi_n\searrow\overline\kappa$. Note that $t(\varphi_n)=b$ in the former case and $t(\varphi_n)=a$ in the latter (for this, recall \Cref{defn:arrows_2}(2)). Also note that, by \Cref{lem:from_CP10_2}, the geodesics $\varphi_i$ are uniquely determined by $\kappa$. 

The \emph{forward safety line} of $\kappa$ is the path $\mf{f}(\kappa)$ constructed as follows. For $1\leq i\leq n$, let $x_i\in\varphi_i$ be the point such that $D(\varphi_{i-1})=D(\varphi_i)\cup\{x_i\}$. Note that $x_i\neq t(\varphi_i)$, and $t(\varphi_{i-1})$ is the point of $\varphi_i$ immediately after $x_i$. Let $\theta_i$ be the arc of the path $\varphi_i$ from $t(\varphi_{i-1})$ to $t(\varphi_i)$. Note that $\theta_i$ is reduced to a singleton if $x_i$ is the second-last point of $\varphi_i$. Finally, define $\mf{f}(\kappa)$ as the concatenation of the paths $\theta_1,\dots,\theta_n$. When $n=0$, we simply set $\mf{f}(\kappa):=\{t(\kappa)\}$. 
In general, $\mf{f}(\kappa)$ is an oriented path from $t(\kappa)$ to $t(\varphi_n)\in\{a,b\}$. The \emph{backward safety line} $\mf{b}(\kappa)$ is defined analogously, and it is an oriented path from a point of $\{a,b\}$ to $i(\kappa)$.

For a subset $U\sq S$, we denote by $\lk_{\mscr{C}}(U)$ the set of elements $x\in S\setminus U$ such that $\{x,u\}\in\mscr{C}$ for all $u\in U$. That is, $\lk_{\mscr{C}}(U)$ is the set of elements that are adjacent to all elements of $U$ within $\wh S_{\mscr{C}}$.

\begin{lem}\label{lem:safety_lines_newnew_2}
    Consider a partial hierarchy $\mc{H}$ and an element $(K,\kappa)\in\mc{H}\setminus\{\overline\kappa\}$.
    \begin{enumerate}
        \item The safety lines $\mf{b}(\kappa)$ and $\mf{f}(\kappa)$ are disjoint from $K\cup K^{\perp}$.
        \item The set $\lk_{\mscr{C}}(K)$ can intersect $\mf{b}(\kappa)$ only at its last point, and $\mf{f}(\kappa)$ only at its first point.
    \end{enumerate}
\end{lem}
\begin{proof}
    We only discuss the forward safety line $\mf{f}(\kappa)$, the other case being identical. Define the arcs $\theta_i\sq\varphi_i$ and points $x_i\in\varphi_i$ as above. Since each $\varphi_i$ is disjoint from $D(\varphi_i)^{\perp}\supseteq K^{\perp}$ by definition, it is clear that $\mf{f}(\kappa)$ is disjoint from $K^{\perp}$.

    Now, suppose that there is a point of intersection $y\in\mf{f}(\kappa)\cap(K\cup\lk_{\mscr{C}}(K))$. Since $K$ is $2$--spherical, it follows that $y$ is at distance $\leq 1$ from all points of $K$ in the graph $\wh S_{\mscr{C}}$. Let $i\geq 0$ be the smallest index such that $y\in\theta_i$, where we artificially define $\theta_0:=\{t(\kappa)\}$. If we had $i\geq 1$, then both $y$ and $x_i\in K$ would lie on the geodesic $\varphi_i$. Since $d(y,x_i)\leq 1$, these points would have to be consecutive along $\varphi_i$ (using that $\varphi_i$ is a geodesic), and it would follow that $y=i(\theta_i)=t(\theta_{i-1})$, violating minimality of the index $i$. This shows that $i=0$ and hence $y=t(\kappa)=i(\mf{f}(\kappa))$. The latter cannot happen for $y\in K$ since $\kappa\cap K=\emptyset$, concluding the proof.
\end{proof}

The following (and its proof) should be compared to Lemma~6.5 and Remark~7.7 in \cite{CP10}.

\begin{prop}\label{prop:existence}
    Let $\star\in\{\mf{w},\mf{s}\}$. If the pair $(\overline K,\overline\kappa)$ is $\star$--inseparable, then there exists a $\star$--hierarchy with $(\overline K,\overline\kappa)$ as its bottom.
\end{prop}
\begin{proof}
    To begin with, we define a particular partial hierarchy $\mc{H}_0$ with one, two, or three elements, depending on whether zero, one, or two of the numbers $\eps_I,\eps_T$ equal $1$. We always have $\overline\kappa\in\mc{H}_0$. If $\eps_I=1$, then $\mc{H}_0$ contains an element $\kappa_I$ with $\overline\kappa\swarrow\kappa_I\searrow\overline\kappa$ 
    and $D(\kappa_I)=\overline K\cup\{a\}$. We are meant to have $i(\kappa_I)=b$, and so we simply choose the geodesic $\kappa_I$ to coincide with $\overline\kappa\setminus\{a\}$ with the reverse orientation. Similarly, if $\eps_T=1$, the partial hierarchy $\mc{H}_0$ contains an element $\kappa_T$ coinciding with $\overline\kappa\setminus\{b\}$ with the reverse orientation and having domain $\overline K\cup\{b\}$.

    The construction of $\mc{H}_0$ ensures that all partial hierarchies $\mc{H}\supseteq\mc{H}_0$ have the following property: for every $\kappa\in\mc{H}\setminus\{\overline\kappa\}$ and every $J\in\mc{J}_S$ with $\kappa\swarrow J$ or $J\searrow\kappa$, the set $J$ intersects $\overline\kappa\setminus\{a,b\}$. 
    If $J$ is $\star$--small, the fact that $(\overline K,\overline\kappa)$ is $\star$--inseparable implies that the set $\{a,b\}$ intersects at most one connected component of $\wh S_{\mscr{C}}\setminus (J\cup J^{\perp})$. (Here it is possible that $J$ contains $a$ or $b$, which does not invalidate the previous statement.)

    Now, partial hierarchies have cardinality at most $|\mc{J}_S|<+\infty$ by \Cref{defn:partial_hierarchy_2}(ii), and so there exists a maximal partial hierarchy $\mc{H}\supseteq\mc{H}_0$. We will show that $\mc{H}$ satisfies \Cref{defn:partial_hierarchy_2}(iv)$_{\star}$, so that it is a $\star$--hierarchy. 

    Suppose for the sake of contradiction that there exist a $\star$--small set $J\in\mc{J}_S$ and two elements $\psi,\varphi\in\mc{H}$ such that $\psi\swarrow J\searrow\varphi$, but $\mc{H}$ contains no geodesic $\kappa$ with $D(\kappa)=J$ and $\psi\swarrow\kappa\searrow\varphi$. Let $p\in\psi$ and $f\in\varphi$ be the points such that $J=D(\psi)\cup\{p\}=D(\varphi)\cup\{f\}$. Observe that we have $p\neq i(\psi)$ and $f\neq t(\varphi)$, otherwise we would have $\psi=\overline\kappa=\varphi$ and we would be able to take $\kappa:=\kappa_I$ or $\kappa:=\kappa_T$, respectively. Let $p'$ be the point of $\psi$ immediately preceding $p$, and let $f'$ be the point of $\varphi$ immediately following $f$. Note that $p'$ and $f'$ lie outside $J\cup J^{\perp}$.
    
    We claim that $p'$ and $f'$ lie in the same connected component of $\wh S_{\mscr{C}}\setminus (J\cup J^{\perp})$. For this, let $\mf{b}$ be the path in $\wh S_{\mscr{C}}$ obtained by concatenating the backward safety line $\mf{b}(\psi)$ with the arc of $\psi$ from $i(\psi)$ to $p'$. Similarly, let $\mf{f}$ be the path that is the concatenation of the terminal arc of $\varphi$ starting at $f'$ with the forward safety line $\mf{f}(\varphi)$. By \Cref{lem:safety_lines_newnew_2}, both $\mf{b}$ and $\mf{f}$ are contained in $\wh S_{\mscr{C}}\setminus (J\cup J^{\perp})$. By the construction of safety lines, the points $i(\mf{b})$ and $t(\mf{f})$ lie in the set $\{a,b\}$
    and, as observed above, $\{a,b\}$ meets only one component of $\wh S_{\mscr{C}}\setminus (J\cup J^{\perp})$. In conclusion, the points $p'$ and $f'$ lie in the same component of $\wh S_{\mscr{C}}\setminus (J\cup J^{\perp})$, proving our claim.
    
    By the claim, there exists a geodesic $\kappa$ from $p'$ to $f'$ within $\wh S_{\mscr{C}}\setminus(J\cup J^{\perp})$, and we obtain a geodesic-with-domain $(J,\kappa)$ with $\psi\swarrow\kappa\searrow\varphi$. We claim that the set $\mc{H}'=\mc{H}\cup\{\kappa\}$ is again a partial hierarchy. The only thing to check is that no two geodesics in $\mc{H}'$ have the same domain, or equivalently, that $\mc{H}$ did not already contain a geodesic $\lambda$ with domain $J$. If this had been the case, then we would have had $\psi'\swarrow\lambda\searrow\varphi'$ for some $\psi',\varphi'\in\mc{H}$ by \Cref{defn:partial_hierarchy_2}(iii), and hence $\psi'\swarrow J\searrow\varphi'$. \Cref{lem:from_CP10_2}(2) would have then implied that $\psi=\psi'$ and $\varphi=\varphi'$, and so $\psi\swarrow\lambda\searrow\varphi$ against our assumptions. 

    In conclusion, $\mc{H}'$ is a partial hierarchy, contradicting maximality of $\mc{H}$. This shows that $\mc{H}$ was indeed a $\star$--hierarchy at the start, as required.
\end{proof}

\subsubsection{From hierarchies to moves}

Having proven that hierarchies indeed exist, we move on to showing how they can be exploited to connect markings by sequences of moves. Before we approach the main result in this direction, we need the following lemma in order to deal with the difference between the moves (N1)$_{\mscr{C}}$ and (M1)$_{\mscr{C}}$ when $\star=\mf{s}$.

\begin{lem}\label{lem:CP_4.6}
    Consider a partial hierarchy $\mc{H}$ with bottom $(\overline K,\overline\kappa)$ and an element $(K,\kappa)\in\mc{H}$. Let $(s,w)$ be a base with support $\Sigma\supseteq K$. Let $m,m'\in\kappa$ be consecutive points such that $\Sigma\cup\{m\}$ is not $2$--spherical and $(K\cup\{m'\})\setminus (\overline K\cup\{a,b\})$ is non-spherical. Then $((s,w),m)\equiv^{\mf{s}} ((s,w),m')$.
\end{lem}
\begin{proof}
    The markings $\mu:=((s,w),m)$ and $\mu:=((s,w),m')$ clearly differ by a single application of move (N1)$_{\mscr{C}}$, but we instead need to connect them by a finite sequence of moves (M1)$_{\mscr{C}}$--(M3)$_{\mscr{C}}$. 
    
    Since $\Sigma\cup\{m\}$ is not $2$--spherical, the marking $\mu$ is complete (see \Cref{rmk:complete_or_semicomplete}). We can thus assume that the set $J:=\Sigma\cup\{m'\}$ is tree--$2$--spherical, since otherwise $\mu$ and $\mu'$ would both be complete and they would simply differ by move (M1)$_{\mscr{C}}$. Since $\Sigma\cup\{m\}$ is not $2$--spherical, 
    we also have $m\not\in J\cup\lk_{\mscr{C}}(J)$. Let $\mf{C}$ be the connected component of $\wh S_{\mscr{C}}\setminus (J\cup\lk_{\mscr{C}}(J))$ containing $m$. We denote by $\partial\mf{C}$ the set of elements of $S\setminus\mf{C}$ that are adjacent to at least one element of $\mf{C}$ within the graph $\wh S_{\mscr{C}}$.

    We claim that $\partial\mf{C}$ contains $(K\cup\{m'\})\setminus (\overline K\cup\{a,b\})$. For this, suppose without loss of generality that $m$ precedes $m'$ along $\kappa$, and let $\mf{b}$ be the path in $\wh S_{\mscr{C}}$ obtained by concatenating the backward safety line $\mf{b}(\kappa)$ with the arc of $\kappa$ from $i(\kappa)$ to $m$. \Cref{lem:safety_lines_newnew_2} shows that $\mf{b}$ is disjoint from 
    \[ K\cup\{m'\}\cup\lk_{\mscr{C}}(K\cup\{m'\}). \]
    Since the latter set contains $J\cup\lk_{\mscr{C}}(J)$, it follows that $\mf{b}\sq\mf{C}$. The construction of safety lines shows that every point of $(K\cup\{m'\})\setminus (\overline K\cup\{a,b\})$ is adjacent (in the graph $\wh S_{\mscr{C}}$) to a point of the path $\mf{b}$, and so we obtain that $(K\cup\{m'\})\setminus (\overline K\cup\{a,b\})\sq\partial\mf{C}$ as claimed. Together with our hypotheses, this shows that the intersection $J\cap\partial\mf{C}$ is non-spherical.

    Now, we can argue exactly as in the proof of \cite[Proposition~4.6]{CP10}.     Namely, since $J$ is tree--$2$--spherical and $m'\not\in K\cup K^{\perp}$, the pair $(s,wm')$ is a base (see \Cref{rmk:tree-2-spherical_support}), and an application of move (M2)$_{\mscr{C}}$ shows that $\mu\equiv^{\mf{s}}((s,wm'),m)$. Let $w_*\in W$ be a longest element such that:
    \begin{enumerate}
        \item $(s,wm'w_*)$ is a base and $|wm'w_*|=|w|+1+|w_*|$;
        \item $((s,wm'w_*),m)\equiv^{\mf{s}}\mu$.
    \end{enumerate}
    Note that $w_*=1$ satisfies the above two conditions, and thus the existence of a longest element with these properties follows from the Parallel Wall Theorem (see \Cref{rmk:PW}). Set $\beta:=(s,wm'w_*)$ from now on for simplicity. We will show that $(\beta,m)\equiv^{\mf{s}}\mu'$ thereby proving the lemma.

    As shown above, the set $J\cap\partial\mf{C}$ is non-spherical. Let $\Xi\sq J\cap\partial\mf{C}$ be the set of elements $u$ with 
    \[ d(\mc{Y}_s,wm'w_*u\cdot 1_S)\leq d(\mc{Y}_s,wm'w_*\cdot 1_S) .\] 
    By \cite[Lemma~8.2]{CP10}, the set $\Xi$ is spherical and so there exists an element $u\in J\cap\partial\mf{C}\setminus\Xi$. By \Cref{rmk:complete_or_base}, either the pair $(s,wm'w_*u)$ is a base, or the pair $(\beta,u)$ is a complete marking. Choose an element $x\in\mf{C}$ that is adjacent to $u$ within the graph $\wh S_{\mscr{C}}$.

    Observe that we have $(\beta,m)\equiv^{\mf{s}}(\beta,x)$. Indeed, setting $\Sigma_*:=\supp(\beta)$, we have $\Sigma_*\supseteq\Sigma$ and hence, since $\Sigma_*$ is $2$--spherical by \Cref{rmk:tree-2-spherical_support}, we also have $\Sigma_*\cup\lk_{\mscr{C}}(\Sigma_*)\sq\Sigma\cup\lk_{\mscr{C}}(\Sigma)$. Thus, there exists a connected component $\mf{C}_*\sq\wh S_{\mscr{C}}\setminus(\Sigma_*\cup\lk_{\mscr{C}}(\Sigma_*))$ with $\mf{C}\sq\mf{C}_*$. For each element $y\in\mf{C}_*$, the marking $(\beta,y)$ is complete (see again \Cref{rmk:complete_or_semicomplete}). Since $m,x\in\mf{C}_*$, it follows that $(\beta,m)\equiv^{\mf{s}}(\beta,x)$ by a repeated application of move (M1)$_{\mscr{C}}$.

    Now, we have seen above that either the pair $(s,wm'w_*u)$ is a base, or the pair $(\beta,u)$ is a complete marking. In the former case, move (M2)$_{\mscr{C}}$ would yield $(\beta,x)\equiv^{\mf{s}}((s,wm'w_*u),x)$, violating maximality of the element $w_*$. Therefore, $(\beta,u)$ must be a complete marking, and hence we have $\mu\equiv^{\mf{s}}(\beta,x)\equiv^{\mf{s}}(\beta,u)$, where the last equivalence is by a single application of move (M1)$_{\mscr{C}}$. Finally, since $\supp(\beta)\cup\{u\}=\Sigma_*$ is $2$--spherical and contains $J$, 
    we obtain $(\beta,u)\equiv^{\mf{s}}\mu'$ by move (M3)$_{\mscr{C}}$, concluding the proof.
\end{proof}

Hierarchies allow us to prove equivalences for pairs of markings of the following form.

\begin{defn}
    A pair of \emph{extremal markings} for $(\overline K,\overline\kappa)$ is pair of markings $(\mu_I,\mu_T)$ obtained as follows. Choose any base $(s,w)$ with $\supp(s,w)\supseteq\overline K$ and set:
    \begin{itemize}
        \item $\mu_I:=((s,w),a)$ if $\eps_I=0$, or $\mu_I:=((s,wa),b)$ if $\eps_I=1$;
        \item $\mu_T:=((s,w),b)$ if $\eps_T=0$, or $\mu_T:=((s,wb),a)$ if $\eps_T=1$.
    \end{itemize}
    Here we implicitly assume that the pairs $(s,wa)$ and $(s,wb)$ are indeed bases when they appear, and we additionally require that their supports contain $\supp(s,w)$ (in case $a$ or $b$ lie in $\supp(s,w)$). We refer to $(s,w)$ as the \emph{spawning base} of the pair $(\mu_I,\mu_T)$.
\end{defn}

We can now essentially copy the discussion in \cite[Section~7]{CP10} to obtain the following result. An important difference is that, for us, the spawning base of a pair of extremal markings can have support that \emph{strictly} contains the bottom domain $\overline K$.

\begin{prop}\label{prop:Sect7_2}
    Let $\star\in\{\mf{w},\mf{s}\}$. Let $\mc{H}$ be a $\star$--hierarchy and let $(\mu_I,\mu_T)$ be a pair of extremal markings for $(\overline K,\overline\kappa)$ with spawning base $(s,w)$. Suppose that $\mu_I$ and $\mu_T$ are complete. Then there exists a finite sequence of markings $\xi_j=((s,w_j),m_j)$ for $0\leq j\leq n$ with the following properties.
    \begin{enumerate}
    \setlength\itemsep{.2em}
        \item All $\xi_j$ are admissible and have the same core $s$. We have $\xi_0=\mu_I$ and $\xi_n=\mu_T$.
        \item There are elements $\kappa_j\in\mc{H}$ such that $m_j\in\kappa_j$ and $\supp(s,w_j)=\supp(s,w)\cup D(\kappa_j)$. Moreover, either $\xi_j$ is complete or the set $D(\kappa_j)\cup\{m_j\}$ is not $\star$--small.
        \item For each $0\leq j<n$, the transition from $\xi_j$ to $\xi_{j+1}$ is of one of four kinds:
            \begin{enumerate}
            \setlength\itemsep{.2em}
                \item[(i)] $\kappa_j=\kappa_{j+1}$ and $w_j=w_{j+1}$ and $m_j$ immediately precedes $m_{j+1}$ along $\kappa_j$;
                \item[(ii)] $\kappa_j\swarrow\kappa_{j+1}$ and $w_{j+1}\in\{w_j,w_jx\}$, where $x$ denotes the point immediately following $i(\kappa_{j+1})$ along $\kappa_j$, and we have $m_j=m_{j+1}=i(\kappa_{j+1})$;
                \item[(iii)] $\kappa_j\searrow\kappa_{j+1}$ and $w_{j+1}\in\{w_j,w_jx\}$, where $x$ denotes the point immediately preceding $t(\kappa_j)$ along $\kappa_{j+1}$, and we have $m_j=m_{j+1}=t(\kappa_j)$; 
                \item[(iv)] we have $m_j=t(\kappa_j)$ and $m_{j+1}=i(\kappa_{j+1})$ and there exists a geodesic $\varphi\in\mc{H}$ such that $\kappa_j\searrow\varphi\swarrow\kappa_{j+1}$, with the point $m_{j+1}$ immediately preceding $m_j$ along $\varphi$.
            \end{enumerate}
    \end{enumerate}
    In particular, we have $\mu_I\equiv^{\star}\mu_T$.
\end{prop}
\begin{proof}
    Set $\Sigma:=\supp(s,w)$ and $\xi_0:=\mu_I$. If $\eps_I=0$, we set $\kappa_0:=\overline\kappa$. If instead $\eps_I=1$, we have $\overline\kappa\swarrow D(\overline\kappa)\cup\{i(\overline\kappa)\}$ and so, by \Cref{lem:from_CP10_2}(3) and \Cref{defn:partial_hierarchy_2}(ii), there exists a unique geodesic $\lambda\in\mc{H}$ with $\overline\kappa\swarrow\lambda$ and $D(\lambda)=D(\overline\kappa)\cup\{i(\overline\kappa)\}$; here we set $\kappa_0:=\lambda$. Either way, we have $m_0\in\kappa_0$ and $\supp(s,w_0)=\Sigma\cup D(\kappa_0)$, so that Item~(2) holds for $j=0$. 
    Writing $\mu_T=((s,w_{\rm last}),m_{\rm last})$, an analogous argument yields a (unique) geodesic $\kappa_{\rm last}\in\mc{H}$ with $\supp(s,w_{\rm last})=\Sigma\cup D(\kappa_{\rm last})$ and $m_{\rm last}=t(\kappa_{\rm last})$. Moreover, either $\eps_T=0$ and $\kappa_{\rm last}=\overline\kappa$, or $\eps_T=1$ and $\kappa_{\rm last}\searrow\overline\kappa$.

    Suppose now that, for some index $k\geq 0$, the markings $\xi_0,\dots,\xi_k$ and elements $\kappa_0,\dots,\kappa_k\in\mc{H}$ have been defined so that Items~(1) and~(2) hold for $j\leq k$ and Item~(3) holds for $j<k$. If $\xi_k=\mu_T$, we set $n:=k$ and the proof is complete. Otherwise, we define $\xi_{k+1}$ and $\kappa_{k+1}$ as follows.

    \smallskip
    {\bf Case~(A):} suppose first that $m_k\neq t(\kappa_k)$. Let $x\in\kappa_k$ be the point immediately after $m_k$, and define $J:=D(\kappa_k)\cup\{x\}$. 

    If the set $J$ is not $\star$--small, then we set $w_{k+1}:=w_k$, $m_{k+1}:=x$ and $\kappa_{k+1}:=\kappa_k$. Here we are in Case~(i) of Item~(3).
    
    Suppose instead that $J$ is $\star$--small. Then, the fact that $\kappa_k\swarrow J$ and \Cref{lem:from_CP10_2}(3) yield a unique geodesic $\lambda\in\mc{H}$ with $\kappa_k\swarrow\lambda$ and $D(\lambda)=J$. We then set $w_{k+1}:=w_k$ if $x\in\Sigma$, and $w_{k+1}:=w_kx$ otherwise, noting that $(s,w_{k+1})$ is a base by \Cref{rmk:tree-2-spherical_support}. We also set $m_{k+1}:=m_k$ and $\kappa_{k+1}:=\lambda$, so that we are in Case~(ii). By the inductive hypothesis, either $\xi_k$ was complete or $D(\kappa_k)\cup\{m_k\}$ was not $\star$--small, and this implies that either $\xi_{k+1}$ is complete or $D(\kappa_{k+1})\cup\{m_{k+1}\}$ is not $\star$--small.

    \smallskip
    {\bf Case~(B):} suppose now that $m_k=t(\kappa_k)$. Since $\xi_k\neq\mu_T$, we have $\kappa_k\searrow\varphi$ for some $\varphi\in\mc{H}$. We have $m_k\neq i(\varphi)$ and the point $x\in\varphi$ immediately before $m_k$ satisfies $D(\kappa_k)=D(\varphi)\cup\{x\}$. Consider the set $J:=D(\varphi)\cup\{m_k\}$. 
    
    If $J$ is not $\star$--small, then we set $w_{k+1}:=w_k$ if $x\in\Sigma$, and $w_{k+1}:=w_kx$ otherwise, noting that in the latter case $\supp(s,w_{k+1})=\supp(s,w_k)\setminus\{x\}$. We also set $m_{k+1}:=m_k$ and $\kappa_{k+1}:=\varphi$, landing in Case~(iii). 
    
    If instead $J$ is $\star$--small, the fact that $\varphi\swarrow J$ and \Cref{lem:from_CP10_2}(3) guarantee the existence of a geodesic $\lambda\in\mc{H}$ with $\varphi\swarrow\lambda$ and $D(\lambda)=J$. We then set $w_{k+1}:=w_k\hat x\hat m_k$, where $\hat x:=1$ if $x\in\Sigma$ and $\hat x:=x$ otherwise, and where $\hat m_k$ is defined analogously. We also set $m_{k+1}:=x$ and $\kappa_{k+1}:=\lambda$, finding ourselves in Case~(iv). Again the inductive hypothesis guarantees that either $\xi_{k+1}$ is complete or $D(\kappa_{k+1})\cup\{m_{k+1}\}$ is not $\star$--small. 

    \smallskip
    This proves Items~(1)--(3) of the proposition. (The sequence cannot go on forever, since the hierarchy is finite.) For the ``in particular'' statement, it suffices to check that $\xi_j\equiv^{\star}\xi_{j+1}$ for each index $j$. In Case~(iv), $\xi_j$ and $\xi_{j+1}$ differ by move (M3)$_{\mscr{C}}$, while in Cases~(ii) and~(iii) they are either equal or differ by move (M2)$_{\mscr{C}}$. In Case~(i), the two markings differ by move (N1)$_{\mscr{C}}$, which suffices to obtain $\xi_j\equiv^{\mf{w}}\xi_{j+1}$. If $\star=\mf{s}$, the markings $\xi_j$ and $\xi_{j+1}$ still differ by move (M1)$_{\mscr{C}}$ or (M3)$_{\mscr{C}}$ if they are both complete or both semicomplete (\Cref{rmk:N1_bad}). If one of them is not complete and the other is not semicomplete, then we can use the ``moreover'' statement in Item~(2) together with \Cref{lem:CP_4.6} to conclude that we nevertheless have $\xi_j\equiv^{\mf{s}}\xi_{j+1}$. This finally proves the whole proposition.
\end{proof}

\begin{rmk}
    In the proof of \Cref{prop:Sect7_2}, it is easy to see that the markings $\xi_j$ for which $D(\kappa_j)\cup\{m_j\}$ is $\star$--small can only appear at the very start and very end of the sequence.
\end{rmk}

\begin{rmk}\label{rmk:w_case}
    When $\star=\mf{w}$, \Cref{prop:Sect7_2} and its proof also work when $\mu_I$ and $\mu_T$ are not complete (as long as they are admissible), simply removing the ``moreover'' statement from Item~(2). Indeed, that statement is only needed in order to apply \Cref{lem:CP_4.6} to prove $\xi_j\equiv^{\mf{s}}\xi_{j+1}$ in Case~(i).
\end{rmk}

\subsubsection{Conclusion}

The above discussion proves \Cref{prop:markings_general}. We briefly summarise here how to combine the various pieces of the argument.

\begin{proof}[Proof of \Cref{prop:markings_general}]
    Let $\overline K$ and $\kappa$ be as in the statement of the proposition. Let $\mu_I$ and $\mu_T$ be the markings appearing in one of Items~(1)--(3). The fact that the pair $(\overline K,\overline\kappa)$ is $\star$--inseparable allows us to apply \Cref{prop:existence} and deduce the existence of a $\star$--hierarchy with bottom $(\overline K,\overline\kappa)$. Together with \Cref{prop:Sect7_2}, this shows that $\mu_I\equiv^{\star}\mu_T$. Here the hypothesis that $\mu_I$ and $\mu_T$ be complete is only needed for $\star=\mf{s}$, as explained in \Cref{rmk:w_case}.
\end{proof}

\section{Geometrisation}\label{sect:geometrisation}

This section finally proves \Cref{thm:step_two}, which is the second and last step in the sketch given in \Cref{sect:strategy}. The proof of \Cref{thm:step_two} will be set up in \Cref{sub:statement_shortening} and finally carried out in \Cref{sub:proof_shortening}. Before then, we collect some miscellaneous lemmas in \Cref{sub:general_preliminaries} and then discuss in \Cref{sub:semivisual} a weakening of the notion of a visual splitting that is important for the proof.

\subsection{General preliminaries}\label{sub:general_preliminaries}

Let $(W,S)$ be a Coxeter system with Davis complex $\A_S$. As in Sec\-tion~\ref{sect:markings}, we denote by $1_S\in\A_S$ the base vertex, and by $\mc{Y}_{\rho}\sq\A_S$ the wall fixed by some $\rho\in S^W$. In this subsection, we first collect some properties relating to the subcomplexes of $\A_S$ associated to $S$--parabolic subgroups of $W$, and then obtain a few lemmas needed in the rest of the section.

\subsubsection{Standard subcomplexes}\label{subsub:standard}

A subcomplex $\mc{C}\sq\A_S$ is \emph{convex} if it contains all geodesics in the $1$--skeleton of $\A_S$ with endpoints in $\mc{C}$. (Such a subcomplex might not be convex with respect to the CAT(0) metric on $\A_S$, but we never consider this metric in what follows.)
Whenever we speak of a halfspace $\mc{H}\sq\A_S$, from now on, we will always refer to a subcomplex of $\A_S$, namely the largest subcomplex contained on the $\mc{H}$--side of the corresponding wall. With this convention, halfspaces are convex subcomplexes (e.g.\ by \cite[Lemma~3.2.14]{Davis}).

Given an $S$--parabolic subgroup $P\leq W$, a \emph{$P$--standard subcomplex} is a minimal $P$--invariant convex subcomplex of $\A_S$. Equivalently, a convex subcomplex $\mc{P}\sq\A_S$ is $P$--standard if and only if the walls of $\A_S$ crossed by $\mc{P}$ are precisely those fixed by the reflections in $P$. When $P=\langle T\rangle$ for a subset $T\sq S$, a subcomplex $\mc{P}\sq\A_S$ is $\langle T\rangle$--standard if and only if there exists a $\langle T\rangle$--equivariant isometry $\phi\colon\A_T\ra\mc{P}$. Note that, when $T$ has no spherical factors, the equivariant isometry $\phi$ is unique (for a given $\mc{P}$). When $T$ is spherical and irreducible, there are two possibilities for $\phi$, differing by composition with left multiplication by the longest element of $\langle T\rangle$. A single parabolic subgroup $P$ will in general admit infinitely many distinct $P$--standard subcomplexes; these are pairwise disjoint and the normaliser $N_W(P)$ permutes them, though not always transitively.

\begin{ex}
    For $s\in S$, the $\langle s\rangle$--standard subcomplexes are the edges dual to the wall $\mc{Y}_s\sq\A_S$.
\end{ex}

Given a $P$--standard subcomplex $\mc{P}\sq\A_S$, we denote by $\pi_{\mc{P}}\colon\A_S^{(0)}\ra\mc{P}^{(0)}$ the nearest-point projection with respect to the path metric on the $1$--skeleton of $\A_S$. This projection is well-defined, for instance as a special case of \cite[Lemma~4.3.1]{Davis}. Given two vertices $x,y\in\A_S$, the walls of $\A_S$ separating $\pi_{\mc{P}}(x)$ from $\pi_{\mc{P}}(y)$ are precisely the walls that cross $\mc{P}$ and separate $x$ from $y$. In particular, no wall crossing $\mc{P}$ separates $x$ from $\pi_{\mc{P}}(x)$.

Given a $P$--standard subcomplex $\mc{P}\sq\A_S$ and a $Q$--standard subcomplex $\mc{Q}\sq\A_S$, the projections $\pi_{\mc{P}}(\mc{Q})$ and $\pi_{\mc{Q}}(\mc{P})$ are two $P\cap Q$--standard subcomplexes, and they are $P\cap Q$--equivariantly isometric to each other via the restrictions of $\pi_{\mc{P}}$ and $\pi_{\mc{Q}}$.

Given an $S$--geometric set of reflections $T\sq S^W$, we say that a point $x\in\A_S$ is \emph{$T$--geometric} if it lies in the sector determined by a geometric choice of halfspaces for the walls of $\A_S$ fixed by the elements of $T$. If $T$ is $S$--compatible, then each $\langle T\rangle$--standard subcomplex of $\A_S$ contains $2^n$ vertices that are $T$--geometric, where $n$ is the number of irreducible spherical factors of $T$.

\subsubsection{Miscellaneous lemmas}

For each $s\in S$, we denote by $\mc{H}_s$ and $\mc{H}_s^*$ the two halfspaces of $\A_S$ bounded by the wall $\mc{Y}_s$. We name these two halfspaces so that $1_S\in\mc{H}_s$. For spherical subsets $K\sq S$, we denote by $w_K$ the longest element of $\langle K\rangle$, with the convention that $w_{\emptyset}=1$. 

For a subset $U\sq S$, we denote by $\A_U\sq\A_S$ the only $\langle U\rangle$--standard subcomplex containing $1_S$. To streamline notation, we will simply write ``$g\in\A_S$'' instead of ``$g1_S\in\A_S$'' for elements $g\in W$.

\begin{lem}\label{lem:ww_2}
    Consider a subset $U\sq S$ with a partition $U=F_1\sqcup F_2$ such that $F_1$ is spherical. 
    Then the intersection 
    \[ \A_U\cap\bigcap_{u\in F_1}\mc{H}_u \cap\bigcap_{u\in F_2}w_{F_1}\mc{H}_u^*\]
    is nonempty if and only if $U$ is spherical, in which case it equals the singleton $\{w_{F_1}w_U\}$.
\end{lem}
\begin{proof}
    For every $u\in F_1$, we can set $u^*:=w_{F_1}uw_{F_1}\in F_1$ and we have $w_{F_1}\mc{H}_u=\mc{H}_{u^*}^*$. Thus, up to applying $w_{F_1}$ to the whole intersection, the lemma states that $\A_U\cap\bigcap_{u\in U}\mc{H}_u^*$ is nonempty if and only if $U$ is spherical, in which case it equals $\{w_U\}$. This is classical, see e.g.\ \cite[Lemma~4.6.1]{Davis}.
\end{proof}

\begin{lem}\label{lem:producing_bases_2}
    Let $T=U\sqcup\{t\}$ be a spherical subset of $S$. Then, for every element $g\in\langle U\rangle$, there exists an element $w\in\langle U\rangle$ such that $w^{-1}\mc{H}_t=g^{-1}\mc{H}_t$ and $(t,w)$ is a base.
\end{lem}
\begin{proof}
    Let $\overline g$ be the projection of the element $g\in\langle U\rangle$ to the parabolic subgroup $\langle U\cap t^{\perp}\rangle$. Set $w:=\overline g^{-1} g$. It is clear that $w^{-1}\mc{H}_t=g^{-1}\mc{H}_t$, so we only need to check that $(t,w)$ is a base. Recalling that $T$ is spherical and using \cite[Remark~3.2(i)]{CP10}, this amounts to checking that, for all walls $\mc{V}$ separating $1_T$ and $w$, the vertex $1_T$ lies in an acute-angled sector spanned by $\mc{V}$ and $\mc{Y}_t$. 
    
    Thus, suppose for the sake of contradiction that $1_T$ and $w$ are separated by a wall $\mc{V}$ such that $1_T$ lies in a right-or-obtuse-angled sector $\Sigma$ spanned by $\mc{V}$ and $\mc{Y}_t$. The sector $\Sigma':=\overline g\Sigma$ then contains $\overline g$ and is spanned by the walls $\mc{V}':=\overline g\mc{V}$ and $\mc{Y}_t$, where $\mc{V}'$ separates $\overline g$ from $g$. 

    Let $r'$ be the reflection in the wall $\mc{V}'$. If the sector $\Sigma'$ were right-angled, then $r'$ and $t$ would commute. Since $t\not\in\supp(r')\sq U$, \Cref{lem:simpler_Deodhar} would imply that $r'$ lies in $\langle U\cap t^{\perp}\rangle$, contradicting the fact that $\overline g$ is the projection of $g$ to the latter subgroup. 

    Thus, the sector $\Sigma'$ is obtuse-angled. It follows that the $T$--parabolic closure $D\leq W$ of the dihedral subgroup $\langle r',t\rangle$ contains a third reflection $r''$ whose fixed wall $\mc{V}''$ cuts the sector $\Sigma'$ into two halves. Let $\gamma\sq\langle U\rangle$ be a geodesic from $g$ to $\overline g$, and observe that there exists an edge $e\sq\A_T$ incident to $\overline g$ and crossing the wall $\mc{Y}_t$. Since the path $\gamma\cup e$ crosses both $\mc{V}'$ and $\mc{Y}_t$, it follows that the geodesic $\gamma$ must cross the wall $\mc{V}''$. Since $\gamma\sq\langle U\rangle$, we obtain that $r',r''\in\langle U\rangle$, which implies that $D\sq\langle U\rangle$ and hence $t\in\langle U\rangle$. This contradicts the fact that $t\not\in U$, concluding the proof.
\end{proof}

Let now $R\sq W$ be a Coxeter generating set that is angle-compatible with $S$. As usual, for $s\in S$, we denote by $\mc{W}_s$ the $s$--fixed wall in the Davis complex $\A_R$, to distinguish it from the wall $\mc{Y}_s\sq\A_S$.

\begin{lem}\label{lem:auxiliary}
    Let $A,B\sq S$ be two irreducible sets with two elements $a\in A\setminus B$ and $b\in B\setminus A$. Suppose that $A$ and $B$ are $R$--compatible, and that $\langle a,b\rangle$ is infinite. Let $\mc{A},\mc{B}\sq\A_R$ be, respectively, an $\langle A\rangle$--standard and a $\langle B\rangle$--standard subcomplex. Suppose that the projection of $\mc{B}$ to $\mc{A}$ contains an $A$--geometric point $p_A$. Then, for every base $(s,w)$ with $a\in\supp(s,w)\sq A$, the point $p_A$ and the wall $w\mc{W}_b$ are contained in the same halfspace bounded by the wall $\mc{W}_s$.
\end{lem}
\begin{proof}
    Set $C:=A\cap B$. The subcomplex $\mc{C}:=\pi_{\mc{A}}(\mc{B})$ is $\langle C\rangle$--standard, and it contains the $A$--geometric point $p_A$ by hypothesis. Since $B$ is $R$--compatible, there are a $B$--geometric point $p_B\in\mc{B}$ and a $\langle C\rangle$--standard subcomplex $\mc{C}'\sq\mc{B}$ containing $p_B$. Note that $\pi_{\mc{A}}(\mc{C}')=\mc{C}$ and that $\mc{W}_b$ crosses an edge of $\mc{B}$ incident to $p_B\in\mc{C}'$.
    
    Since $a\in\supp(s,w)$, we have $\supp(s,w)\not\sq C$, and so the wall $w^{-1}\mc{W}_s$ does not cross $\mc{C}$. Since $\supp(s,w)\sq A$ and $\langle a,b\rangle$ is infinite, we also have $w^{-1}\mc{W}_s\cap\mc{W}_b=\emptyset$ (e.g.\ using Remarks~\ref{rmk:complete_or_base} and~\ref{rmk:tree-2-spherical_support}). Now, $p_A$ and $\mc{C}$ are on the same side of $w^{-1}\mc{W}_s$, and this side also contains $\mc{C}'$ and $\mc{W}_b$. This shows that $w\mc{W}_b$ and $wp_A$ are contained in the same halfspace bounded by the wall $\mc{W}_s$. Finally, since $p_A$ is $A$--geometric and $(s,w)$ is a base, the points $wp_A$ and $p_A$ lie on the same side of $\mc{W}_s$, thus completing the proof.
\end{proof}

\subsection{Semivisual splittings}\label{sub:semivisual}

The concept of a visual splitting (\Cref{defn:S-visual}) is not robust enough for us to carry out the proof of \Cref{thm:step_two} (see \Cref{rmk:semivisual_motivation} below). For this reason, we will need to work with a slightly more general class of splittings, which we term \emph{semivisual}.

Let $(W,S)$ be a Coxeter system, let $\Om\sq S$ be a subset and set $H:=\langle\Om\rangle$. The \emph{$\Om$--core} of a splitting $H\acts\mc{T}$ is the (unique) smallest subtree $\mc{F}\sq\mc{T}$ that intersects the fixed set of every element of $\Om$. Note that $\mc{F}$ is always a finite subtree, since $\Om$ is finite.
As in \Cref{sect:splittings}, for each vertex $v\in\mc{F}$ and edge $e\sq\mc{F}$ we denote by $\Om_v$ and $\Om_e$ the subsets of $\Om$ fixing $v$ and $e$, respectively, and by $H_v$ and $H_e$ their $H$--stabilisers. We have $\langle \Om_v\rangle\leq H_v$ and $\langle \Om_e\rangle\leq H_e$, but in general these inclusions are proper.

\begin{defn}
    A splitting $H\acts\mc{T}$ is \emph{$\Om$--semivisual} if the following hold for the $\Om$--core $\mc{F}\sq\mc{T}$.
    \begin{enumerate}
    \item For each edge $e\sq\mc{F}$, we have $H_e=\langle \Om_e\rangle$.
    \item Denoting by $\mc{V}^*\sq\mc{F}$ the set of vertices such that $H_v=\langle \Om_v\rangle$, we have that every $\mc{T}$--elliptic subset of $\Om$ is contained in $\Om_v$ for some $v\in\mc{V}^*$. In particular, we have $\Om=\bigcup_{v\in\mc{V}^*}\Om_v$.
    \end{enumerate}
\end{defn}

Note that an $\Om$--semivisual splitting is $\Om$--visual precisely when we have $\mc{V}^*=\mc{F}^{(0)}$ and in addition the $\Om$--core $\mc{F}\sq\mc{T}$ is a \emph{fundamental} subtree. The latter is equivalent to the requirement that, for each vertex $v\in\mc{F}$, no two edges of $\mc{F}$ incident to $v$ be in the same $H_v$--orbit.

\begin{rmk}\label{rmk:maximal_V*}
    If $H\acts\mc{T}$ is $\Om$--semivisual and strongly reduced, then the sets $\Om_v$ with $v\in\mc{V}^*$ are precisely the maximal subsets of $\Om$ that are elliptic in $\mc{T}$.
\end{rmk}

In the rest of the section, we will consider the following setting. As usual, let $S,R\sq W$ be angle-compatible Coxeter generating sets, let $\mscr{C}$ be the family of $\{S,R\}$--compatible subsets of $W$, and let $\mscr{T}_S$ be the collection of generating sets twist-equivalent to $S$ relative to $\mscr{C}$. 

\begin{setup}\label{setup_two}
    Let $\Om\sq S$ be a $\mscr{C}$--inflexible, $S$--tucked subset. Set $H:=\langle\Om\rangle$ and consider a splitting $H\acts\mc{T}$ such that:
    \begin{enumerate}
        \item $\mc{T}$ is strongly reduced and $\Om$--semivisual;
        \item $\mc{T}$ is relative to $\mscr{C}$, and the $H$--stabiliser of each vertex of $\mc{T}$ is generated by a set in $\mscr{C}$;
        \item all edge-stabilisers of $\mc{T}$ are $H$--conjugate to each other. 
    \end{enumerate}
\end{setup}

We recall the statement of \Cref{thm:step_two}, since it is the main goal of the rest of the section:

\begin{thm}\label{thm:step_two_repeated}
    Given \Cref{setup_two}, suppose that $\Om$ is $\mscr{T}_S$--parabolic and $\mscr{T}_S$--tucked, and that \Cref{thmintro:main} holds for all proper $S$--parabolic subgroups of $W$. Then there exist a Coxeter generating set $S'\in\mscr{T}_S$ and a subset $\Om'\sq S'$ such that $\langle\Om'\rangle=\langle\Om\rangle$ and $\Om'$ is $R$--geometric.
\end{thm}

\begin{rmk}
    We simply ask that $\Om$ be $\mscr{C}$--inflexible and $S$--tucked in \Cref{setup_two}, because these are the hypotheses under which we will be working most of the time. They suffice to produce a finite sequence of elementary twists of $(W,S)$ replacing $\Om$ with a set $\Om'$ that is a little closer to being $R$--geometric. The assumption that $\Om$ be $\mscr{T}_S$--parabolic and $\mscr{T}_S$--tucked is needed only to ensure that $\Om'$ be again inflexible and tucked, so that the geometrisation procedure can continue.

    Similarly, the hypothesis that \Cref{thmintro:main} hold for all proper parabolic subgroups is only needed in order to invoke \Cref{prop:inflexible} and deduce inflexibility from $\mscr{T}_S$--parabolicity.
\end{rmk}

\begin{rmk}\label{rmk:semivisual_motivation}
    The existence of the splitting $H\acts\mc{T}$ in \Cref{setup_two} is entirely equivalent to asking that $\Om$ be a $\mscr{C}$--shrub. In one direction, any $\mscr{C}$--shrub corresponds to a splitting with the above properties via \Cref{lem:visual_equivalence} (in fact, to an $\Om$--\emph{visual} splitting with these properties). Conversely, if $\Om$ is a set with a splitting $\langle\Om\rangle\acts\mc{T}$ as above, then we obtain a visual decomposition $\mf{F}=(\mc{F},\{\Om_v\}_v)$ of $\Om$, where $\mc{F}$ is the $\Om$--core of $\mc{T}$ and $\Om_v$ is as above. We can then collapse some edges $\mc{F}$ to ensure that $\mf{F}$ is strongly reduced, and the other requirements in \Cref{defn:shrub} immediately follow.

    The reason why we stated \Cref{setup_two} as above (instead of simply asking that $\Om$ be a shrub) is again to be found in the next subsections (particularly in the proof of \Cref{lem:semivisual_stays}). We will be modifying $S$ and $\Om$ by elementary twists, while keeping the subgroup $H$ unchanged, and it is convenient to maintain the splitting $H\acts\mc{T}$ unaltered throughout this procedure. The new sets $\Om'$ will all be $\mscr{C}$--shrubs, but $\mc{T}$ will only be $\Om'$--semivisual in general, even if we start with an $\Om$--visual splitting.

    Note that our elementary twists will not come from the splitting $\mc{T}$, since they need to extend to $W$, while $\mc{T}$ does not (a priori). The splitting $\mc{T}$ is only used in order to ``measure'' how far $\Om$ is from being $R$--geometric.
\end{rmk}

With these clarifications out of the way, we move on to the proof of \Cref{thm:step_two_repeated}. In preparation for it, we will have to develop a substantial amount of terminology and notation in the next subsection, which will culminate in the statement of the ``shortening theorem'' (\Cref{thm:shortening}). After this, \Cref{sub:proof_shortening} will contain the proofs of \Cref{thm:shortening} and, at the very end, \Cref{thm:step_two_repeated}.

\subsection{Statement of the shortening theorem}\label{sub:statement_shortening}

Let $\Om$, $H$ and $\mc{T}$ be as in \Cref{setup_two} throughout.

We encourage the reader to keep in mind the simplified setting in which the $\Om$--core of $\mc{T}$ is a single edge, so that we simply have $\Om=X\cup Y$ for two sets $X,Y\in\mscr{C}$ with $\{x,y\}\not\in\mscr{C}$ for all $x\in X\setminus Y$ and $y\in Y\setminus X$. In this case, most of the terminology and notation in this subsection become trivial, while the conceptual core of the argument (in \Cref{sub:proof_shortening}) remains essentially unaltered.

\subsubsection{Blobs}

Let $\mc{F}_{\Om}\sq\mc{T}$ be the $\Om$--core. For each vertex $v\in\mc{F}_{\Om}$ and edge $e\sq\mc{F}_{\Om}$, let $\Om_v$ and $\Om_e$ be the subsets of $\Om$ fixing $v$ and $e$. Let $\mc{V}_{\Om}^*$ be the set of vertices $v\in\mc{F}_{\Om}$ with $H_v=\langle\Om_v\rangle$. Define:
\begin{align*}
    \mc{V}_{\Om}&:=\{\Om_v\mid v\in\mc{V}_{\Om}^*\} , & \mc{E}_{\Om}&:=\{\Om_e\mid e\sq\mc{F}_{\Om}\} .
\end{align*}
Since $\mc{T}$ is $\Om$--semivisual, we have $\Om=\bigcup_{V\in\mc{V}_{\Om}}V$ and $H_e=\langle\Om_e\rangle$ for every edge $e\sq\mc{F}_{\Om}$. The other items in \Cref{setup_two} imply that the elements of $\mc{V}_{\Om}$ are precisely the maximal subsets of $\Om$ lying in $\mscr{C}$ (also see \Cref{rmk:maximal_V*}), and that any two sets in $\mc{E}_{\Om}$ are $H$--conjugate. 

For each $V\in\mc{V}_{\Om}$ and $E\in\mc{E}_{\Om}$, we define:
\begin{align*}
    \mc{E}_{\Om}(V)&:=\{E\in\mc{E}_{\Om}\mid E\sq V\}, &  {\rm blob}_{\Om}(E):=\hull_{\mc{T}}\{v\in\mc{F}_{\Om}\mid \Om_e\sq\Om_v\}.
\end{align*}
Since $\mc{T}$ is strongly reduced, each set $V\in\mc{V}_{\Om}$ fixes a unique point of $\mc{T}$ 
and so we have $V\setminus E\neq\emptyset$ for all $E\in\mc{E}_{\Om}(V)$. Each set ${\rm blob}_{\Om}(E)$ is a subtree of $\mc{F}_{\Om}$ and we refer to it as the \emph{$E$--blob}. The $E$--blob can be equivalently described as the union of all edges $e\sq\mc{F}_{\Om}$ with $\Om_e=E$. We also define the subset ${\rm Blob}_{\Om}(E)\sq\Om$ as the union of the sets $\Om_x$ with $x\in{\rm blob}_{\Om}(E)$.

As we vary the set $E\in\mc{E}_{\Om}$, the various $E$--blobs cover the tree $\mc{F}_{\Om}$. Distinct blobs share at most one point, since there are no proper inclusions between edge-stabilisers. The set $\mc{E}_{\Om}(V)$ is in $1$--to--$1$ correspondence with the set of blobs containing the only $V$--fixed vertex of $\mc{F}_{\Om}$.

\begin{rmk}\label{rmk:blob_intersections}
    All intersection points between blobs lie in $\mc{V}_{\Om}^*$. Indeed, suppose that we have $E\cup E'\sq\Om_v$ for a vertex $v\in\mc{F}_{\Om}$ and two distinct elements $E,E'\in\mc{E}_{\Om}$. The set $E\cup E'$ is elliptic, and so it is contained in a set $V\in\mc{V}_{\Om}$. At the same time, $E\cup E'$ does not fix any edges of $\mc{F}_{\Om}$, because $E\neq E'$ and there are no proper inclusions between edge-stabilisers of $\mc{T}$. Thus, $v$ is the only fixed point of $E\cup E'$ and we must have $V=\Om_v$. This shows that $v\in\mc{V}_{\Om}^*$.
\end{rmk}

\subsubsection{Neighbours and their tails}

Given $V\in\mc{V}_{\Om}$ and $E\in\mc{E}_{\Om}$, we define:
\begin{align*}
    {\rm Nbr}_{\Om}(V)&:=\{V'\in\mc{V}_{\Om}\mid V\cap V'\in\mc{E}_{\Om}\}, & {\rm Nbr}_{\Om}(V;E)&:=\{V'\in\mc{V}_{\Om}\mid E\sq V'\} .
\end{align*}
We refer to the elements of these sets are the \emph{neighbours} of $V$, and the neighbours of $V$ \emph{in the direction of $E$}, respectively. Note that we have a partition ${\rm Nbr}_{\Om}(V)=\bigsqcup_{E\in\mc{E}_{\Om}(V)}{\rm Nbr}_{\Om}(V;E)$. 

\begin{rmk}
    The term `neighbour' can be misleading: in general, $V\in\mc{V}_{\Om}$ and $W\in{\rm Nbr}_{\Om}(V)$ do \emph{not} fix adjacent vertices of $\mc{T}$. They just fix vertices of the same blob.
\end{rmk}

For a set $V\in\mc{V}_{\Om}$ and a neighbour\footnote{In this subsection, we only work with the group $H=\langle\Om\rangle$ and never need to refer to the ambient Coxeter group. Thus, the letter $W$ will typically denote elements of $\mc{V}_{\Om}$.} $W\in{\rm Nbr}_{\Om}(V)$, we define the subtree ${\rm tail}_{\Om}(V;W)\sq\mc{F}_{\Om}$ as the closure of the connected component of $\big(\mc{F}_{\Om}\setminus{\rm blob}_{\Om}(E)\big)\cup\{w\}$ containing $w$, where $w$ denotes the $W$--fixed vertex. As in the case of blobs, we denote by ${\rm Tail}_{\Om}(V;W)\sq\Om$ the union of the sets $\Om_x$ with $x\in{\rm tail}_{\Om}(V;W)$. In particular, we have ${\rm Tail}_{\Om}(V;W)\cap{\rm Blob}_{\Om}(V\cap W)=W$.

Fixing $V$ and varying $W\in{\rm Nbr}_{\Om}(V)$, the subtrees ${\rm tail}_{\Om}(V;W)$ and ${\rm blob}(V\cap W)$ cover $\mc{F}_{\Om}$. The sets ${\rm Tail}_{\Om}(V;W)$ together with $V$ cover $\Om$.

We will need the observations collected in the following lemma. Recall that a \emph{factor} of a subset $A\sq S$ is a subset $A_0\sq A$ such that $A\sq A_0\cup A_0^{\perp}$. An \emph{irreducible factor} is a factor that is irreducible and nonempty. The notation $\mc{I}_{\bullet}$ and $\mc{J}_{\bullet}$ was introduced at the start of \Cref{sect:markings}. (The reader interested only in the FC-type case can simply assume that $J\in\mc{I}_S$ in the following statement.)

\begin{lem}\label{lem:separation}
    Consider a set $V\in\mc{V}_{\Om}$ and a neighbour $W\in{\rm Nbr}_{\Om}(V)$. Then the following hold.
    \begin{enumerate}
        \item The set ${\rm Tail}_{\Om}(V;W)\setminus (V\cap W)$ is connected within the graph $\wh S_{\mscr{C}}$. 
        \item If $E\sq J\cup J^{\perp}$ for some $E\in\mc{E}_{\Om}(V)$ and some $J\in\mscr{C}$ with $J\sq\Om$, then $J\sq{\rm Blob}_{\Om}(E)$.
    \end{enumerate}
    If $V_0\sq V$ is an irreducible factor, we have the following additional facts.
    \begin{enumerate}
    \setcounter{enumi}{2}
        \item If $V_0\cup W$ intersects two distinct connected components of the graph $\wh S_{\mscr{C}}\setminus (J\cup J^{\perp})$ for some $J\in\mc{J}_S$ such that $J\setminus\Om$ is spherical, then we have $J\sq{\rm Blob}_{\Om}(V\cap W)$.
        \item If in the previous item we additionally have that $J\cap (V\setminus W)\neq\emptyset$, then $J\sq V_0$.
        \item If $W\cup U$ intersects two components of $\wh S_{\mscr{C}}\setminus (J\cup J^{\perp})$ for some $U\in{\rm Nbr}_{\Om}(V)$ and some $J\in\mc{J}_S$ with $J\setminus\Om$ spherical, then we have either $J\sq{\rm Blob}_{\Om}(V\cap W)$ or $J\sq{\rm Blob}_{\Om}(V\cap U)$.
    \end{enumerate}
\end{lem}
\begin{proof}
    Regarding Item~(1), consider two elements $\om,\om'\in{\rm Tail}_{\Om}(V;W)\setminus (V\cap W)$. Let $v,v'\in\mc{F}_{\Om}$ be the two vertices fixed by $\om$ and $\om'$ that are closest to each other. Let $\gamma\sq\mc{F}_{\Om}$ be the geodesic from $v$ to $v'$ and let $E_1,\dots,E_k$ be the distinct elements of $\mc{E}_{\Om}$ whose blobs share nontrivial arcs with $\gamma$, in order of appearance as we move from $v$ to $v'$. The sets $E_i\cup E_{i+1}$ are all elliptic in $\mc{T}$, as are the sets $\{\om\}\cup E_1$ and $E_k\cup\{\om'\}$. Thus, all of these sets lie in $\mscr{C}$ and span cliques in $\wh S_{\mscr{C}}$. Since the $E_i$ are all different from $V\cap W\in\mc{E}_{\Om}$, they each contain a point $\om_i\in E_i\setminus (V\cap W)$. In conclusion, $\om,\om_1,\dots,\om_k,\om'$ is a path in $\wh S_{\mscr{C}}\setminus (V\cap W)$ connecting $\om$ to $\om'$.

    For Item~(2), suppose that $J\not\sq{\rm Blob}_{\Om}(E)$. Since $J\in\mscr{C}$ and $J\sq\Om$, the set $J$ fixes a vertex $v\in\mc{F}_{\Om}$. Choose $v$ so that it is closest to ${\rm blob}_{\Om}(E)$, and let $w\in{\rm blob}_{\Om}(E)$ be the vertex closest to $v$. Let $f\sq\mc{F}_{\Om}$ be the edge incident to $w$ in the direction of $v$. We have ${\rm Blob}_{\Om}(E)\cap (J\cup J^{\perp})\sq\Om_f$. Since $E\neq\Om_f$, there is a point $\om\in E\setminus\Om_f$ and we have $\om\not\in J\cup J^{\perp}$. Thus $E\not\sq J\cup J^{\perp}$.

    We now discuss Items~(3) and~(4). Consider a set $J\in\mc{J}_S$ such that $J\setminus\Om$ is spherical and the union $V\cup W$ intersects two components of the graph $\wh S_{\mscr{C}}\setminus (J\cup J^{\perp})$. Observe that we must have $J\sq\Om$: this follows from the fact that $\Om$ is $\mscr{C}$--inflexible when $J$ is spherical, and from the fact that $\Om$ is $S$--tucked when $J$ is not. Since each point of $V\cap W$ is adjacent to all other points of $V\cup W$ in the graph $\wh S_{\mscr{C}}$, we must have $V\cap W\sq J\cup J^{\perp}$. Item~(2) then implies that $J\sq{\rm Blob}_{\Om}(V\cap W)$ as desired. If $J$ intersects $V\setminus W$, it is clear that we must have $J\sq V$. In this case, either $J\sq V_0$ or $J\sq V_0^{\perp}$, but in the latter case the set $(V_0\cup W)\setminus (J\cup J^{\perp})$ is actually contained in $W$, and so it spans a clique in $\wh S_{\mscr{C}}$, contradicting the fact that it is disconnected. 

    Finally, regarding Item~(5), we again have $J\sq\Om$ by inflexibility and tuckedness. As above, the union $J\cup J^{\perp}$ must contain either $V\cap W$ or $V\cap U$. Item~(2) then yields that $J$ is contained either in ${\rm Blob}_{\Om}(V\cap W)$, or in ${\rm Blob}_{\Om}(V\cap U)$ as desired.
\end{proof}

\subsubsection{Reference systems}

We now introduce a gadget that allows us to measure how far the set $\Om$ is from being $R$--geometric. A \emph{reference system} $\mf{R}$ for $\Om$ is the data of a $\langle V\rangle$--standard subcomplex $\mf{R}(V)\sq\A_R$ and a $V$--geometric point $p_V\in\mf{R}(V)$ for each set $V\in\mc{V}_{\Om}$ (see \Cref{subsub:standard} for definitions). Given $\mf{R}$, we denote by $\pi_V\colon\A_R\ra\mf{R}(V)$ the nearest-point projection.

\begin{defn}\label{defn:good/excellent}
    A pair $(V;W)$ with $V\in\mc{V}_{\Om}$ and $W\in{\rm Nbr}_{\Om}(V)$ is \emph{$\mf{R}$--good} if $p_V\in\pi_V(\mf{R}(W))$. It is \emph{$\mf{R}$--excellent} if $p_V=\pi_V(p_W)$. In this case, we also say that $W$ is an $\mf{R}$--good/excellent neighbour. 
    
    A set $E\in\mc{E}_{\Om}(V)$ is \emph{$\mf{R}$--good} if all pairs $(V;W)$ with $V,W\in\mc{V}_{\Om}$ and $V\cap W=E$ are $\mf{R}$--good. The reference system $\mf{R}$ is \emph{perfect} if all pairs $(V;W)$ with $V\in\mc{V}_{\Om}$ and $W\in{\rm Nbr}_{\Om}(V)$ are $\mf{R}$--excellent.
\end{defn}

In order to prove \Cref{thm:step_two_repeated}, we will have to modify $S$ and $\Om$ by performing elementary twists until $\Om$ is replaced by an $R$--geometric set $\Om'$. In view of the next lemma, our goal will actually be to ensure that $\Om'$ admits a perfect reference system.

\begin{lem}\label{lem:perfect->geometric}
    If $\Om$ admits a perfect reference system $\mf{R}$, then $\Om$ is $R$--geometric.
\end{lem}
\begin{proof}
    For each set $V\in\mc{V}_{\Om}$ and element $\om\in V$, let $\mc{H}_{\om}^V\sq\A_R$ be the side of the wall $\mc{W}_{\om}$ that contains the $V$--geometric point $p_V\in\mf{R}(V)$. 

    We claim that, for each $\om\in\Om$, the halfspace $\mc{H}_{\om}^V$ is independent of the choice of the set $V\in\mc{V}_{\Om}$ containing $\om$. For this, it suffices to show that, if $\om\in V\cap W$ for some $V\in\mc{V}_{\Om}$ and $W\in{\rm Nbr}_{\Om}(V)$, then $\mc{H}_{\om}^V=\mc{H}_{\om}^W$. In turn, the latter equality follows from the fact that we have $p_V=\pi_V(p_W)$.
    
    By the claim, we obtain a choice of halfspaces $\{\mc{H}_{\om}\}_{\om\in\Om}$ in $\A_R$. Since $\Om$ is contained in $S$, it is universal. Thus, by \Cref{lem:1.6}, it suffices to check that this halfspace choice is $2$--geometric in $\A_R$. For this, consider a pair of distinct elements $\om,\om'\in\Om$.

    Let $v,v'\in\mc{F}_{\Om}$ be the closest pair of vertices that are fixed by $\om$ and $\om'$, respectively. If $v=v'$, then $\{\om,\om'\}$ is contained in some set $V\in\mc{V}_{\Om}$, and the fact that the point $p_V$ is $V$--geometric implies that the halfspace pair $\{\mc{H}_{\om},\mc{H}_{\om'}\}$ is geometric. 
    
    Suppose instead that $v\neq v'$. In this case, we have $\{\om,\om'\}\not\in\mscr{C}$ and hence the walls $\mc{W}_{\om}$ and $\mc{W}_{\om'}$ are disjoint. We will complete the proof by showing that the wall $\mc{W}_{\om'}$ is contained in the halfspace $\mc{H}_{\om}$. (A symmetric argument shows that $\mc{W}_{\om}$ is contained in $\mc{H}_{\om'}$ as well.)
    
    We have $v,v'\in\mc{V}_{\Om}^*$ by \Cref{rmk:blob_intersections}, and so the sets $V:=\Om_v$ and $V'=\Om_{v'}$ lie in $\mc{V}_{\Om}$. Let $W\in{\rm Nbr}_{\Om}(V)$ be the neighbour for which $V'\sq{\rm Tail}_{\Om}(V;W)$. By \Cref{lem:separation}(1), there is a path $\gamma\sq \wh S_{\mscr{C}}\setminus (V\cap W)$ connecting $\om'$ to a point $\psi\in W\setminus V$. Since $\om\in V\setminus W$, the walls $\mc{W}_{\phi}$ with $\phi\in\gamma$ are all disjoint from $\mc{W}_{\om}$ and hence they are all on the same side of $\mc{W}_{\om}$ (for instance, because this corresponds to a repeated application of move (M1)$_{\mscr{C}}$ from \Cref{sub:moves}). Let $\mc{C}\sq\mf{R}(W)$ be the $\langle V\cap W\rangle$--standard subcomplex that contains the point $p_W$. The projection $\pi_V(\mc{C})$ is a $\langle V\cap W\rangle$--standard subcomplex of $\mf{R}(V)$, which is disjoint from $\mc{W}_{\om}$ since $\om\in V\setminus W$. Since $p_V=\pi_V(p_W)\in \pi_V(\mc{C})$, the halfspace $\mc{H}_{\om}=\mc{H}_{\om}^V$ contains the subcomplexes $\pi_V(\mc{C})$ and $\mc{C}$. The wall $\mc{W}_{\psi}$ is disjoint from $\mc{W}_{\om}$ and crosses an edge of $\mf{R}(W)$ incident to $p_W\in\mc{C}$. This shows that $\mc{W}_{\psi}\sq\mc{H}_{\om}$ and hence $\mc{W}_{\om'}\sq\mc{H}_{\om}$ as desired, concluding the proof of the lemma.
\end{proof}

\subsubsection{Separating and non-separating sets}

Given $V\in\mc{V}_{\Om}$ and $W\in{\rm Nbr}_{\Om}(V)$, we define:
\begin{align*}
    \mc{N}(V,W)&:=\{J\sq S \text{ irreducible} \mid \text{$V\cup W$ intersects $\leq 1$ components of $\wh S_{\mscr{C}}\setminus (J\cup J^{\perp})$} \} , \\
    \mc{S}(V,W)&:=\{J\sq S \text{ irreducible} \mid \text{$V\cup W$ intersects $2$ components of $\wh S_{\mscr{C}}\setminus (J\cup J^{\perp})$} \} .
\end{align*}
We refer to the elements of $\mc{N}(V,W)$ and $\mc{S}(V,W)$ as \emph{non-separating} and \emph{separating} sets, respectively. We denote by $\mc{N}_{\rm sph}(V,W)$ and $\mc{S}_{\rm sph}(V,W)$ the subsets of spherical elements, so that:
\[ \mc{I}_S=\mc{N}_{\rm sph}(V,W)\sqcup\mc{S}_{\rm sph}(V,W) .\]
For $J\in\mc{S}(V,W)$, we also define:
\[ {\rm Comp}_{\Om}(V;W,J):=\{U\in{\rm Nbr}_{\Om}(V)\mid U\not\sq J\cup J^{\perp} \text{ and } J\in\mc{N}(W,U) \} . \]
In other words, the sets $U\in{\rm Comp}_{\Om}(V;W,J)$ are the neighbours of $V$ that intersect the same component of $\wh S_{\mscr{C}}\setminus (J\cup J^{\perp})$ as $W$. Note that each set in $\mc{V}_{\Om}$ spans a clique in the graph $\wh S_{\mscr{C}}$ and hence intersects at most one component of $\wh S_{\mscr{C}}\setminus (J\cup J^{\perp})$. Also note that we have $J\in\mc{S}(V,U)$ for every $U\in{\rm Comp}_{\Om}(V;W,J)$.

As mentioned, the proof of \Cref{thm:step_two_repeated} will require modifying $\Om$ and $S$ by elementary twists. We now make a few preliminary observations about this twisting procedure.

\begin{rmk}[Performing a twist]\label{rmk:performing_twist}
    Consider a set $V\in\mc{V}_{\Om}$, a neighbour $W\in{\rm Nbr}_{\Om}(V)$, and a set $K\in\mc{S}_{\rm sph}(V,W)$ with $K\sq V$. We can then modify $S$ and $\Om$ by an elementary twist with respect to $K$ relative to $\mscr{C}$. More precisely, there exists a generating set $S'\in\mscr{T}_S$ containing the set 
    \[ \Om':=V\cup\bigcup_{X\in\mc{X}}{\rm Tail}(V;X)\cup\bigcup_{Y\in\mc{Y}} w_K\big({\rm Tail}(V;Y)\big)w_K ,\]
    where $\mc{X}:={\rm Nbr}_{\Om}(V)\setminus{\rm Comp}_{\Om}(V;W,K)$ and $\mc{Y}:={\rm Comp}_{\Om}(V;W,K)$. In particular, we have $w_KWw_K\sq\Om'$. (Here the fact that $\Om'$ is a union of tails is a consequence of \Cref{lem:separation}.) For future reference, we say that $S'$ and $\Om'$ are obtained by a \emph{$(V;W,K)$--twist}.
\end{rmk}

A fundamental observation is that the set $\Om'$ constructed in \Cref{rmk:performing_twist} still fits in the framework of this subsection. This is the content of the next lemma and it is precisely why we are forced to work with semivisual splittings instead of visual ones.

\begin{lem}\label{lem:semivisual_stays}
    Consider $V\in\mc{V}_{\Om}$, $W\in{\rm Nbr}_{\Om}(V)$, and $K\in\mc{S}_{\rm sph}(V,W)$ with $K\sq V$. If $\Om'$ is obtained by a $(V;W,K)$--twist, then we have $\langle\Om\rangle=\langle\Om'\rangle$ and the splitting $H\acts\mc{T}$ is $\Om'$--semivisual. 
\end{lem}
\begin{proof}
    The fact that $\langle\Om\rangle=\langle\Om'\rangle$ is clear, since $K\sq\Om$. To check that $\mc{T}$ is still $\Om'$--semivisual, define $\mc{X}$ and $\mc{Y}$ as in \Cref{rmk:performing_twist}. For each neighbour $U\in{\rm Nbr}_{\Om}(V)$, write $\tau_U:={\rm tail}_{\Om}(V;U)\sq\mc{F}_{\Om}$ for simplicity. Recall that $\mc{F}_{\Om}$ is covered by the blobs containing $v$ together with the pairwise-disjoint subtrees $\tau_U$ with $U\in{\rm Nbr}_{\Om}(V)$. Now, define $\mc{V}_{\Om'}^*$ as the set of points $u\in\mc{T}$ such that:
    \begin{itemize}
        \item either $u=v$;
        \item or $u\in\mc{V}_{\Om}^*\cap\tau_X$ for some $X\in\mc{X}$;
        \item or $u=w_Ky$ for some $y\in\mc{V}_{\Om}^*\cap\tau_Y$ with $Y\in\mc{Y}$.
    \end{itemize}
    For each $u\in\mc{V}_{\Om'}^*$, define $\Om_u':=\Om_u$ in the first two cases and $\Om_u':=w_K\Om_yw_K$ in the third. 
    
    It is clear that we have $\Om_u'\sq\Om'$ and $H_u=\langle\Om'_u\rangle$ for each $u\in\mc{V}_{\Om'}^*$. Note that a subset of $\Om'$ is elliptic in $\mc{T}$ if and only if it lies in $\mscr{C}$, in which case it must be contained either in $V$, or in some ${\rm Tail}(V;X)$ with $X\in\mc{X}$, or again in some $w_K\big({\rm Tail}(V;Y)\big)w_K$ with $Y\in\mc{Y}$. Thus, every elliptic subset of $\Om'$ is contained in $\Om_u'$ for some $u\in\mc{V}_{\Om'}^*$, and in particular the sets $\Om_u'$ cover $\Om'$.
    
    Let $\mc{F}_{\Om'}$ be the convex hull of $\mc{V}_{\Om'}^*$ in $\mc{T}$, and observe that this is the $\Om'$--core of $\mc{T}$. We are left to check that the $H$--stabiliser of every edge of $\mc{F}_{\Om'}$ is generated by a subset of $\Om'$. For an edge $e\sq\mc{F}_{\Om'}$, we have the following four cases to check. 
    \begin{itemize}
        \item If $e\sq\tau_X$ for some $X\in\mc{X}$, then $e\sq\mc{F}_{\Om}$ and $H_e=\langle\Om_e\rangle$, where $\Om_e\sq{\rm Tail}_{\Om}(V;X)\sq\Om'$.
        \item If $e\sq w_K\tau_Y$ for some $Y\in\mc{Y}$, then $f:=w_Ke\sq\mc{F}_{\Om}$ and $H_e=\langle w_K\Om_fw_K\rangle$, where we have $w_K\Om_fw_K\sq w_K\big({\rm Tail}_{\Om}(V;Y)\big)w_K\sq\Om'$.
        \item If $e$ separates $v$ from some $\tau_X$ with $X\in\mc{X}$, then $e\sq\Sigma\sq\mc{F}_{\Om}$ and $H_e=\langle\Om_e\rangle$ with $\Om_e\sq V$.
        \item Finally, if $e$ separates $v$ from some $w_K\tau_Y$ with $Y\in\mc{Y}$, then $f:=w_Ke$ is contained in $\Sigma\sq\mc{F}_{\Om}$ and separates $v$ from $\tau_Y$ (note that $w_K$ fixes $v$). In particular, we have $H_e=\langle w_K\Om_fw_K\rangle$ and again $w_K\Om_fw_K\sq w_K\big({\rm Tail}_{\Om}(V;Y)\big)w_K\sq\Om'$.
    \end{itemize}
    This concludes the proof of the lemma.
\end{proof}

\begin{rmk}\label{rmk:neighbour-preserving}
    If $\Om'$ is obtained from $\Om$ by a $(V;W,K)$--twist, the proof of \Cref{lem:semivisual_stays} yields a bijection $\phi\colon\mc{V}_{\Om}\ra\mc{V}_{\Om'}$ with $\phi(V)=V$. If $U_1,U_2\in\mc{V}_{\Om}$ are neighbours, then so are $\phi(U_1)$ and $\phi(U_2)$. 
\end{rmk}

\subsubsection{The shortening theorem}

We are finally ready to state the shortening theorem, whose proof will occupy the next subsection. This is the main step in the proof of \Cref{thm:step_two_repeated}.

Recall that, given a reference system $\mf{R}$ and a set $V\in\mc{V}_{\Om}$, we denote by $\pi_V\colon\A_R\ra\mf{R}(V)$ the nearest-point projection. Also recall that a factor of a subset $A\sq S$ is any subset $A_0\sq A$ such that $A\sq A_0\cup A_0^{\perp}$; in particular, $A_0$ may be empty or reducible. 

\begin{thm}\label{thm:shortening}
    Let $\Om$, $H$ and $\mc{T}$ be as in \Cref{setup_two}. Let $\mf{R}$ be a reference system for $\Om$. Consider a set $V\in\mc{V}_{\Om}$ and some $W\in{\rm Nbr}_{\Om}(V)$. Then there exist $S'\in\mscr{T}_S$ and $\Om'\sq S'$ satisfying the following.
    \begin{enumerate}
        \setlength\itemsep{.25em}
        \item We have $\langle\Om'\rangle=\langle\Om\rangle$ and the splitting $\mc{T}$ is $\Om'$--semivisual.
        \item There is a neighbour-preserving bijection $\phi\colon\mc{V}_{\Om}\ra\mc{V}_{\Om'}$ with $\phi(V)=V$. For each $N\in{\rm Nbr}_{\Om}(V)$, there is $g_N\in\langle V\rangle$ such that $\phi(U)=g_NUg_N^{-1}$ for all $U\in\mc{V}_{\Om}$ with $U\sq{\rm Tail}_{\Om}(V;N)$.
        \item If $W$ is not an $\mf{R}$--good neighbour of $V$, then there is a $V$--geometric point $q\in\mf{R}(V)$ with
        \[ d(q,g_W\pi_V(\mf{R}(W)))<d(p_V,\pi_V(\mf{R}(W))) ,\]
        and such that, for every $\mf{R}$--good neighbour $X\in{\rm Nbr}_{\Om}(V)$, we have $q\in g_X\pi_V(\mf{R}(X))$.
        \item Suppose that all pairs of neighbouring sets in $\mc{V}_{\Om}$ are $\mf{R}$--good. If $W$ is not an $\mf{R}$--excellent neighbour of $V$, then we have 
        \[ d(p_V,g_W\pi_V(p_W))<d(p_V,\pi_V(p_W)) .\]
        In addition, for every $\mf{R}$--excellent neighbour $X\in{\rm Nbr}_{\Om}(V)$, we have $g_X=1$.
    \end{enumerate}
\end{thm}

The sets $S'$ and $\Om'$ in \Cref{thm:shortening} are obtained by performing twists with respect to subsets in $\mc{I}_V$, but we emphasise that \emph{finitely many} such twists are usually required (not a single one).

\subsection{Proof of the shortening theorem}\label{sub:proof_shortening}

We now start working towards the proof of \Cref{thm:shortening}. Let $\Om$ and $\mc{T}$ be as in the statement of the theorem. 

The plan for this subsection is as follows. First, we obtain some tools allowing us to apply the results on marking equivalences from \Cref{sect:markings} (from \Cref{defn:delta} to \Cref{lem:C-admissible_equivalence}). Second, we introduce the concept of an \emph{atom} and study it (from \Cref{defn:atom} to \Cref{prop:atoms_2}). Atoms are important because they will make it clear which sequences of elementary twists we need to perform in order bring $\Om$ closer to being $R$--geometric or, more precisely, to bring a reference system closer to being perfect. Third, we obtain two lemmas showing that these twists do not spoil any good or excellent portions of our reference systems (Lemmas~\ref{lem:sep_good} and~\ref{lem:sep_excellent}). Finally, we prove the shortening theorem and, at the very end, also \Cref{thm:step_two_repeated}.

As in the previous subsection, the reader may wish to keep in mind the simplified setting in which $\mc{F}_{\Om}$ is a single edge. This will not make much of a difference in the beginning, but it will remove some technicalities and notational complexity from the later discussion (in particular, Lemmas~\ref{lem:sep_good} and~\ref{lem:sep_excellent} can be ignored in this case). 

Throughout, we fix a reference system $\mf{R}$ for $\Om$, a set $V\in\mc{V}_{\Om}$, and a neighbour $W\in{\rm Nbr}_{\Om}(V)$. For every subset $M\sq V$, we denote by $\mf{R}(M)$ the $\langle M\rangle$--standard subcomplex of $\mf{R}(V)$ containing the point $p_V$, and by $\pi_M\colon\A_R\ra\mf{R}(M)$ the nearest-point projection. We also fix an irreducible factor $V_0\sq V$ and write $p_{V_0}:=p_V$, since $p_V\in\mf{R}(V_0)$ by definition.

One should imagine that $p_{V_0}\not\in\pi_{V_0}(\mf{R}(W))$ and that our goal is to perform twists with respect to sets in $\mc{I}_{V_0}$ so as to reduce the distance $d(p_{V_0},\pi_{V_0}(\mf{R}(W)))$. However, we do not make any assumptions on the relative position of $p_{V_0}$ and $\pi_{V_0}(\mf{R}(W))$ for the moment, since important parts of \Cref{thm:shortening} are not about this situation.

To streamline notation, it is convenient to rename the above objects as $A:=V_0$ and $B:=W$ in the following discussion. We set $C:=A\cap B$ and $\mc{C}:=\pi_A(\mf{R}(B))$, which will be the most important subcomplex of $\A_R$ for our argument. Note that $\mc{C}$ is a $\langle C\rangle$--standard subcomplex of $\A_R$. The set $A$ is irreducible, while $B$ and $C$ may not be. We also have an $A$--geometric point $p_A\in\mf{R}(A)$ chosen by the reference system. We simply write $\mc{N}:=\mc{N}(A,B)$ and $\mc{S}:=\mc{S}(A,B)$, and similarly define $\mc{N}_{\rm sph}$ and $\mc{S}_{\rm sph}$.
    
For each irreducible subset $J\sq A$, we denote by $\wh J$ 
the union of $J$ with all irreducible factors of $C$ that intersect $A\setminus J^{\perp}$.
Note that $\wh J$ is still an irreducible subset of $A$. We always have $C\sq \wh J\cup\wh J^{\perp}$, and we have $\wh J=J$ if and only if $C\sq J\cup J^{\perp}$. In particular, recalling that $A$ and $B$ span cliques in $\wh S_{\mscr{C}}$, we can have $J\in\mc{S}$ only if $\wh J=J$.

Another fundamental object is the following.

\begin{defn}\label{defn:delta}
    For a subset $J\sq S$, denote by $\delta_A J\sq A$ the set of points that occur as the \emph{unique} intersection point with $A$ of some geodesic in $\wh S_{\mscr{C}}\setminus (J\cup J^{\perp})$ that intersects $B$. 
\end{defn}

The word ``unique'' in \Cref{defn:delta} is essential, as it is what will allow us to exploit inflexibility and tuckedness of $\Om$ to place ourselves in a situation where the hypotheses of \Cref{prop:markings_new} are satisfied (see \Cref{lem:C-admissible_equivalence} below). 

\begin{rmk}\label{rmk:nonempty_delta}
    For every $J\in\mc{N}$ with $A\not\sq J\cup J^{\perp}$ and $B\not\sq J\cup J^{\perp}$, we have $\delta_A J\neq\emptyset$. Indeed, $\delta_A J$ contains any point of $A\setminus (J\cup J^{\perp})$ that is nearest to $B$ within $\wh S_{\mscr{C}}\setminus (J\cup J^{\perp})$.
\end{rmk} 

If $J\sq A$ and $\wh J=J$, we necessarily have $\delta_A J\sq A\setminus C$. We also have the following key observation.

\begin{lem}\label{lem:delta_works_2}
    Consider two irreducible subsets $N\sq J$ such that $J\sq A$ and $N\not\sq C$. Suppose $N\in\mc{N}$ and $J\in\mc{S}$. Then we have $J\cap\delta_A N\neq\emptyset$.
\end{lem}
\begin{proof}
    To begin with, observe that $B\not\sq N\cup N^{\perp}$. Indeed, since $B\setminus C\neq\emptyset$, there exists an irreducible factor $B_0\sq C$ with $B_0\setminus C\neq\emptyset$ and hence $B_0\not\sq N$ (since $N\sq A$). We also have $B_0\not\sq N^{\perp}$, since no point of $B_0\setminus C$ is adjacent to a point of $N\setminus C\sq A\setminus C$ within the graph $\wh S_{\mscr{C}}$. In conclusion, since $B_0$ is irreducible, it follows that $B_0\not\sq N\cup N^{\perp}$ as desired.
    
    Now, choose an element $b\in B\setminus (N\cup N^{\perp})$. Since $J$ is irreducible and $J\neq N$, we have $J\not\sq N\cup N^{\perp}$. Since $J\sq A$ and $N\in\mc{N}$, there is a path $\gamma\sq\wh S_{\mscr{C}}\setminus (N\cup N^{\perp})$ from $b$ to $J\setminus (N\cup N^{\perp})$. Choose $\gamma$ so that it is a shortest such path; in particular, $\gamma$ is a geodesic. Let $a\in J$ be the terminal endpoint of $\gamma$, and let $u$ be the point of $\gamma$ immediately preceding $a$. Since $A\in\mscr{C}$, the set $A$ spans a clique in $\wh S_{\mscr{C}}$. The fact that $\gamma$ is shortest from $b$ to $J$ then implies that $\gamma\cap A\sq\{a,u\}$. If $u\not\in A$, then $a\in J\cap\delta_A N$ and we are done. 
    
    If instead $u\in A$, we reach a contradiction. Indeed, we have $u\in A\setminus J$ because $\gamma$ is shortest. In this case, the path $\gamma\setminus\{a\}$ connects $b$ to $u\in A$ while avoiding the union $J\cup J^{\perp}$, since $J^{\perp}\sq N^{\perp}$ and $J\sq A\setminus\{u\}$. This contradicts the fact that $J\in\mc{S}$, proving the lemma. 
\end{proof}

Recall that $\mc{W}_s$ denotes the $s$--fixed wall in $\A_R$, for each $s\in S$. For $a\in A$, we denote by $\mc{H}_a$ the side of $\mc{W}_a$ that contains the $A$--geometric point $p_A\in\mf{R}(A)$, and by $\mc{H}_a^*$ the other side of $\mc{W}_a$. Note that the intersection of all halfspaces $\mc{H}_a$ with $a\in A$ is the singleton $\{p_A\}$. The intersection of all halfspaces $\mc{H}_a^*$ is empty if $A$ is non-spherical, while it equals the singleton $\{w_Ap_A\}$ if $A$ is spherical, where $w_A$ denotes as usual the longest element of $\langle A\rangle$. We also define the set
\[ \Delta(\mc{C}):=\{a\in A\setminus C\mid \mc{C}\sq\mc{H}_a^* \}. \] 

\begin{rmk}
    An element $a\in A\setminus C$ lies in $\Delta(\mc{C})$ if and only if we have $\mc{W}_b\sq\mc{H}_a^*$ for some (or equivalently, all) elements $b\in B\setminus C$. Indeed, we have $a\not\in\langle B\rangle$ (since $A\cup B$ is contained in a Coxeter generating set) and so the wall $\mc{W}_a$ is disjoint from $\mf{R}(B)$ and from its projection $\mc{C}\sq\mf{R}(A)$. The walls $\mc{W}_b$ all cross edges incident to the $B$--geometric point $p_B\in\mf{R}(B)$. Therefore, the point $\pi_A(p_B)$ is on the same side of $\mc{W}_a$ as the walls $\mc{W}_b$ and, since $\pi_A(p_B)\in\mc{C}$, the subcomplex $\mc{C}$ is on that side as well.
\end{rmk}

\begin{rmk}\label{rmk:empty_F}
    If $\Delta(\mc{C})=\emptyset$, then $p_A\in\mc{C}$. Indeed, we have $\mc{C}\sq\mc{H}_a$ for all $a\in A\setminus C$. Since $\mc{C}$ is a $\langle C\rangle$--standard subcomplex and $C\in\mscr{C}$, there exists a point $p\in\mc{C}$ with $p\in\mc{H}_c$ for all $c\in C$. In conclusion, we have $p\in\mc{H}_a$ for all $a\in A$, which implies that $p=p_A$.
\end{rmk}

We say that a base $(a,w)$ with $a\in A$ and $w\in\langle A\rangle$ is \emph{$\mc{C}$--independent} if $\supp(a,w)\not\sq C$. Equivalently, we have $w^{-1}aw\not\in\langle C\rangle$, 
that is, the wall $w^{-1}\mc{W}_a\sq\A_R$ does not intersect the subcomplex $\mc{C}$. Admissibility implies that we have $w^{-1}\mc{W}_a\cap\mc{W}_b=\emptyset$ for all $b\in B\setminus C$,
and the union $\mc{C}\cup\mc{W}_b$ is contained in a single halfspace bounded by $w^{-1}\mc{W}_a$. Given two $\mc{C}$--independent bases $(a,w),(a,w')$ with the same core, we write 
\[ (a,w)\equiv_{\mc{C}}(a,w') \] 
if the complexes $w\mc{C}$ and $w'\mc{C}$ are on the same side of the wall $\mc{W}_a$; equivalently, $w\mc{W}_b$ and $w'\mc{W}_b$ are on the same side of $\mc{W}_a$ for some/all $b\in B\setminus C$. The results of \Cref{sect:markings} now yield the following:

\begin{lem}\label{lem:C-admissible_equivalence}
    Let $(a,w)$ be a $\mc{C}$--independent base with support $\Sigma\sq A$. Consider an irreducible subset $K\sq\Sigma$ with $K\not\sq C$ and $\wh K=K$. Then, for every element $x\in\delta_A K$, the following holds. 
    \begin{enumerate}
        \item If $(a,wx)$ is a $\mc{C}$--independent base, 
        then $(a,w)\equiv_{\mc{C}} (a,wx)$.
        \item If the subgroup $\langle a,wxw^{-1}\rangle$ is infinite, then $\mc{C}\sq w^{-1}\mc{H}_a$.
    \end{enumerate}
\end{lem}
\begin{proof}
    Since $x\in\delta_A K$, there exists a geodesic $\gamma\sq\wh S_{\mscr{C}}\setminus (K\cup K^{\perp})$ from $x$ to a point $b\in B$ with $\gamma\cap A=\{x\}$. Since $\wh K=K$, we have $b\in B\setminus C$ and $x\in A\setminus C$, and hence $\{x,b\}\not\in\mscr{C}$.

    We would like to apply \Cref{prop:markings_new} to the geodesic $\gamma$, for which we need to check that the pair $(K,\gamma)$ is strongly inseparable. Thus, suppose that we have $K\sq J$ for some $J\in\mc{J}_S$ such that $J\cap\gamma\neq\emptyset$ and $J\setminus K$ is spherical. Observe that $J$ intersects the set $A\setminus C$, because $K$ does. 
    Thus, if $x$ and $b$ were to lie in distinct components of $\wh S_{\mscr{C}}\setminus (J\cup J^{\perp})$, then \Cref{lem:separation}(4) would imply that $J\sq A$. However, this would contradict the fact that $\gamma\cap A=\{x\}$ and $x\not\in J$.

    Now, in the notation of \Cref{sect:markings}, the two items of the lemma can be rephrased as claiming that
    \[ ((a,w),b)\equiv_R ((a,wx),b) \qquad\text{and}\qquad ((a,w),b)\equiv_R ((a,w),x),\]
    respectively. These markings are complete under the respective assumptions of the two items of the lemma, and so these equivalences follow from Items~(2) and~(1) of \Cref{prop:markings_new}, respectively.
\end{proof}

We now come to the most important concept of this entire subsection, namely that of an \emph{atom}. As we will see, atoms will yield some ``obvious'' elementary twists that we need to perform.

\begin{defn}\label{defn:atom}
    An \emph{atom} is a minimal irreducible subset $T\sq A$ such that $T\cap \Delta(\mc{C})\neq\emptyset$ and $T\in\mc{S}\cup\{A\}$.
\end{defn}

Note that every atom $T$ satisfies $\wh T=T$. In order to study atoms, we are in turn led to the concept of a ``progeny'', which we now discuss.

Consider an irreducible subset $U_0\sq A$ with $U_0\not\sq C$. A \emph{progeny} of $U_0$ is a (possibly trivial) sequence $U_0,\dots,U_k\sq A$ of irreducible sets satisfying the following conditions for $0\leq i\leq k-1$:
\begin{enumerate}
    \item we have $U_{i+1}=U_i\cup\{u_{i+1}\}$ for an element $u_{i+1}\in A\setminus (U_i\cup U_i^{\perp})$;
    \item if $\wh U_i\neq U_i$, then $u_{i+1}\in C$;
    \item if $\wh U_i = U_i$, then $u_{i+1}\in\delta_A U_i$.
\end{enumerate}
A progeny is \emph{$A$--maximal} (or simply \emph{maximal}) if it cannot be extended, that is, if either $U_k=A$, or $\wh U_k=U_k$ and $\delta_AU_k=\emptyset$. The latter condition is equivalent to the fact that $U_k\in\mc{S}$, by \Cref{rmk:nonempty_delta}. Note that every irreducible subset $U_0\sq A$ with $U_0\not\sq C$ is the starting point of at least one maximal progeny (possibly having $k=0$).

If $T$ is an atom, we speak of \emph{$T$--progenies} when referring to progenies consisting of subsets of $T$. A $T$--progeny is \emph{$T$--maximal} if it cannot be extended to a longer $T$--progeny. 

In order to state the next results, we need one last general piece of notation.
If $U\sq T$ are tree--$2$--spherical subsets of $S$, we define the element $g_{T,U}\in\langle T\setminus U\rangle$ as the product $x_1\dots x_k$ where $x_1,\dots,x_k$ is any ordering of the elements of $T\setminus U$ with the property that we have $i<j$ whenever $x_i$ separates $x_j$ from $U$ in the Dynkin diagram of $T$. Since the Dynkin diagram of $T$ is a tree and $U$ is a subtree, the definition of $g_{T,U}$ is independent of the particular choice of an ordering on $T\setminus U$. Given three irreducible sets $V\sq U\sq T$, we have $g_{T,V}=g_{U,V}g_{T,U}$. When $U=\{u\}$, we write $g_{T,u}$ in place of $g_{T,\{u\}}$.

The next three results seek to understand the structure of atoms, and they will culminate in \Cref{prop:atoms_2}, which describes them rather precisely. We will see in \Cref{prop:atoms_2} that all atoms are spherical, but for now we need to content ourselves with partial results in this direction.

\begin{lem}\label{lem:atoms_1}
    If $T$ is an atom, then all the following hold.
    \begin{enumerate}
        \item For every $T$--maximal $T$--progeny $U_0,\dots,U_k$ with $U_0\cap \Delta(\mc{C})\neq\emptyset$, we have $U_k=T$.
        \item The set $T$ is tree--$2$--spherical.
        \item For every $\mc{C}$--independent base $(a,w)$ with $\Sigma:=\supp(s,w)$ contained in $T$ and meeting $\Delta(\mc{C})$, we have $(a,w)\equiv_{\mc{C}}(a,wg_{Q,\Sigma})$ for all irreducible sets $Q$ with $\Sigma\sq Q\sq T$.
        \item We have $(f,w)\equiv_{\mc{C}}(f,w')$ for every pair of $\mc{C}$--independent bases with $f\in T\cap \Delta(\mc{C})$ and $w,w'\in\langle T\setminus\{f\}\rangle$.
        \item For every element $f\in T\cap\Delta(\mc{C})$, the difference $T\setminus\{f\}$ is spherical.
    \end{enumerate}
\end{lem}
\begin{proof}
    We begin with Item~(1). Consider a $T$--maximal $T$--progeny $U_0,\dots,U_k$ with $U_0\cap \Delta(\mc{C})\neq\emptyset$. Since $\wh U_k\sq\wh T=T$, we must have $\wh U_k=U_k$, as otherwise we could extend the $T$--progeny by adding to $U_k$ an element of $\wh U_k\setminus U_k$. We must also have $U_k\in\mc{S}\cup\{T\}$, as otherwise $T\cap\delta_AU_k$ would be nonempty by \Cref{lem:delta_works_2}, 
    and so we could again extend the $T$--progeny by adding to $U_k$ an element of $T\cap\delta_AU_k$. Now, since $T$ is an atom, it follows that $U_k=T$, proving Item~(1).

    Regarding Item~(2), consider an element $f\in T\cap \Delta(\mc{C})$ and a $T$--maximal $T$--progeny $U_0,\dots,U_k$ of $U_0=\{f\}$. Write $U_{i+1}=U_i\cup\{u_{i+1}\}$ for some elements $u_{i+1}$, and suppose for the sake of contradiction that there exists an index $m\leq k$ such that $U_m$ is not tree--$2$--spherical. Observe that we have $(f,u_1\dots u_{i-1})\equiv_{\mc{C}}(f,u_1\dots u_i)$ for all indices $1\leq i\leq m-1$: if $u_i\in C$, this is simply an application of move (M2)$_{\mscr{C}}$ from \Cref{sub:moves}, since $B$ spans a clique in $\wh S_{\mscr{C}}$; if instead $u_i\in\delta_AU_{i-1}$, this follows from \Cref{lem:C-admissible_equivalence}(1). Combining these equivalences, we obtain $(f,1)\equiv_{\mc{C}}(f,u_1\dots u_{m-1})$ and, since $\mc{C}\sq\mc{H}_f^*$, this means that $\mc{C}\sq u_{m-1}\dots u_1\mc{H}_f^*$. At the same time, since $U_m$ is not tree--$2$--spherical, the subgroup $\langle f,u_1\dots u_{m-1}u_mu_{m-1}\dots u_1\rangle$ is infinite (e.g.\ by \Cref{rmk:tree-2-spherical_support}), and hence \Cref{lem:C-admissible_equivalence}(2) implies that $\mc{C}\sq u_{m-1}\dots u_1\mc{H}_f$, which is the required contradiction.

    Regarding Item~(3), consider a $\mc{C}$--independent base $(a,w)$ such that $\Sigma:=\supp(a,w)$ is contained in $T$ and intersects $\Delta(\mc{C})$. Let $\Sigma=:U_0,\dots,U_k=T$ be a $T$--maximal $T$--progeny of $\Sigma$. Defining the elements $u_i$ as above, move (M2)$_{\mscr{C}}$ and \Cref{lem:C-admissible_equivalence}(1) again yield 
    \[ (a,w)\equiv_{\mc{C}}(a,wu_1\dots u_k)=(a,wg_{T,\Sigma}) .\] 
    (Here we are using Item~(2) and \Cref{rmk:tree-2-spherical_support} to conclude that all of these pairs truly are bases.)
    Now, if we have $\Sigma\sq Q\sq T$ for an irreducible set $Q$, a double application of this observation gives
    \[ (a,w)\equiv_{\mc{C}}(a,wg_{T,\Sigma})=(a,wg_{Q,\Sigma}g_{T,Q})\equiv_{\mc{C}}(a,wg_{Q,\Sigma}). \]
    
    We now discuss Item~(4). It suffices to show that $(f,1)\equiv_{\mc{C}}(f,w)$ for every $\mc{C}$--independent base $(f,w)$ with $f\in T\cap \Delta(\mc{C})$ and $w\in\langle T\setminus\{f\}\rangle$. Write $w$ as a reduced word $x_1\dots x_n$ for some (not necessarily distinct) elements $x_i\in T\setminus\{f\}$. Set $w_i:=x_1\dots x_i$ and let us show that $(f,w_i)\equiv_{\mc{C}}(f,w_{i+1})$ for each $i$.
    
    By Item~(1) applied with $U_0=\{f\}$, we know that every point $x\in T\setminus\{f\}$ admits a set $N(x)\sq T$ such that $f\in N(x)$ and $x\in\delta_A N(x)$. Set $K:=\supp(f,w_i)$, $N:=N(x_{i+1})$ and $L:=N\cup K$. Since $f\in K\cap N$, the set $L$ is irreducible. The fact that $f$ lies in $K\cap N$ also implies that $x_{i+1}$ commutes with all elements of $N\setminus K=L\setminus K$: indeed, since the Dynkin diagram of $T$ is a tree by Item~(2), the only Dynkin-neighbour of $x_{i+1}$ that lies in $N$ is the neighbour separating $x_{i+1}$ from $f$ (or possibly $f$ itself), and this neighbour also lies in $K$.
    Consequently, setting $K':=K\cup\{x_{i+1}\}$ and $L':=L\cup\{x_{i+1}\}$, we obtain
    \[ (f,w_i)\equiv_{\mc{C}} (f,w_ig_{L,K}) \equiv_{\mc{C}} (f,w_ig_{L,K}x_{i+1}) = (f,w_{i+1}g_{L,K}) = (f,w_{i+1}g_{L',K'}) \equiv_{\mc{C}} (f,w_{i+1}) , \]
    where the first and last equivalences follow from Item~(3); the second equivalence uses \Cref{lem:C-admissible_equivalence}(1) and the fact that $L\supseteq N$ and $x_{i+1}\in\delta_A N$. This proves Item~(4).

    Finally, we deal with Item~(5). Suppose for the sake of contradiction that $T\setminus\{f\}$ is non-spherical for some $f\in T\cap\Delta(\mc{C})$. Let $w\in\langle T\setminus\{f\}\rangle$ be a longest element such that the pair $(f,w)$ is a base; this exists by \Cref{rmk:PW}. Since $T$ is tree--$2$--spherical, we have $\supp(f,w)=T$, so that $(f,w)$ is $\mc{C}$--independent. Now, since $T\setminus\{f\}$ is non-spherical, there exists an element $u\in T\setminus\{f\}$ such that $((f,w),u)$ is a complete marking: this follows from \cite[Lemma~8.2]{CP10} and \Cref{rmk:complete_or_base}. Then, considering the set $N(u)$ defined in the proof of Item~(4) and invoking \Cref{lem:C-admissible_equivalence}(2) with $K:=N(u)$, we obtain that $\mc{C}\sq w^{-1}\mc{H}_f$. At the same time, Item~(4) yields $(f,1)\equiv_{\mc{C}}(f,w)$ and hence $\mc{C}\sq w^{-1}\mc{H}_f^*$. This is a contradiction, concluding the proof of Item~(5).
\end{proof}

We will also need the following technical result. Recall that, for each subset $M\sq A$, we denote by $\pi_M$ the nearest-point projection to $\mf{R}(M)$. We say that two elements of $S$ are \emph{Dynkin-adjacent} if they are distinct and commute.

\begin{lem}\label{lem:atoms_5}
 Let $T$ be a spherical atom. If $K\subsetneq T$ is an irreducible subset with $K\not\sq C$ and $\wh K=K\in\mc{N}_{\rm sph}$ and $K\cap \Delta(\mc{C})=\emptyset$, then some element of $K$ is Dynkin-adjacent to an element of $T\setminus (K\cup C\cup \Delta(\mc{C}))$.
\end{lem}
\begin{proof}
    Let $K\subsetneq T$ be a subset with $K\not\sq C$ and $\wh K=K\in\mc{N}_{\rm sph}$ and $K\cap \Delta(\mc{C})=\emptyset$. Since $K$ is a proper subset of $T$, it has Dynkin-neighbours in $T$ and, since $\wh K=K$, they all lie in $T\setminus (K\cup C)$. Suppose for the sake of contradiction that all these neighbours lie in $\Delta(\mc{C})$.

    Since $K\in\mc{N}_{\rm sph}$, there exists an element $f\in T\cap\delta_AK$ by \Cref{lem:delta_works_2}, and we have $f\in \Delta(\mc{C})$ by our assumptions. Since $K\not\sq C$, we can pick a shortest path $a_0,\dots,a_{n-1}\in K$ in the Dynkin diagram of $T$ with $a_0\in K\setminus C$ and $a_n=f$; we have $n\geq 1$ and $a_i\in C$ for all $i\not\in\{0,n\}$. By \Cref{lem:atoms_1}(4), we have the equivalence $(a_n,1)\equiv_{\mc{C}}(a_n,a_{n-1}\dots a_0)$.

    We claim that we also have $(a_0,1)\equiv_{\mc{C}}(a_0,a_1\dots a_n)$. In order to see this, note that $\mc{C}\sq\mc{H}_k$ for all $k\in K\setminus C$ because $K\cap \Delta(\mc{C})=\emptyset$. The projection $\pi_K(\mc{C})$ is a $\langle C\cap K\rangle$--standard subcomplex, and so it contains a point $p$ with $p\in\mc{H}_c$ for all $c\in C\cap K$. Combining these two observations, we have that $p\in\mc{H}_k$ for all $k\in K$, and hence $p=p_A$. This means that $p_A$ and $\mc{C}$ lie on the same side of every wall that is disjoint from $\mc{C}$ and crosses a $\langle K\rangle$--standard subcomplex. In particular, we have $(a_0,1)\equiv_{\mc{C}}(a_0,g_{K,a_0})$.
    Now, recalling that $f\in T\cap\delta_AK$, we have $(a_0,g_{K,a_0})\equiv_{\mc{C}}(a_0,g_{K,a_0}f)$ by \Cref{lem:C-admissible_equivalence}(1). Finally, recalling that $f=a_n$, \Cref{lem:atoms_1}(3) yields $(a_0,g_{K,a_0}f)\equiv_{\mc{C}} (a_0,a_1\dots a_n)$, proving our claim.

    Now, observe that the set $\{a_0,\dots,a_n\}$ is irreducible and spherical
    with Dynkin diagram a path, an so it is of one of the types $A_{n+1}$, $B_{n+1}$, $H_{n+1}$, and $F_4$ (for $n=3$). 
    
    Suppose first that it is of type $F_4$. By the above discussion and using move (M2)$_{\mscr{C}}$, we have $(a_0,1)\equiv_{\mc{C}}(a_0,a_1)\equiv_{\mc{C}}(a_0,a_1a_2a_3)$ and $(a_3,1)\equiv_{\mc{C}}(a_3,a_2a_1a_0)$. Since $\mc{C}\sq\mc{H}_{a_0}\cap\mc{H}_{a_3}^*$, it follows that 
    \[ \mc{C}\sq a_1\mc{H}_{a_0}\cap a_3a_2a_1\mc{H}_{a_0}\cap a_0a_1a_2\mc{H}_{a_3}^* \quad\text{ and hence }\quad a_0a_3\mc{C}\sq\mc{H}_{a_1}\cap a_2\mc{H}_{a_1}\cap a_1\mc{H}_{a_2}^* .\]
    Since $a_1$ and $a_2$ do not commute, the latter intersection is empty, yielding a contradiction.

    To conclude, suppose that $\{a_0,\dots,a_n\}$ is of one of types $A_{n+1}$, $B_{n+1}$ and $H_{n+1}$. Say that each pair $(a_i,a_{i+1})$ has label $3$, possibly except for $i=n-1$. (The case when a label $\neq 3$ occurs for $i=0$ is handled exactly in the same way.) Again, we have $\mc{C}\sq\mc{H}_{a_0}\cap\mc{H}_{a_n}^*$, as well as the equivalences $(a_0,1)\equiv_{\mc{C}}(a_0,a_1\dots a_{n-1})\equiv_{\mc{C}}(a_0,a_1\dots a_n)$ and $(a_n,1)\equiv_{\mc{C}}(a_n,a_{n-1}\dots a_0)$. This implies that
    \[ \mc{C}\sq a_{n-1}\dots a_1\mc{H}_{a_0}\cap a_n\dots a_1\mc{H}_{a_0}\cap a_0\dots a_{n-1}\mc{H}_{a_n}^* , \]
    and hence 
    \[ a_{n-2}\dots a_0\mc{C}\sq \mc{H}_{a_{n-1}}\cap a_n\mc{H}_{a_{n-1}}\cap a_{n-1}\mc{H}_{a_n}^* . \]
    Once more, since $a_{n-1}$ and $a_n$ do not commute, the above intersection is empty, yielding a contradiction. This concludes the proof of the lemma.
\end{proof}

We can now finally prove the main structural result about atoms. Recall that $w_K$ denotes the longest element of a spherical subgroup $\langle K\rangle$.

\begin{prop}\label{prop:atoms_2}
    Let $T$ be an atom. Setting $L:=T\setminus (C\cup\Delta(\mc{C}))$, the following hold.
    \begin{enumerate}
        \item The set $T$ is spherical. 
        \item We have $p_A\in\pi_T(w_Tw_{\wh L}\mc{C})$ and $d(p_A,w_Tw_{\wh L}\mc{C})<d(p_A,\mc{C})$.
        \item All irreducible factors of $\wh L$ lie in $\mc{S}_{\rm sph}$.
    \end{enumerate}
\end{prop}
\begin{proof}
    We begin with Items~(1) and~(2). Since $T\cap\Delta(\mc{C})\neq\emptyset$ and $\wh L\cap\Delta(\mc{C})=\emptyset$, \Cref{lem:atoms_1}(5) guarantees that the set $\wh L$ is spherical. In particular, the longest element $w_{\wh L}$ is well-defined.

    We claim that $\mc{C}\sq w_{\wh L}\mc{H}_f^*$ for every element $f\in T\cap \Delta(\mc{C})$. Towards this, let $K_1,\dots,K_k$ be the irreducible factors of $\wh L$ containing elements Dynkin-adjacent to $f$. We have $w_{\wh L}\mc{H}_f^*=w_{K_1}\dots w_{K_k}\mc{H}_f^*$ and the $w_{K_i}$ commute pairwise. By \Cref{lem:producing_bases_2}, there exist elements $w_i\in\langle K_i\rangle$ such that the pairs $(f,w_i)$ are bases and we have $w_i^{-1}\mc{H}_f^*=w_{K_i}\mc{H}_f^*$. In particular, we have $w_{\wh L}\mc{H}_f^*=w_1^{-1}\dots w_k^{-1}\mc{H}_f^*$ and the pair $(f,w_1\dots w_k)$ is itself a base. Since $f\in \Delta(\mc{C})$, we have $\mc{C}\sq\mc{H}_f^*$ and so our goal of proving the inclusion $\mc{C}\sq w_{\wh L}\mc{H}_f^*$ reduces to showing the equivalence of $\mc{C}$--independent bases $(f,1)\equiv_{\mc{C}} (f,w_1\dots w_k)$. In turn, the latter follows from \Cref{lem:atoms_1}(4), proving our claim.

    Summing up, we have $\mc{C}\sq\mc{H}_{\ell}$ for all $\ell\in L$ by definition, and we have just seen that $\mc{C}\sq w_{\wh L}\mc{H}_f^*$ for every element $f\in T\cap \Delta(\mc{C})$. Set $C_1:=\wh L\cap C$ and $C_2:=(C\cap T)\setminus C_1$, and observe that we have $C\cap T=C_1\sqcup C_2$, where each $C_i$ is a union of irreducible factors of $C\cap T$. Invoking \Cref{lem:atoms_1}(5) again, we see that $C_1$ and $C_2$ are spherical. The projection $\pi_T(\mc{C})$ is a $\langle C\cap T\rangle$--standard subcomplex, and so it contains a point $p\in\bigcap_{c\in C_1}\mc{H}_c\cap\bigcap_{c\in C_2}\mc{H}_c^*$. Moreover, we have $w_{\wh L}\mc{H}_c^*=\mc{H}_c^*$ for all $c\in C_2$. Combining all of this, we obtain that $p$ lies in the intersection
    \[ \mf{R}(T)\cap\bigcap_{\ell\in\wh L}\mc{H}_{\ell}\cap\bigcap_{f\in T\setminus\wh L} w_{\wh L}\mc{H}_f^* . \]
    Invoking \Cref{lem:ww_2}, it follows that $T$ is spherical, proving Item~(1). Moreover, $p=w_{\wh L}w_Tp_A$. 
    
    Since $p\in\pi_T(\mc{C})$, we also get that $p_A\in \pi_T(w_Tw_{\wh L}\mc{C})$. As a consequence:
    \begin{align*}
        d(p_A,\mc{C})&=d(p_A,\pi_T(\mc{C}))+d(\pi_T(\mc{C}),\mc{C}) \\
        & = d(p_A,\pi_T(\mc{C}))+d(\pi_T(w_Tw_{\wh L}\mc{C}),w_Tw_{\wh L}\mc{C}) \\
        & = d(p_A,\pi_T(\mc{C}))+d(p_A,w_Tw_{\wh L}\mc{C}) .
    \end{align*}
    Since $T\cap\Delta(\mc{C})\neq\emptyset$, we have $d(p_A,\pi_T(\mc{C}))>0$. Thus, we obtain that $d(p_A,\mc{C})>d(p_A,w_Tw_{\wh L}\mc{C})$, completing the proof of Item~(2).

    Finally, regarding Item~(3), let $K$ be an irreducible factor of $\wh L$. Since $T$ is an atom, we have $T\cap \Delta(\mc{C})\neq\emptyset$ and hence $K\neq T$. It is also clear that $K\not\sq C$, that $\wh K=K$, that $K\cap \Delta(\mc{C})=\emptyset$, and that all elements of $T\setminus K$ Dynkin-adjacent to elements of $K$ lie in $\Delta(\mc{C})$. Thus, \Cref{lem:atoms_5} shows that $K\not\in\mc{N}_{\rm sph}$, as desired. This concludes the proof of the proposition.
\end{proof}

Before proving the shortening theorem, we need two last lemmas. Here we need to recall that $A$ coincides with an irreducible factor $V_0$ of some $V\in\mc{V}_{\Om}$, and the set $B=W$ lies in ${\rm Nbr}_{\Om}(V)$. Also recall that the existence of an atom $T\sq A$ implies that $\Delta(\mc{C})\neq\emptyset$ and hence that $p_A\not\in\mc{C}$. In particular, it implies that $p_V\not\in\pi_V(\mf{R}(W))$.

\begin{lem}\label{lem:sep_good}
    Let $T\sq A$ be an atom. Let $K$ denote either $T$ itself, or an irreducible factor of the subset $\wh L\sq T$ from \Cref{prop:atoms_2}. If $X\in{\rm Nbr}_{\Om}(V)$ satisfies $p_A\in\pi_A(\mf{R}(X))$, then:
    \begin{enumerate}
        \item either no connected component of $\wh S_{\mscr{C}}\setminus (K\cup K^{\perp})$ meets both $W$ and $X$;
        \item or $K$ is an irreducible factor of $V\cap X$.
    \end{enumerate}
\end{lem}
\begin{proof}
    We can assume that $W\not\sq K\cup K^{\perp}$, otherwise the lemma is clear. Let $\mf{C}$ be the connected component of $\wh S_{\mscr{C}}\setminus (K\cup K^{\perp})$ that contains the clique $W\setminus(K\cup K^{\perp})$.

    Observe that $\mf{C}\cap V=\emptyset$. Indeed, when $K$ is an irreducible factor of $\wh L$, this follows from the fact that $K\in\mc{S}_{\rm sph}(V,W)$, which was shown in \Cref{prop:atoms_2}(3). When $K=T\neq A$, it is again because $K\in\mc{S}_{\rm sph}(V,W)$, which holds by the definition of atom. Finally, when $K=T=A$, it is because $V\sq K\cup K^{\perp}$, since $A$ is an irreducible factor of $V$.

    Suppose that $\mf{C}\cap X\neq\emptyset$. Since $\mf{C}\cap V=\emptyset$ and $X$ spans a clique in $\wh S_{\mscr{C}}$, we must have $V\cap X\sq K\cup K^{\perp}$. If $K\sq V\cap X$, then $K$ is an irreducible factor of $V\cap X$ and we are done. Thus, we can assume in the rest of the proof that $K\setminus X\neq\emptyset$, and our goal becomes reaching a contradiction.

    Since $T^{\perp}\sq K^{\perp}$ and $T\sq V$, we have $\mf{C}\cap (T\cup T^{\perp})=\emptyset$. We also have $T\setminus X\neq\emptyset$ since $K\setminus X\neq\emptyset$. From now on, we will only work with the set $T$ and we can safely forget about the set $K$. Choose an element $f\in\Delta(\mc{C})\cap T$, which exists by the definition of atom. Setting $w:=g_{T,f}$ and recalling that $\mc{H}_f$ denotes the side of the wall $\mc{W}_f$ containing $p_V=p_A$, we have $\mc{W}_b\sq w^{-1}\mc{H}_f^*$ for all elements $b\in W\setminus V$, by \Cref{lem:atoms_1}(4).

    Let $\gamma\sq\mf{C}$ be a geodesic from a point $b\in W$ to a point $x\in X$. We claim that the pair $(T,\gamma)$ is strongly inseparable. To see this, suppose that we have $T\sq J$ for a set $J\in\mc{J}_S$ that contains a point $z\in\gamma$ and has a spherical difference $J\setminus T$. If $b$ and $x$ were to lie in distinct connected components of the graph $\wh S_{\mscr{C}}\setminus (J\cup J^{\perp})$, then \Cref{lem:separation}(5) would imply that $J$ is contained either in ${\rm Blob}_{\Om}(V\cap W)$ or in ${\rm Blob}_{\Om}(V\cap X)$. However, note that $T$ intersects $V\setminus W$ by definition of atom, and it intersects $V\setminus X$ as shown in the previous paragraph. Thus, the only possibility is that $J\sq V$, which however violates the fact that $z\in J\cap\mf{C}$ and $\mf{C}\cap V=\emptyset$.

    In conclusion, we can apply \Cref{prop:markings_new} to the pair $(T,\gamma)$. Note that the markings $((f,w),b)$ and $((f,w),x)$ are both admissible and complete: indeed, we have seen that the set $T=\supp(f,w)$ contains points $b'\in V\setminus W$ and $x'\in V\setminus X$, and we have $b\in W\setminus V$ and $x\in X\setminus V$, which implies that the subgroups $\langle b,b'\rangle$ and $\langle x,x'\rangle$ are infinite; this immediately yields admissibility, while completeness follows from \Cref{rmk:complete_or_semicomplete}. Now, \Cref{prop:markings_new}(1) implies that $((f,w),b)\equiv_R((f,w),x)$. Since $\mc{W}_b\sq w^{-1}\mc{H}_f^*$, this means that $\mc{W}_x\sq w^{-1}\mc{H}_f^*$.
    
    At the same time, the hypothesis that $p_A\in\pi_A(\mf{R}(X))$ implies that $p_A$ and $\mc{W}_x$ are on the same side of $w^{-1}\mc{W}_f$. The points $p_A$ and $wp_A$ are on the same side of $\mc{W}_f$, because $(f,w)$ is a basis, and so we have $p_A\in w^{-1}\mc{H}_f$. Hence we have $\mc{W}_x\sq w^{-1}\mc{H}_f$, which is the required final contradiction.
\end{proof}

\begin{lem}\label{lem:sep_excellent}
    Consider a set $K\in\mc{I}_S$ and two sets $X,Y\in{\rm Nbr}_{\Om}(V)$ such that $K$ is an irreducible factor of both $V\cap X$ and $V\cap Y$. Suppose that the sets $V\cap X$ and $V\cap Y$ are $\mf{R}$--good. 
    \begin{enumerate}
        \item If $X$ and $Y$ meet the same connected component of $\wh S_{\mscr{C}}\setminus (K\cup K^{\perp})$, then $\pi_K(p_X)=\pi_K(p_Y)$.
        \item If $X$ and $V$ meet the same component of $\wh S_{\mscr{C}}\setminus (K\cup K^{\perp})$, then $\pi_K(p_X)=\pi_K(p_V)$.
    \end{enumerate}
\end{lem}
\begin{proof}
    Set $E:=V\cap X$ and $F:=V\cap Y$, 
    and also $\Xi:={\rm Nbr}_{\Om}(V;E)\cup {\rm Nbr}_{\Om}(V;F)\cup\{V\}$. Supposing for the sake of contradiction that the lemma fails, let $\gamma\sq\wh S_{\mscr{C}}\setminus (K\cup K^{\perp})$ be a shortest path with endpoints $x,y$ respectively lying in two sets $X',Y'\in\Xi$ with $\pi_K(p_{X'})\neq\pi_K(p_{Y'})$.

    We claim that the pair $(K,\gamma)$ is strongly inseparable. Indeed, suppose that we have $K\sq J$ for a set $J\in\mc{J}_S$ that contains a point $z\in\gamma$ and has a spherical difference $J\setminus K$. If $x$ and $y$ were to lie in distinct connected components of $\wh S_{\mscr{C}}\setminus (J\cup J^{\perp})$, then \Cref{lem:separation}(5) would imply that $J$ is contained in either ${\rm Blob}_{\Om}(E)$ or ${\rm Blob}_{\Om}(F)$. Since $J$ spans a clique in $\wh S_{\mscr{C}}$, there exists a set $Z\in\Xi$ such that $z\in J\sq Z$. Since $z\not\in\{x,y\}$ and we have either $\pi_K(p_Z)\neq\pi_K(p_{X'})$ or $\pi_K(p_Z)\neq\pi_K(p_{Y'})$, this contradicts the fact that the path $\gamma$ is shortest.

    Thus, we can apply \Cref{prop:markings_new} to the pair $(K,\gamma)$. For this, we define two markings $\mu_x,\mu_y$ as follows. First, pick any base $(s,u)$ with support $K$. Then set $\mu_x:=((s,u),x)$ if $K\cup\{x\}$ is not tree--$2$--spherical, and $\mu_x:=((s,ux),y)$ otherwise. The marking $\mu_y$ is defined analogously, swapping the roles of $x$ and $y$. Since $K$ is an irreducible factor of $E$ and $F$, and the sets $K\cup\{x\}$ and $K\cup\{y\}$ are irreducible by construction, we see that $x,y\not\in E\cup F$. It follows that the subgroup $\langle x,y\rangle$ is infinite, and so the markings $\mu_x$ and $\mu_y$ are admissible and complete (see \Cref{rmk:complete_or_semicomplete}). We then obtain $\mu_x\equiv_R\mu_y$ by one of the three items of \Cref{prop:markings_new} (depending on how many among $K\cup\{x\}$ and $K\cup\{y\}$ are tree--$2$--spherical).

    Now, we claim that the marking $\mu_x$ defines the side of the wall $u^{-1}\mc{W}_s$ containing the point $p_{X'}$ (and the same for $\mu_y$ and $p_{Y'}$). This is clear if $K\cup\{x\}$ is not tree--$2$--spherical, simply because the point $p_{X'}$ is $X'$--geometric. If $K\cup\{x\}$ is instead tree--$2$--spherical, this follows from \Cref{lem:auxiliary} (applied with $A:=X'$, $a:=x$, $B:=Y'$, $b:=y$ and $w:=ux$) 
    together with the fact that $E$ and $F$ are $\mf{R}$--good.
    
    In conclusion, the equivalence $\mu_x\equiv_R\mu_y$ and the previous paragraph show that the points $p_{X'}$ and $p_{Y'}$ are on the same side of the wall $u^{-1}\mc{W}_s$. At the same time, since $p_{X'}$ is $X'$--geometric, the projection $\pi_K(p_{X'})$ is $K$--geometric, and the same is true of $\pi_K(p_{Y'})$. Since $K$ is irreducible and $\pi_K(p_{X'})\neq\pi_K(p_{Y'})$, it follows that $K$ is spherical and $\pi_K(p_{X'})=w_K\pi_K(p_{Y'})$. In particular, the points $\pi_K(p_{X'})$ and $\pi_K(p_{Y'})$ are separated by \emph{all} walls of $\mf{R}(K)$, contradicting the fact that they are on the same side of $u^{-1}\mc{W}_s$. This concludes the proof.
\end{proof}

We are finally ready to prove the shortening theorem.

\begin{proof}[Proof of \Cref{thm:shortening}]
    Let $\Om$ and $\mc{T}$ be as in \Cref{setup_two}, and let $\mf{R}$ be a reference system for $\Om$. Consider a set $V\in\mc{V}_{\Om}$ and a neighbour $W\in{\rm Nbr}_{\Om}(V)$. 

    In the terminology of \Cref{rmk:performing_twist}, our goal is to perform $(V;W,K)$--twists with respect to several sets $K\in\mc{I}_V$ so as to produce a new generating set $S'\in\mscr{T}_S$ and a new subset $\Om'\sq S'$. That Items~(1) and~(2) of the theorem hold is a consequence of \Cref{lem:semivisual_stays} and of the discussion in Remarks~\ref{rmk:performing_twist} and~\ref{rmk:neighbour-preserving}. In particular, for each neighbour $X\in{\rm Nbr}_{\Om}(V)$, we will replace the set ${\rm Tail}_{\Om}(V;X)\sq\Om$ by its conjugate by an element $g_X\in\langle V\rangle$. 
    
    Thus, we only need to check that Items~(3) and~(4) hold.

    \smallskip
    {\bf Item~(3).} Here we suppose that $W$ is not an $\mf{R}$--good neighbour, so we have $p_V\not\in\pi_V(\mf{R}(W))$ and we can choose an irreducible factor $V_0\sq V$ such that $p_{V_0}:=p_V\not\in\pi_{V_0}(\mf{R}(W))$. Our goal is to perform twists so that there exists a $V$--geometric point $q$ satisfying $d(q,g_W\pi_V(\mf{R}(W)))<d(p_V,\pi_V(\mf{R}(W)))$. We first explain how to produce the element $g_W$ and the point $q$, and then we check that this does not spoil the neighbours of $V$ that are already $\mf{R}$--good.
    
    We can harness the framework developed in this subsection setting $A:=V_0$, $B:=W$ and $\mc{C}:=\pi_A(\mf{R}(B))$. Since $p_A\not\in\mc{C}$, \Cref{rmk:empty_F} yields $\Delta(\mc{C})\neq\emptyset$, and so there exists an atom $T\sq A$.

    Suppose first that $T\neq A$ and consider the element $g:=w_Tw_{\wh L}$ appearing in \Cref{prop:atoms_2}(2). Item~(2) of the lemma shows that all irreducible factors of $\wh L$ lie in $\mc{S}_{\rm sph}(V,W)$. The set $T$ also lies in $\mc{S}_{\rm sph}(V,W)$ because it is an atom $\neq A$. 
    Finally, \Cref{prop:atoms_2}(2) guarantees that we have $d(p_A,g\mc{C})<d(p_A,\mc{C})$, and hence $d(p_V,g\pi_V(\mf{R}(W)))<d(p_V,\pi_V(\mf{R}(W)))$. We thus perform $(V;W,K)$--twists with $K$ equaling each of the irreducible factors of $\wh L$, and then with $K=T$. The result is that $g_W=g$ and we can simply set $q:=p_V$.

    Suppose instead that $T=A$. In this case, $A$ is spherical by \Cref{prop:atoms_2}(1). If $A\cup A^{\perp}$ disconnects the graph $\wh S_{\mscr{C}}$, then we proceed exactly as in the previous case. Otherwise, \Cref{prop:atoms_2}(2) shows that $\pi_A(w_{\wh L}\mc{C})$ contains the $V$--geometric point $w_Ap_A$. We then simply perform $(V;W,K)$--twists with $K$ equaling the irreducible factors of $\wh L$. We obtain $g_W=w_{\wh L}$ and set $q:=w_Ap_A$.

    We are left to check that this procedure has not spoiled any $\mf{R}$--good neighbours of $V$. For an arbitrary neighbour $Y\in{\rm Nbr}_{\Om}(V)$, the above procedure yields $g_Y=1$ unless $Y$ and $W$ meet the same connected component of $\wh S_{\mscr{C}}\setminus (K\cup K^{\perp})$, either for $K=T$ or for $K$ equaling an irreducible factor of $\wh L$. If $X$ is an $\mf{R}$--good neighbour of $V$, \Cref{lem:sep_good} shows that the latter can happen only if $K$ is contained in $X$. Thus, for every $\mf{R}$--good neighbour $X$, we have $g_X\in\langle V\cap X\rangle$ and hence $g_X\pi_V(\mf{R}(X))=\pi_V(\mf{R}(X))$.
    
    Now, when $T\neq A$ or when $\wh S_{\mscr{C}}\setminus (A\cup A^{\perp})$ is disconnected, we have $q=p_V$ and therefore $q\in g_X\pi_V(\mf{R}(X))$ as required. When $T=A$ and $\wh S_{\mscr{C}}\setminus (A\cup A^{\perp})$ is connected, \Cref{lem:sep_good} guarantees that the set $A$ is contained in every $\mf{R}$--good neighbour $X\in{\rm Nbr}_{\Om}(V)$. Thus, we again have $q=w_Ap_A\in w_A\pi_V(\mf{R}(X))=\pi_V(\mf{R}(X))$ as desired.

    \smallskip
    {\bf Item~(4).} Here we suppose that all pairs of neighbouring elements of $\mc{V}_{\Om}$ are $\mf{R}$--good, and that $\pi_V(p_W)\neq p_V$. Our goal is to perform a twist so that $d(p_V,g_W\pi_V(p_W))<d(p_V,\pi_V(p_W))$. 
    
    Since the pair $(V;W)$ is $\mf{R}$--good by hypothesis, we have $p_V\in\pi_V(\mf{R}(W))$. Thus, the fact that $\pi_V(p_W)\neq p_V$ implies the existence of an irreducible factor $K\sq V\cap W$ such that $\pi_K(p_W)\neq p_V$. Since the point $p_W$ is $W$--geometric, the projection $\pi_K(p_W)$ is $K$--geometric, and so $K$ is spherical and we have $\pi_K(p_W)=w_Kp_V$. By \Cref{lem:sep_excellent}(2), we have $K\in\mc{S}_{\rm sph}(V,W)$ and so we can perform a $(V;W,K)$--twist. This produces $g_W=w_K$ and yields the desired inequality.

    We are left to check that $g_X=1$ for all $\mf{R}$--excellent neighbours $X\in{\rm Nbr}_{\Om}(V)$. That is, we need to check that $X$ does not meet the same connected component of $\wh S_{\mscr{C}}\setminus (K\cup K^{\perp})$ as the set $W$. Since $X$ is an $\mf{R}$--excellent neighbour and $K\sq V$, we have $\pi_K(p_X)=p_V\neq\pi_K(p_W)$ and so it suffices to invoke \Cref{lem:sep_excellent}(1). This concludes the proof of the shortening theorem.
\end{proof}

We now quickly obtain \Cref{thm:step_two_repeated}, which is the same as \Cref{thm:step_two}. As explained in \Cref{sect:strategy}, this also completes the proof of \Cref{thmintro:main}.

\begin{proof}[Proof of \Cref{thm:step_two_repeated}]
    Let again $\Om$ and $\mc{T}$ be as in \Cref{setup_two}. In addition, we now assume that the set $\Om$ is $\mscr{T}_S$--parabolic and $\mscr{T}_S$--tucked.
    
    Let $\mscr{O}$ be the collection of pairs $(S',\Om')$ such that we have $\Om'\sq S'\in\mscr{T}_S$ and $\langle\Om'\rangle=\langle\Om\rangle$, and such that the splitting $\mc{T}$ is $\Om'$--semivisual. By \Cref{prop:inflexible}, the set $\Om'$ is $\mscr{C}$--inflexible for every $(S',\Om')\in\mscr{O}$, and it is also $S'$--tucked by our hypotheses. This allows us to apply the shortening theorem to any of these pairs. Our goal is to find some $(S',\Om')\in\mscr{O}$ such that $\Om'$ is $R$--geometric, which, in view of \Cref{lem:perfect->geometric}, simply amounts to $\Om'$ admitting a perfect reference system. 

    We start with an arbitrary reference system $\mf{R}$ for $\Om$. We will apply \Cref{thm:shortening} several times to produce new elements $(S^i,\Om^i)\in\mscr{O}$. Whenever we do this, there is a natural neighbour-preserving identification $\mc{V}_{\Om}\ra\mc{V}_{\Om^i}$, which we denote by $V\mapsto V^i$. The new set $\Om^i$ is equipped with a natural reference system $\mf{R}^i$ defined as follows. Let the set $V$ and the elements $g_N$ be as in the statement of \Cref{thm:shortening}. For an arbitrary element $Z\in\mc{V}_{\Om}\setminus\{V\}$, let $Y\in{\rm Nbr}_{\Om}(V)$ be the set such that $Z\sq{\rm Tail}_{\Om}(V;Y)$, and recall that we have $g_YZg_Y^{-1}\sq\Om^i$. Thus, we set $\mf{R}^i(Z):=g_Y\mf{R}(Z)$ and $p_Z^i:=g_Yp_Z$. We also set $\mf{R}^i(V):=\mf{R}(V)$ and $p_V^i:=p_V$, except after an application of Item~(3) of \Cref{thm:shortening}, where we instead set $p_V^i:=q$.
    
    After this preliminary discussion, we now finally get to the proof of the theorem. We start by picking an arbitrary element $V\in\mc{V}_{\Om}$. If there is neighbour $W\in{\rm Nbr}_{\Om}(V)$ that is not $\mf{R}$--good, we repeatedly apply \Cref{thm:shortening}(3) to it until we obtain an element $(S^1,\Om^1)\in\mscr{O}$ such that $V^1=V$ and the pair $(V^1;W^1)$ is $\mf{R}^1$--good. If $X,Y\in\mc{V}_{\Om}$ was an arbitrary pair of neighbours and $(X;Y)$ was $\mf{R}$--good, then $(X^1,Y^1)$ remains $\mf{R}^1$--good: if $X=V$, then this is part of the statement of \Cref{thm:shortening}(3); if $Y=V$, then this is due to the fact that $X^1$ is $\langle V\rangle$--conjugate to $X$; finally, if $V\not\in\{X,Y\}$, then this is because $X\cup Y$ is contained in ${\rm Tail}_{\Om}(V;U)$ for some $U\in{\rm Nbr}_{\Om}(V)$, and so there exists an element $h\in\langle V\rangle$ such that $X^1=hXh^{-1}$ and $Y^1=hYh^{-1}$.

    In conclusion, there are strictly more pairs of $\mf{R}^1$--good neighbours in $\mc{V}_{\Om^1}$ than there are pairs of $\mf{R}$--good neighbours in $\mc{V}_{\Om}$. Repeating this procedure finitely many times, we eventually obtain an element $(S^2,\Om^2)\in\mscr{O}$ such that all pairs of neighbours in $\mc{V}_{\Om^2}$ are $\mf{R}^2$--good. Finally, repeating this whole procedure again using Item~(4) of \Cref{thm:shortening} in place of Item~(3), we obtain an element $(S^3,\Om^3)\in\mscr{O}$ such that all pairs of neighbours in $\mc{V}_{\Om^3}$ are $\mf{R}^3$--excellent. That is, the reference system $\mf{R}^3$ is perfect, proving the theorem.
\end{proof}

\appendix

\section{Isomorphisms between Coxeter groups}\label{app:automorphisms}

This appendix briefly summarises the results from the literature needed to deduce from the Twist Conjecture a solution to the Isomorphism Problem for Coxeter groups (\Cref{corintro:iso_problem} in the Introduction). We also deduce that (relative) automorphism groups of Coxeter groups are finitely generated (\Cref{corintro:automorphisms}). Large parts of our treatment are based on the surveys \cite{Muehlherr-survey,SRS}.

\subsection{Saturated Coxeter systems}

Let $(W,S)$ be a Coxeter system.

\begin{defn}[\cite{Muehlherr-survey}]
An element $\tau\in S$ is a \emph{pseudo-transposition} if there exists a (necessarily unique) element $t\in S$ such that the following hold:
\begin{enumerate}
    \item $o(\tau t)=2(2k+1)$ for some $1\leq k<\infty$;
    \item $o(\tau s)\in\{2,\infty\}$ for all $s\in S\setminus\{\tau,t\}$;
    \item $\tau^{\perp}\sq t^{\perp}$.
\end{enumerate}
The Coxeter system $(W,S)$ is \emph{saturated}\footnote{This terminology comes from \cite{SRS}. The term \emph{reduced} was used in \cite{Howlett-Muehlherr,Muehlherr-survey}, and \emph{expanded} in \cite{Ratcliffe-Tschantz}.} if it contains no pseudo-transpositions.
\end{defn}

If $\tau\in S$ is a pseudo-transposition, we obtain a new Coxeter system $(W,S')$ by considering $S':=S\setminus\{\tau\}\cup\{\tau t\tau,\rho\}$, where $\rho$ is is the longest element of $\langle\tau, t\rangle$. This procedure is known as a \emph{blow-up}. The following was shown in \cite[Proposition~6]{Howlett-Muehlherr}; see also the treatment in \cite{MRT07}.

\begin{prop}[\cite{Howlett-Muehlherr}]\label{prop:saturated}
    For each Coxeter system $(W,S)$, any sequence of blow-ups terminates in less than $|S|$ steps, resulting in a saturated Coxeter generating set $S'\sq W$.
\end{prop}

In fact, the maximal-cardinality Coxeter generating sets of $W$ are precisely those that are saturated; see \cite[Theorem~9.1]{MRT07}.

\subsection{From saturation to reflection-compatibility}

Two Coxeter generating sets $S,R\sq W$ are \emph{reflection-compatible} if $S^W=R^W$; see \cite[Corollary~A.2]{CP10} for the equivalence with other common formulations of this notion. For an element $w\in W$, the \emph{finite continuation} ${\rm FC}(w)$ is the intersection of all maximal finite subgroups of $W$ containing $w$. The subgroup ${\rm FC}(w)$ is $S$--parabolic for all Coxeter generating sets $S\sq W$; moreover, it can be computed from the Coxeter matrix of $S$, see \cite[Theorem~7]{FHM}.

The following is \cite[Theorem~1]{Howlett-Muehlherr} and \cite[Theorem~4.3]{Muehlherr-survey}.

\begin{thm}[\cite{Howlett-Muehlherr}]\label{appthm:HM}
    Let $(W,S)$ be a saturated Coxeter system. There exists an explicit finite subgroup $\Sigma_S\leq\Aut(W)$ with the following properties:
    \begin{enumerate}
        \item for every saturated Coxeter generating set $R\sq W$, there exists $\s\in\Sigma_S$ such that $\s(S)$ and $R$ are reflection-compatible;
        \item for every $s\in S$, the subgroup ${\rm FC}(s)\leq W$ is $\Sigma_S$--invariant.
    \end{enumerate}
\end{thm}

In particular, Item~(2) guarantees that the cardinality $|\Sigma_S|$ is at most the product of the cardinalities of the subgroups ${\rm FC}(s)$ with $s\in S$ (better bounds are also easy to obtain).

The subgroup $\Sigma_S$ has a generating set that can be explicitly described in terms of any Coxeter matrix for $(W,S)$: it consists of the elements known as \emph{$s$--transvections}\footnote{These are unrelated to the automorphisms of right-angled Coxeter and Artin groups known commonly known as \emph{transvections} \cite{Servatius,GPR}.} and \emph{$J$--local automorphisms}, for certain elements $s\in S$ and subsets $J\sq S$. The definition of these two types of automorphisms can be found in \cite[Section~4]{Muehlherr-survey} in most cases, and in \cite[Definition~7]{Howlett-Muehlherr} in full generality.

\subsection{From reflection-compatibility to angle-compatibility}

Let $(W,S)$ be any Coxeter system. An \emph{edge} of $S$ is a subset $\{s,t\}\sq S$ where $s,t$ are distinct and $o(st)$ is finite. We denote by ${\rm Edges}(S)$ the set of all edges of $S$.

\begin{defn}[\cite{Marquis-Muehlherr}]
    Let $J=\{u,v\}\sq S$ be an edge. A \emph{$J$--deformation} of $S$ is an injection $\delta\colon S\hookrightarrow S^W$ with the following properties:
    \begin{enumerate}
        \item $\delta(S)$ is a Coxeter generating set $W$;
        \item $\delta(u)=u$ and $\delta(v)$ is an $\langle u,v\rangle$--conjugate of $v$ (or vice versa);
        \item there is a bijection $\Delta\colon{\rm Edges}(S)\ra{\rm Edges}(\delta(S))$ such that $\Delta(J)=\delta(J)$ and $\Delta(E)$ is $W$--conjugate to $E$ for all $E\in{\rm Edges}(S)\setminus\{J\}$.
    \end{enumerate}
    An \emph{angle-deformation} of $S$ is a $J$--deformation for some edge $J\sq S$.
\end{defn}

Note that not all angle-deformations extend to automorphisms of $W$ and, a priori, there is no control on the length of the elements of $W$ needed to conjugate each $s\in S$ to $\delta(s)$.

\begin{thm}[\cite{Marquis-Muehlherr}]\label{appthm:MM}
    Given two reflection-compatible Coxeter generating sets $S,R\sq W$, there is an algorithmically computable sequence of Coxeter generating sets $S=:S_0,\dots,S_k$ such that:
    \begin{enumerate}
        \item for each $i$, there exists a subset $J_i\sq S_i$ such that $S_{i+1}$ is a $J_i$--deformation of $S_i$;
        \item $S_k$ and $R$ are angle-compatible.
    \end{enumerate}
\end{thm}

If $S$ does not have any spherical subsets of type $H_3$, then all angle-deformations needed in the previous theorem extend to automorphisms of $W$, and they are all $J_i$--deformations $\delta_i$ admitting elements $w_{s,i}\in\langle J_i\rangle$ such that $\delta(s)=w_{s,i}sw_{s,i}^{-1}$ for all $s\in S$ (which gives a bound on $|w_{s,i}|$).

The combination of \Cref{prop:saturated} with Theorems~\ref{appthm:HM} and~\ref{appthm:MM} reduces the resolution of the Isomorphism Problem for Coxeter groups to the resolution of the Twist Conjecture. Thus, together with \Cref{thmintro:main}, this yields \Cref{corintro:iso_problem} from the Introduction.

Let $\Aut_{\rm AC}(W;S)\leq\Aut(W)$ be the subgroup of automorphisms $\varphi\in\Aut(W)$ such that $\varphi(S)$ is angle-compatible with $S$. Observe that $\Aut_{\rm AC}(W;S)$ has finite index in $\Aut(W)$, since $W$ has only finitely many conjugacy classes of finite subgroups. As a consequence of Theorems~\ref{appthm:HM} and~\ref{appthm:MM}, we also obtain an algorithm to compute a relative set of generators:

\begin{cor}\label{cor:from_Aut_to_Aut_AC}
    For every saturated Coxeter system $(W,S)$, there is an algorithmically computable finite subset $F\sq\Aut(W)$ such that $\Aut(W)$ is generated by $F\cup\Aut_{\rm AC}(W;S)$.
\end{cor}

\subsection{Automorphism groups}

Finally, it is not hard to deduce from \Cref{thmintro:main} that the group $\Aut_{\rm AC}(W;S)$ is finitely generated.

\begin{cor}\label{cor:fg_Aut_AC}
    For every Coxeter system $(W,S)$, the group $\Aut_{\rm AC}(W;S)$ admits an algorithmically computable finite set of generators.
\end{cor}
\begin{proof}
    Let $\mc{G}$ be the set of $W$--conjugacy classes of Coxeter generating sets of $W$ in the angle-compatibility class of $S$. We make $\mc{G}$ into a graph by joining by an edge any pair of classes represented by generating sets differing by an elementary twist. The group $\Aut_{\rm AC}(W;S)$ acts on $\mc{G}$ preserving its structure of a graph. Every vertex-stabiliser is finitely generated: it is an extension of the group of inner automorphisms of $W$ by a finite subgroup, namely the group of permutations of the corresponding generating set that extend to automorphisms of $W$.

    The quotient $\overline{\mc{G}}:=\mc{G}/\Aut_{\rm AC}(W;S)$ inherits the structure of a graph. 
    Since $W$ is Hopfian (e.g.\ because it is residually finite), the vertices of $\overline{\mc{G}}$ are naturally in bijection with the possible Coxeter matrices of Coxeter generating sets of $W$ angle-compatible with $S$, up to conjugation by permutation matrices. All of these Coxeter matrices have the same size 
    and the same finite set of possible entries, which shows that $\overline{\mc{G}}$ has finitely many vertices. Also note that the graph $\mc{G}$ is locally finite, because only finitely many different elementary twists can be applied on any given Coxeter generating set. As a consequence, the graph $\overline{\mc{G}}$ also has finitely many edges. 

    Now, \Cref{thmintro:main} can be rephrased as stating that the graph $\mc{G}$ is connected. Combined with the fact that the action $\Aut_{\rm AC}(W;S)\acts\mc{G}$ has finitely generated point-stabilisers and the quotient graph $\overline{\mc{G}}$ is finite, this shows that $\Aut_{\rm AC}(W;S)$ is finitely generated.

    The graph $\overline{\mc{G}}$ can be constructed algorithmically: first, we list all Coxeter matrices with the same size and the same possible entries as the Coxeter matrix of $(W,S)$, up to conjugation by permutation matrices; then we add edges corresponding to all possible elementary twists that can be performed on any of these Coxeter matrices; finally, we select the connected component of the resulting graph that contains the Coxeter matrix of $S$. Therefore, a finite generating set for $\Aut_{\rm AC}(W;S)$ is algorithmically computable: it is given by conjugations by the elements of $S\sq W$, by the permutations of $S$ that extend to automorphisms of $W$ and, finally, by a finite set of elements of $\Aut(W)$ obtained by lifting a finite set of generators for the fundamental group $\pi_1(\overline{\mc{G}},\overline S)$, where $\overline S\in\overline{\mc{G}}$ denotes the projection of $S\in\mc{G}$.
\end{proof}

Combined with \Cref{cor:from_Aut_to_Aut_AC}, \Cref{cor:fg_Aut_AC} proves \Cref{corintro:automorphisms} from the Introduction. Since the twists provided by \Cref{thmintro:main} are relative to the family of compatible sets, the argument used for \Cref{cor:fg_Aut_AC} also immediately yields finite generation of \emph{relative} automorphism groups:

\begin{cor}
    Let $(W,S)$ be a Coxeter system and let $\mscr{P}$ be any collection of $S$--parabolic subgroups of $W$. Let $\Aut(W;\mscr{P})\leq\Aut(W)$ be the group of automorphisms that coincide with an inner automorphism of $W$ on each element of $\mscr{P}$. Then $\Aut(W;\mscr{P})$ is finitely generated.
\end{cor}

\bibliography{./mybib}
\bibliographystyle{alpha}

\end{document}